\documentclass[12pt]{amsart}
\usepackage{a4,amscd,amssymb,amsfonts, epsf,graphics,caption2}
\usepackage[section]{algorithm}
\usepackage{algorithmic}
\pdfoutput=1
\usepackage[pdftex]{graphicx}

\newcommand{\doo}{\partial}
\newcommand{\R}{{\mathbb R}}

\newcommand{\C}{{\mathbb C}}
\newcommand{\Om}{\Omega}
\newcommand{\ra}{\rightarrow}
\newcommand{\T}{{\mathbf{t}}}

\newcommand{\DOm}{\partial \Omega}

\newcommand{\Nmu}{m}

\newcommand{\Rdos}{\mathbb{R}^2}
\newcommand{\s}{\sigma}

\newcommand{\dbar}{\overline{\partial}}

\newcommand{\CGO}{{\sc cgo}}

\newcommand{\norm}[1]{\left\|#1\right\|}

\theoremstyle{plain}
\newtheorem{theorem}{Theorem}[section]

\theoremstyle{definition}
\newtheorem{definition}{Definition}[section]
\theoremstyle{notation}
\newtheorem*{notation}{Notation}
\theoremstyle{remark}

\newtheorem*{acknowledgements}{Acknowledgements}

\numberwithin{equation}{section}

\title[Nonlinear Fourier analysis]{Nonlinear Fourier analysis\\ for discontinuous conductivities:\\ computational results}
\author{K. Astala, L. P\"aiv\"arinta, J. M. Reyes and S. Siltanen}

\begin{document}

\maketitle

\tableofcontents

\begin{abstract}
Two reconstruction methods of Electrical Impedance Tomography (EIT) are numerically compared for nonsmooth conductivities in the plane based on the use of complex geometrical optics (CGO) solutions to D-bar equations involving the global uniqueness proofs for Calder\'on problem exposed in [Nachman; Annals of Mathematics \textbf{143}, 1996] and [Astala and P\"aiv\"arinta; Annals of Mathematics \textbf{163}, 2006]: the Astala-P\"aiv\"arinta theory-based {\em low-pass transport matrix method} implemented in [Astala et al.; Inverse Problems and Imaging \textbf{5}, 2011] and the {\em shortcut method} which considers ingredients of both theories. The latter method is formally similar to the Nachman theory-based regularized EIT reconstruction algorithm studied in [Knudsen, Lassas, Mueller and Siltanen; Inverse Problems and Imaging \textbf{3}, 2009] and several references from there.

New numerical results are presented using parallel computation with size parameters larger than ever, leading mainly to two conclusions as follows. First, both methods can approximate piecewise constant conductivities better and better as the cutoff frequency increases, and there seems to be a Gibbs-like phenomenon producing ringing artifacts. Second, the transport matrix method loses accuracy away from a (freely chosen) pivot point located outside of the object to be studied, whereas the shortcut method produces reconstructions with more uniform quality.
\end{abstract}


\section{Introduction}

We study a widely applicable nonlinear Fourier transform in
dimension two. We perform numerical tests related to the nonlinear
Gibbs phenomenon with much larger cutoff frequencies than before.
Furthermore, we compare two computational inverse transformations,
called {\em low-pass transport matrix method} and {\em shortcut
method} in terms of accuracy.

The inverse conductivity problem of Calder\'on \cite{Calder'on1980} is the main source of applications of the nonlinear Fourier transform we consider. Let $\Om \subset\R^2$ be the unit disc and let $\sigma: \Om\to (0,\infty)$ be an essentially bounded measurable function satisfying $\sigma(x)\geq c>0$ for almost every $x\in \Om$. Let $u\in H^{1}(\Om)$ be the unique solution to the following elliptic Dirichlet problem:
\begin{eqnarray}
  \nabla\cdot\sigma\nabla u &=& 0 \text{ in }\Om, \label{Calderon:eq1} \\
  u\big|_{\doo\Om} &=& \phi\in H^{1/2}(\doo \Om).\label{Calderon:eq2}
\end{eqnarray}
The inverse conductivity problem consists on recovering $\sigma$  from the Dirichlet-to-Neumann (DN) map or voltage-to-current map defined by
$$
  \Lambda_{\sigma} : \phi \mapsto \sigma\frac{\doo u}{\doo \nu}\Big|_{\doo \Om}.
$$
Here $\nu$ is the unit outer normal to the boundary. Note that the map $\Lambda: \sigma \mapsto \Lambda_{\sigma}$ is nonlinear.

The inverse conductivity problem is related to many practical applications, including the medical imaging technique called
{\em electrical impedance tomography} (EIT). There one attaches electrodes to the skin of a patient, feeds electric currents into the body and
measures the resulting voltages at the electrodes. Repeating the measurement with several current patterns yields a current-to-voltage data matrix that can be
used to compute an approximation $\Lambda_{\sigma}^\delta$ to $\Lambda_{\sigma}$. Since different organs and tissues have different conductivities, recovering
$\sigma$ computationally from $\Lambda_{\sigma}^\delta$  amounts to creating an image of the inner structure of the patient.
See \cite{Mueller2012,Cheney1999} for more information on EIT and its applications.

Recovering $\sigma$ from $\Lambda_{\sigma}^\delta$ is a nonlinear and ill-posed inverse problem, whose computational solution requires regularization. Several categories of solution methods have been suggested and tested in the literature; in this work we focus on so-called D-bar methods based on complex geometrical optics (CGO) solutions. There are three main flavors of D-bar methods for EIT:
\begin{itemize}
\item Schr\"odinger equation approach for twice differentiable $\sigma$. Introduced by Nachman in 1996 \cite{Nachman1996}, implemented numerically in \cite{Siltanen2000,Mueller2003,Isaacson2004,Isaacson2006,Knudsen2009,Mueller2012}.
\item First-order system approach for once differentiable $\sigma$. Introduced by Brown and Uhlmann in 1997 \cite{Brown1997}, implemented numerically in \cite{Knudsen2003,Knudsen2004a,Hamilton2012}.
\item Beltrami equation approach assuming no smoothness ($\sigma\in L^\infty(\Omega)$). Introduced by Astala and P\"aiv\"arinta in 2006 \cite{Astala2006a}, implemented numerically in \cite{Astala2006,Astala2010,Astala2011,Mueller2012}. The assumption $\sigma\in L^\infty(\Omega)$ was the one originally used by Calder\'on in \cite{Calder'on1980}.
\end{itemize}

Using these approaches, a number of conditional stability results
have been studied for the Calder\'on problem. The most recent
results in the plane were obtained by Clop, Faraco and Ruiz in
\cite{Clop2010}, where stability in $L^2$-norm was proven for
conductivities on Lipschitz domains in the fractional Sobolev
spaces $W^{\alpha,p}$ with $\alpha>0$, $1<p<\infty$, and in
\cite{Faraco2013}, where the Lipschitz condition on the boundary
of the domain was removed. In dimension $d\geq 3$, conditional
stability in H\"older norm for just $C^{1+\varepsilon}$
conductivities on bounded Lipschitz domains was proved by Caro,
Garc\'ia and the third author in \cite{Caro2013} using the method
presented in \cite{Haberman2013}.

The three aforementioned D-bar methods for EIT in the two-dimensional case are based on the use of nonlinear Fourier transforms specially adapted to the inverse conductivity problem. Schematically, the idea looks like this:

\begin{picture}(320,155)
\put(17,15){\Large$\sigma(z)$}
\put(220,15){\Large$\sigma(z)$}
\put(110,125){\Large $\tau(k)$}
\put(20,73){\Large$\Lambda_\sigma$}
\thicklines
\put(35,33){\vector(1,1){84}}
\put(59,47){\rotatebox{45}{\footnotesize Transform}}
\put(35,85){\vector(2,1){70}}
\put(140,117){\vector(1,-1){88}}
\put(164,97){\rotatebox{-45}{\footnotesize Inverse transform}}
\put(27,33){\vector(0,1){34}}
\put(0,50){\footnotesize Data}
\thinlines
  \multiput(0,107)(4,0){97}{\line(1,0){2}}
  \multiput(0,39)(4,0){97}{\line(1,0){2}}
\put(220,112){\footnotesize Nonlinear frequency domain ($k$-plane)}
\put(278,27){\footnotesize Spatial domain ($z$-plane)}
\end{picture}

\noindent The main point above is that the nonlinear Fourier transform can be calculated from the infinite-precision data $\Lambda_\sigma$, typically via solving a second-kind Fredholm boundary integral equation for the traces of the CGO solutions on $\partial\Omega$.

In practice one is not given the infinite-precision data $\Lambda_\sigma$, but rather the noisy and finite-dimensional approximation $\Lambda_\sigma^\delta$. Typically all we know about $\Lambda_\sigma^\delta$ is that $\|\Lambda_\sigma-\Lambda_\sigma^\delta\|_Y<\delta$ for some (known) noise level $\delta>0$ measured in an appropriate norm $\|\,\cdot\,\|_Y$. Most CGO-based EIT methods need to be regularized by a truncation $|k|<R$ in the nonlinear frequency-domain as illustrated in  Figure \ref{fig:scheme}.

\begin{figure}

\begin{picture}(320,250)
\put(-30,2){\includegraphics[height=2.5cm]{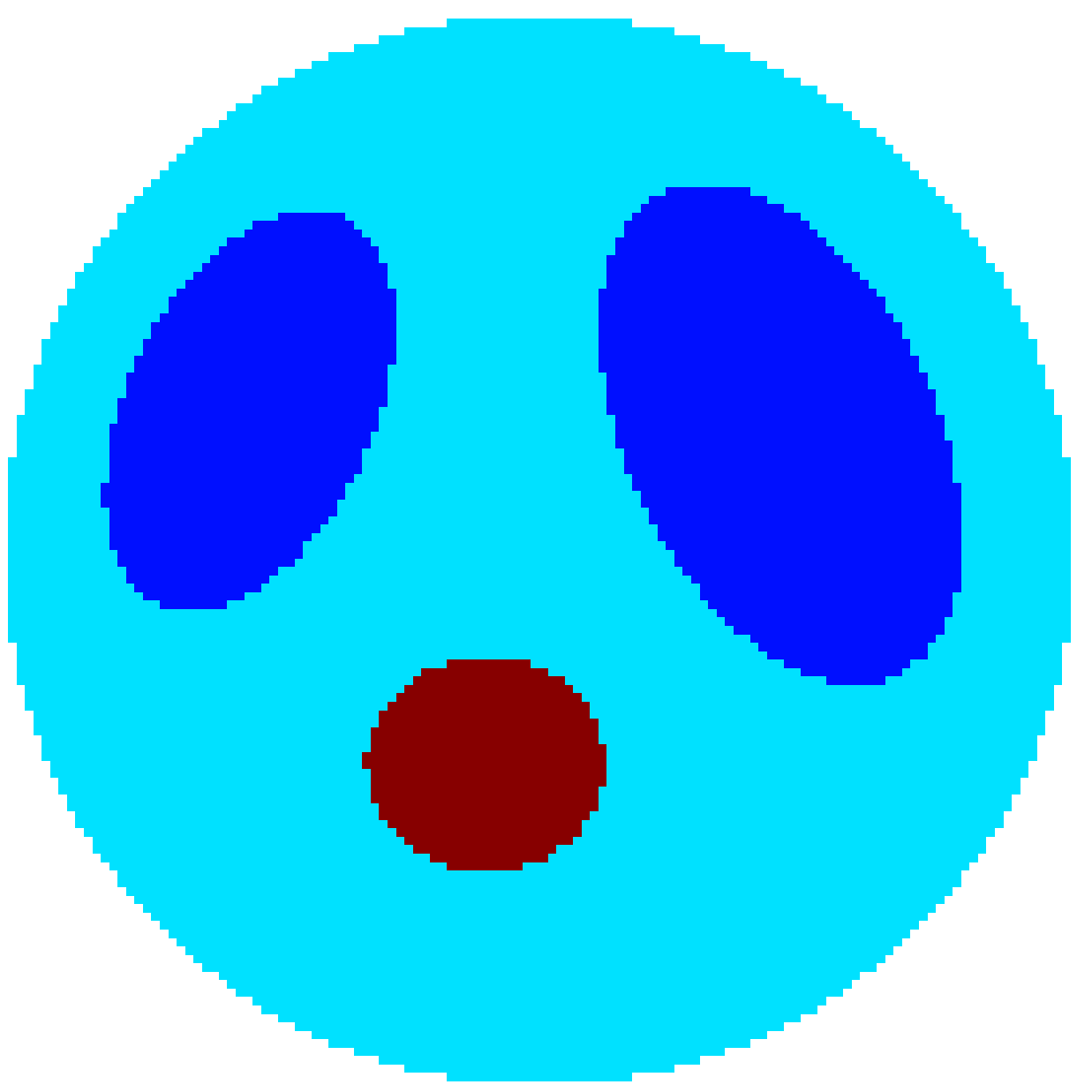}}
\put(273,2){\includegraphics[height=2.5cm]{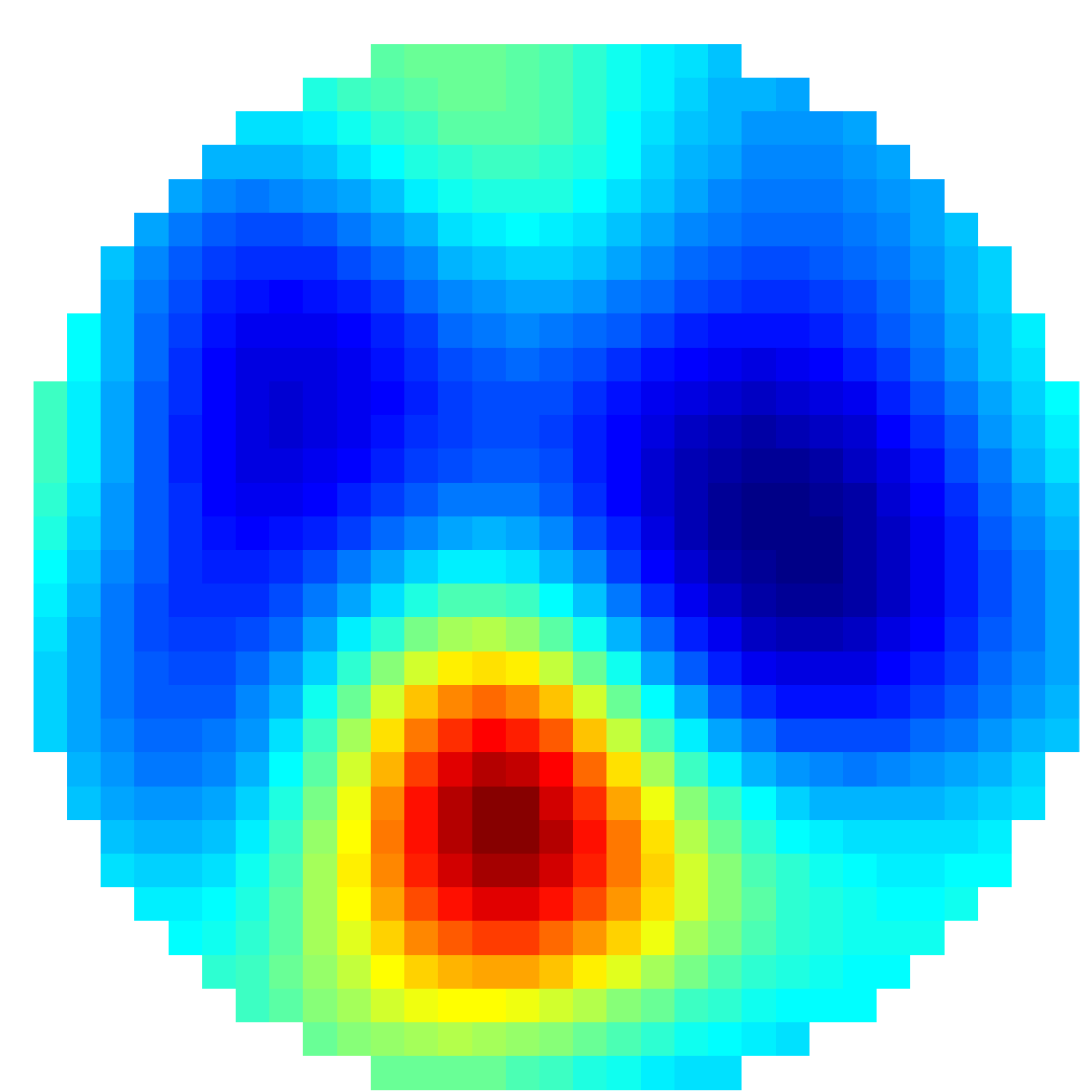}}
\put(140,2){\includegraphics[height=2.5cm]{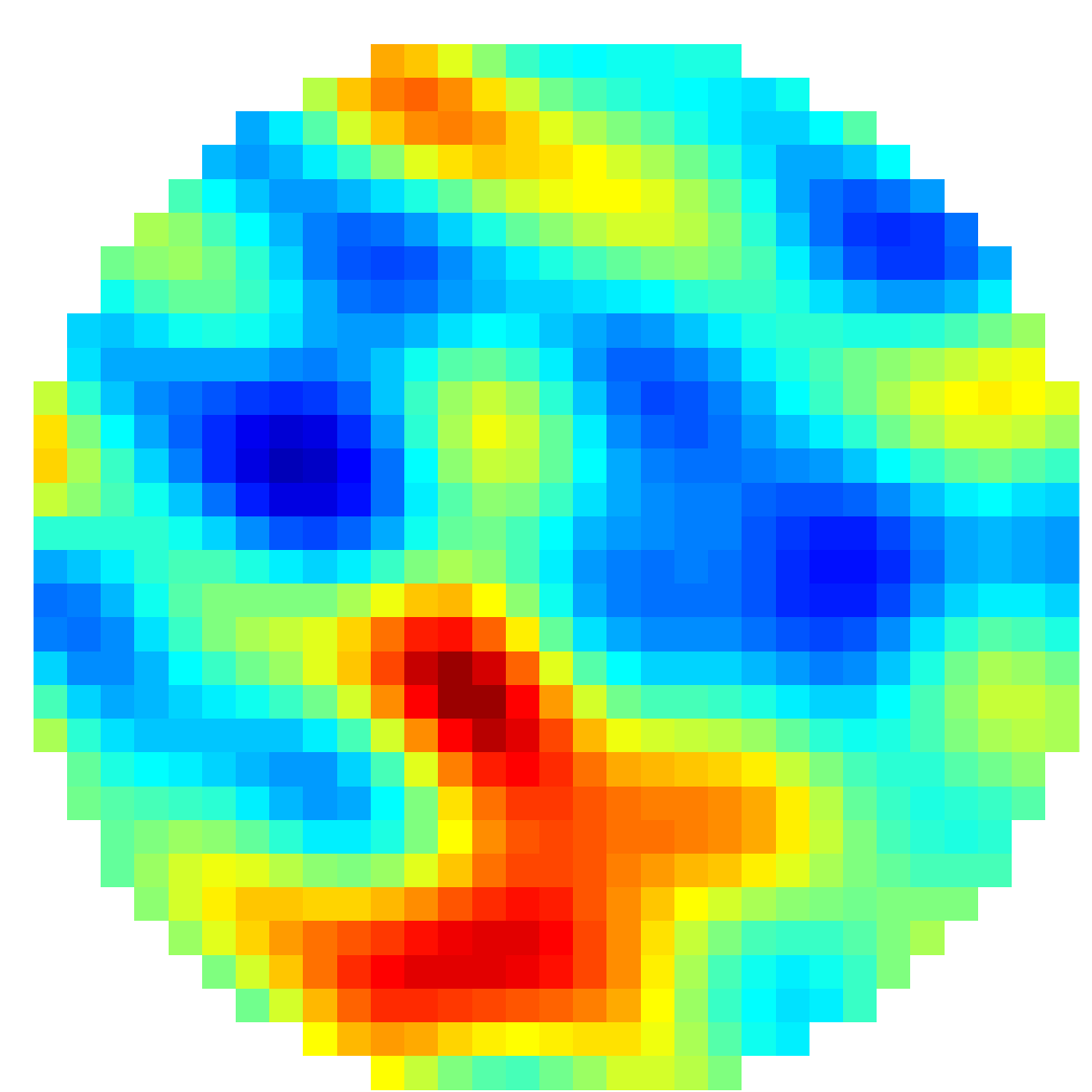}}
\put(132,155){\includegraphics[height=3cm]{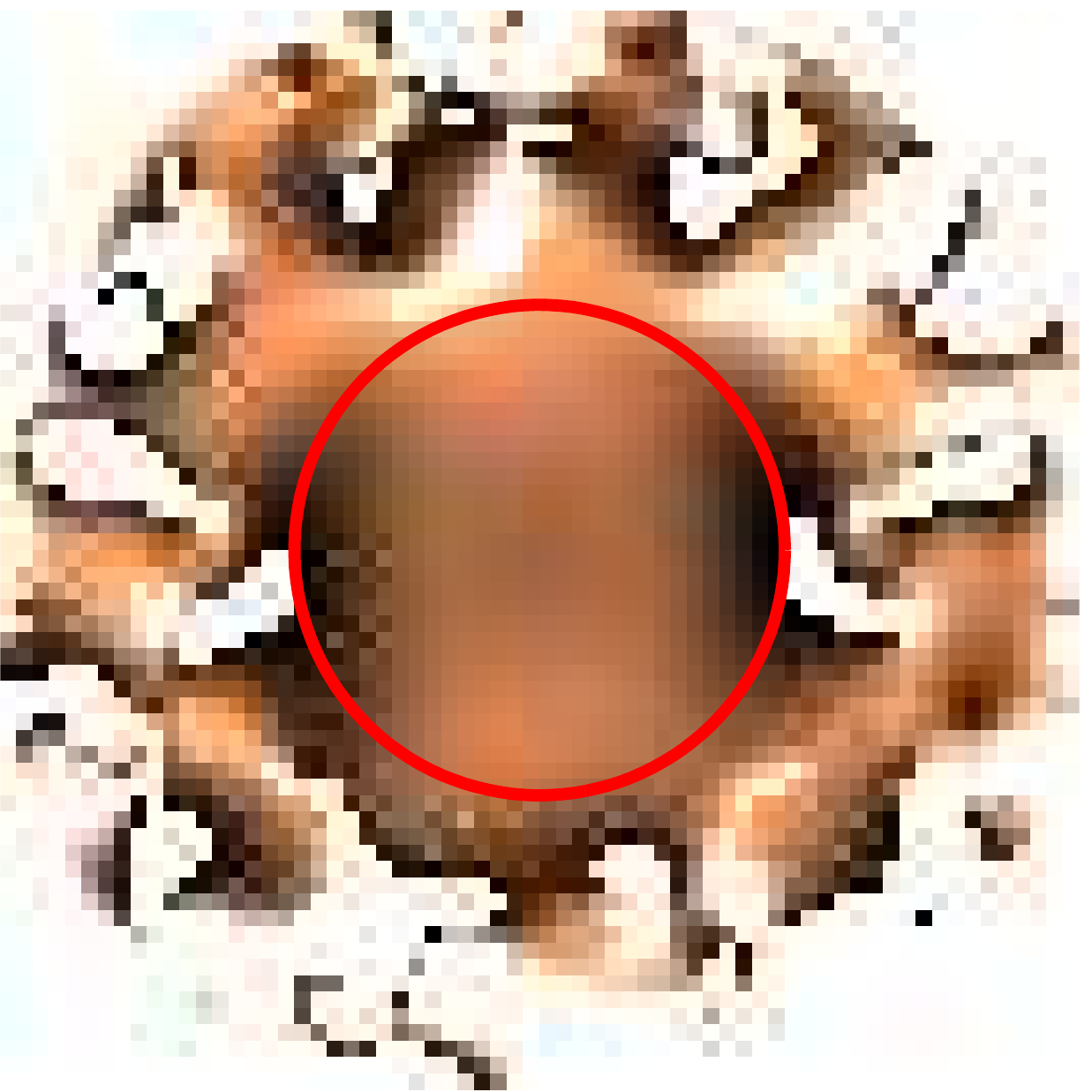}}
\put(265,155){\includegraphics[height=3cm]{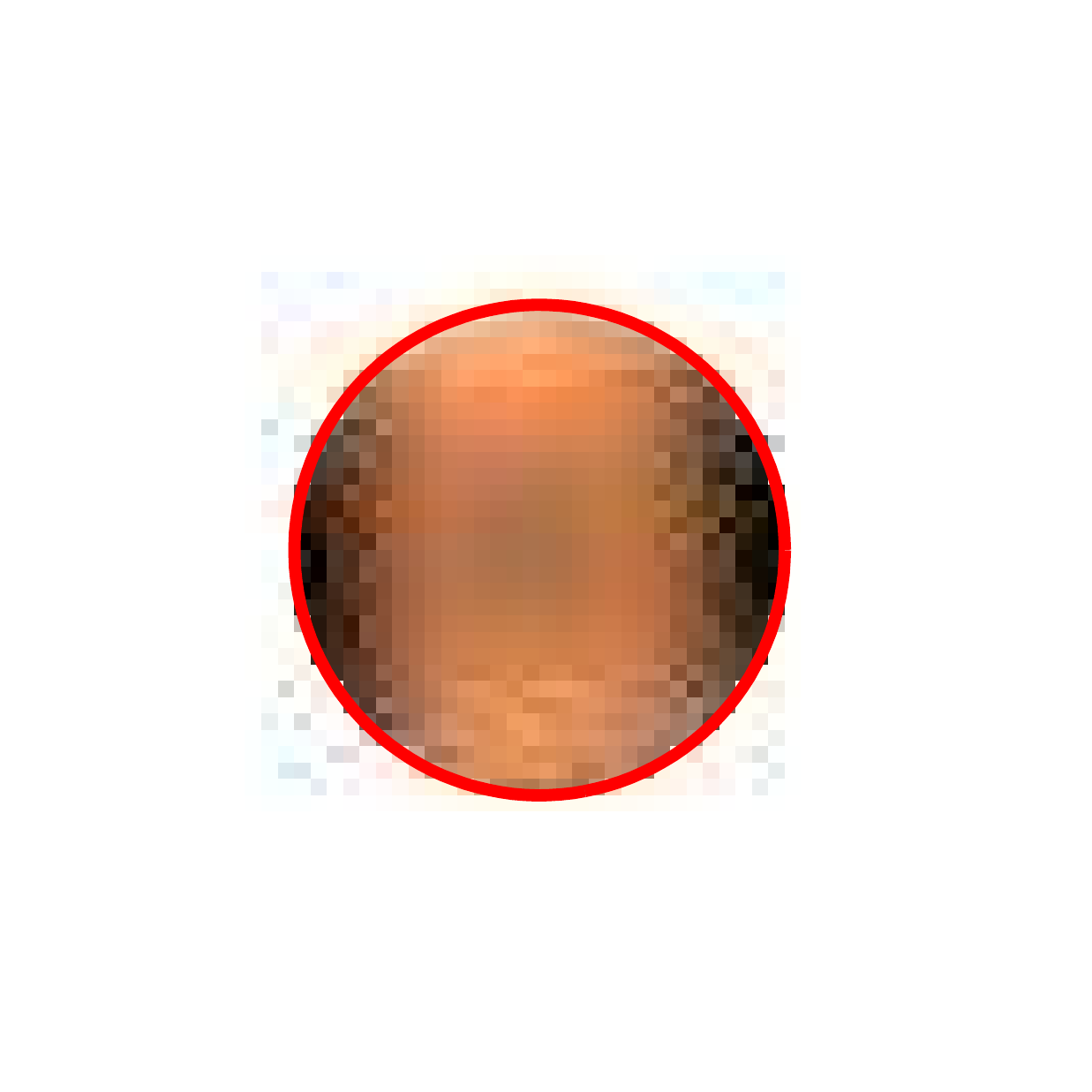}}
\put(-10,110){\includegraphics[height=1cm]{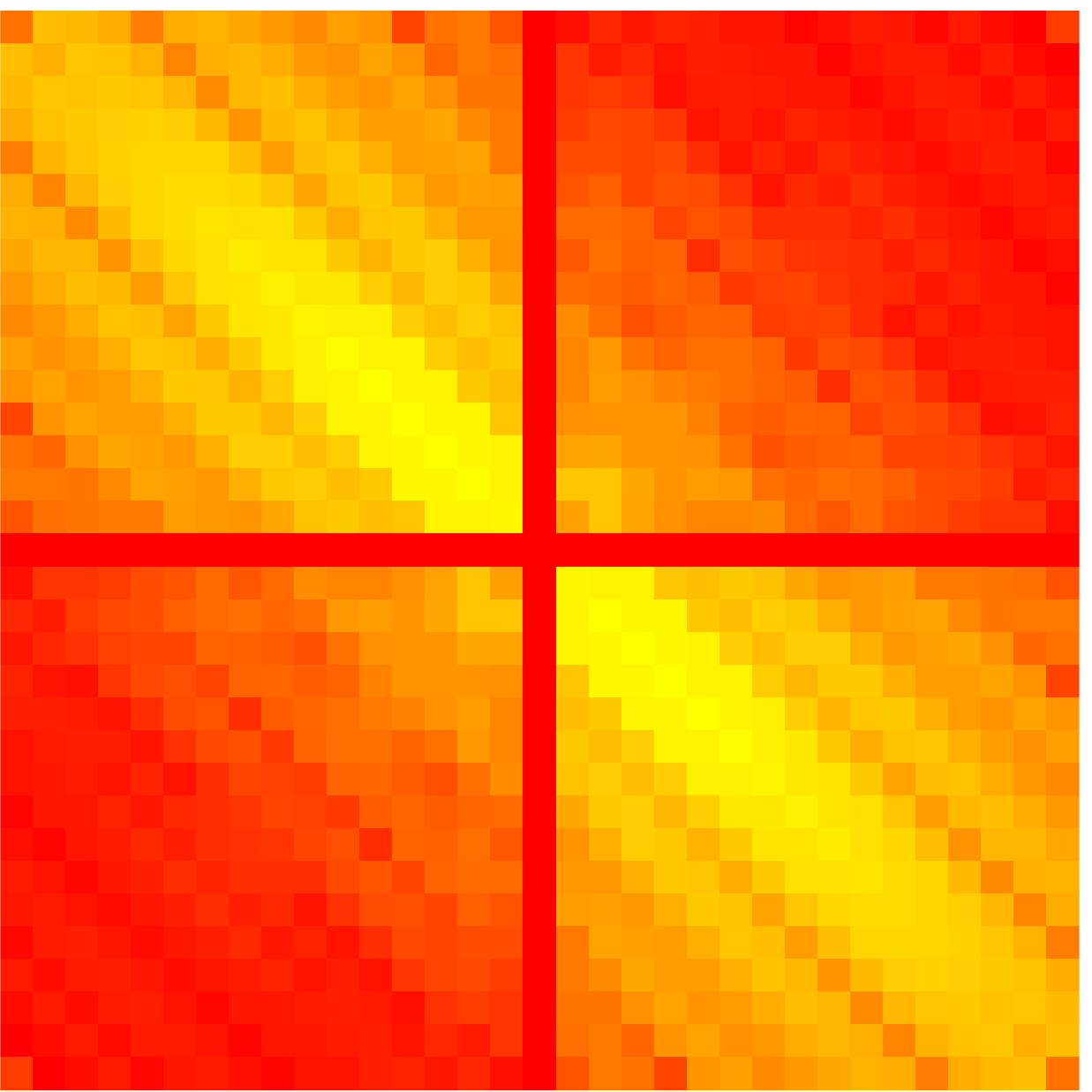}} \thicklines
\linethickness{.1mm} \put(308,168){\vector(0,-1){93}}
\put(312,140){\rotatebox{-90}{\footnotesize Inverse}}
\put(299,140){\rotatebox{-90}{\footnotesize transform}}
\put(175,155){\vector(0,-1){80}}
\put(179,140){\rotatebox{-90}{\footnotesize Inverse}}
\put(166,140){\rotatebox{-90}{\footnotesize transform}}
\put(5,75){\vector(0,1){28}} \put(25,140){\vector(3,1){100}}
\put(-27,120){$\Lambda_\sigma^\delta$}
\put(225,198){\vector(1,0){60}} \put(228,202){\footnotesize
Lowpass} \put(228,186){\footnotesize $|k|<4$} \thinlines
\multiput(-30,147)(4,0){97}{\line(1,0){2}}
\multiput(-30,85)(4,0){97}{\line(1,0){2}}
\put(228,73){\footnotesize $z$-plane} \put(228,152){\footnotesize
$k$-plane} \put(32,0){$\sigma(z)$} \put(207,0){(a)}
\put(340,0){(b)}
\end{picture}
\caption{\label{fig:scheme}Schematic illustration of the nonlinear low-pass filtering approach to regularized EIT. This image corresponds to the so-called {\em shortcut method} defined below. The simulated heart-and-lungs phantom $\sigma(z)$ (bottom left) gives rise to a finite voltage-to-current matrix $\Lambda_{\sigma}^\delta$ (orange square), which can be used to determine the nonlinear Fourier transform. Measurement noise causes numerical instabilities in the transform (irregular white patches), leading to a bad reconstruction (a). However, multiplying the transform by the characteristic function of the disc $|k|<4$ yields a lowpass-filtered transform, which in turn gives a noise-robust approximate reconstruction (b).}
\end{figure}
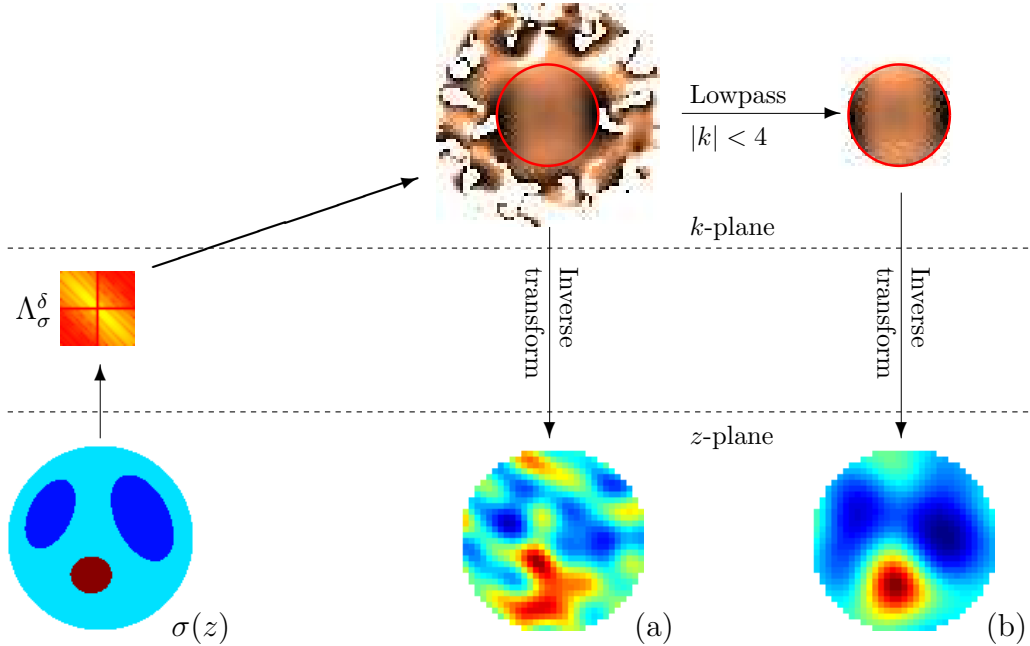

The regularization step results in a smooth reconstruction. This
smoothing property of the nonlinear low-pass filter resembles
qualitatively the effect of linear low-pass filtering of images.
In particular, the smaller $\mbox{R}$ is, the blurrier the
reconstruction becomes.

The cut-off frequency $\mbox{R}$ is determined by the noise amplitude $\delta$, and typically $\mbox{R}$ cannot exceed $7$ in practical situations. However, it is interesting to understand how the conductivity is represented by its nonlinear Fourier transform. To study this numerically, we compute nonlinear Fourier transforms in large discs such as $|k|<60$ and compute low-pass filtered reconstructions using inverse nonlinear Fourier transform and observe the results.

\begin{figure}[!htp]
\begin{picture}(250,130)
\put(-40,5){\includegraphics[width=12cm]{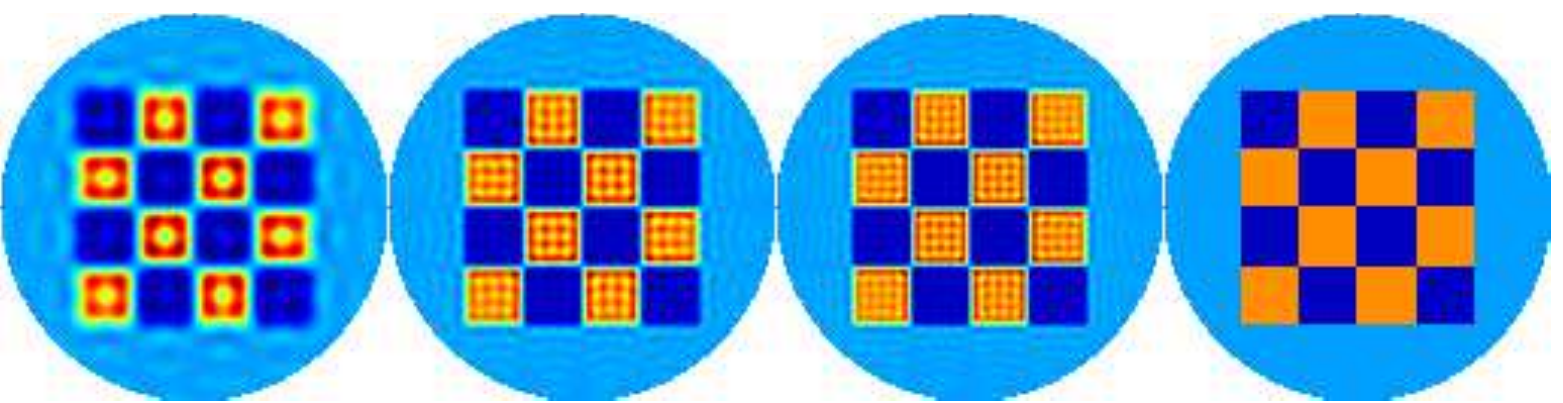}}
 \put(-20,100){\large $\text{R} = 20$}
 \put(65,100){\large $\text{R} = 40$}
 \put(150,100){\large $\text{R} = 50$}
 \put(240,100){\large True}
\end{picture}
\caption{\label{fig:nice_PWC_sigma2_INTRO}Convergence of
nonlinearly low-pass filtered conductivity to the true
discontinuous conductivity. Note the ringing artefacts reminiscent
of the Gibbs effect of linear low-pass filtering. This picture
illustrates the so-called shortcut method, and it is the
first-ever computation of this kind with cutoff frequencies larger
than 20.}
\end{figure}
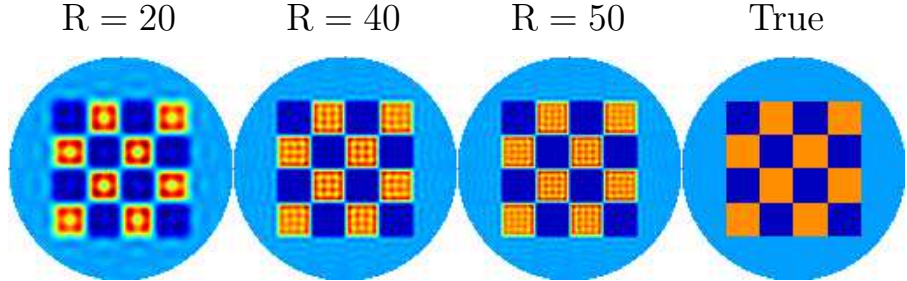

We compare two reconstruction methods based on the use of the CGO solutions (\ref{CGOsol}): the {\em low-pass transport matrix method} implemented in \cite{Astala2011}, and a {\em shortcut method} based on solving a D-bar equation. The latter method does not have rigorous analysis available yet, but it is formally similar to the regularized EIT reconstruction algorithm studied in \cite{Nachman1996,Siltanen2000,Mueller2003,Isaacson2004,Isaacson2006,Knudsen2007,Knudsen2009}.

Our new computational findings can be roughly summarized by the following two points. First, both methods can approximate piecewise constant conductivities better and better as the cutoff frequency $\mbox{R}$ increases, and there seems to be a Gibbs-like phenomenon producing ringing artifacts. Second, the transport matrix method loses accuracy away from a (freely chosen) pivot point located outside of $\overline{\Omega}$, whereas the shortcut method produces reconstructions with more uniform quality.

Figure \ref{fig:nice_PWC_sigma2_INTRO} shows some of our numerical
results via the {\it shortcut method} for a nonsymmetric
conductivity distribution.

The rotationally symmetric examples presented in Section
\ref{section:symm_examples} provide numerical evidence of the fact
that discontinuous conductivities can be reconstructed with the
shortcut method more and more accurately when $\mbox{R}$ tends to
infinity.

In the computational experiments we show in this paper, the truncated scattering transform is computed with a very large
cutoff frequency through the Beltrami equation solver. In practice, it is not possible to obtain such scattering data from
boundary measurements unless we had unrealistically high-precision EIT measurements available.

Notwithstanding the aforesaid, our new results are relevant for applications at least in these two ways:

\begin{enumerate}
\item We provide strong new evidence suggesting that the shortcut method can recover organ boundaries at the low-noise limit. Even with practical noise levels we can recover low-pass filtered jumps in the cases of nested inclusions or checkerboard-type discontinuity curves. While there are many dedicated methods for inclusion detection in EIT, for example \cite{Friedman1987,Bruhl2000,Ikehata2000a,Bruhl2001,Ikehata2004,Hyvonen2007,Ide2007,Lechleiter2008a,Hanke2008,Gebauer2008,Uhlmann2008,Harrach2009,Ide2010,Harrach2013}, none of them can locate nested inclusions.

\item The nonlinear Fourier transform studied here can be used to solve the nonlinear Noviko-Veselov (NV) equation. It is a nonlinear evolution equation generalizing the celebrated Korteweg-de Vries (KdV) equation into dimension (2+1). There has been significant recent progress in linearizing the NV equation using inverse scattering methods, see \cite{Lassas2012a,Lassas2012b,Perry2012,Perry2013,Croke2014}. In the potential applications of the NV equation there should be no restriction to as low frequencies as in EIT.
\end{enumerate}

We remark that all our numerical experiments have been computed using either a 192 Gigabyte computer with Linux or parallel computation provided by Techila.\footnote {Techila Technologies Ltd is a privately held Finnish provider of High Performance Computing middleware solutions.}

\section{The nonlinear Fourier transform}

In dimension two Calder\'on's problem \cite{Calder'on1980} was solved using complex geometrical optics (\CGO) solutions by Astala and P\"aiv\"arinta  \cite{Astala2006a}.
In the case of $L^\infty$-conductivities $\sigma$ the \CGO \  solutions need to be constructed via  the Beltrami equation
\begin{equation}\label{CGOsol}
  \dbar_z f_\mu = \mu\,\overline{\partial_z f_\mu}, \quad  \text{with } f_\mu(z,k) = e^{ikz}(1+\omega(z,k)),
\end{equation}
where
\begin{equation*}
  \mu := \frac{1-\sigma}{1+\sigma},\qquad \omega(z,k)={\mathcal O}\bigg(\frac{1}{z}\bigg) \mbox{ as }|z|\to \infty.
\end{equation*}
Here $z=z_1+iz_2\in\C$, $\dbar_z=(\partial/\partial z_1 + i\partial/\partial z_2)/2$ and $k$ is a complex parameter.

More specifically, for $\sigma\in L^{\infty}(\C)$ with $c^{-1}\leq \sigma(z)\leq c$ almost every $z\in\C$ and $\s\equiv 1$ outside the unit disc, and for $\kappa<1$ with $|\mu(z)|\leq\kappa\,$ a.e. $z\in\C$, Theorem 4.2 in \cite{Astala2006a} establishes that for any $2< p<1+1/\kappa$ and complex parameter $k$, there exists a unique solution $f_{\mu}(\cdot,k)$ in the Sobolev space $W^{1,p}_{\text{loc}}(\C)$ such that
\begin{equation}\label{form:beltrami}
\dbar_z f_{\mu} (z,k) = \mu(z)\, \overline{\partial_z f_{\mu}
(z,k) },\qquad\text{for }z\in\C,
\end{equation}
holds with the asymptotic formula
\begin{equation}\label{form:normcond}
f_{\mu} (z,k) = e^{ikz} \bigg(1+\mathcal{O}\bigg({1\over
z}\bigg)\bigg)\qquad\text{as }|z|\rightarrow\infty .
\end{equation}
In addition, $f_{\mu}(z,0)\equiv 1$.

Write $f_{-\mu}$ for the solution to the Beltrami equation with coefficient $-\mu$. The scattering transform $\tau$ of Astala
and P\"aiv\"arinta (AP) is given by the formula
\begin{equation}\label{scat_AP}
  \overline{\tau(k)} = \frac{1}{2\pi}\int_\C \dbar_z (\omega(z,k) - \omega^-(z,k)) \,dz_1dz_2,
\end{equation}
where $\omega(z,k) := e^{-ikz} f_{\mu}(z,k)-1$, $\omega^-(z,k) := e^{-ikz} f_{-\mu}(z,k)-1$. It is known that $|\tau(k)|\leq 1$ for all $k\in\C$.

Defining the functions as follows for $k,z\in\C$:
\begin{eqnarray}
  \label{form:hplus}h_+(z,k) &:=& 1/2\, (f_\mu(z,k) + f_{-\mu}(z,k)),\\
  \label{form:hminus}h_-(z,k) &:=& i/2 \, (\,\overline{f_\mu(z,k)} \, - \,\overline{f_{-\mu}(z,k)}\, ),
\end{eqnarray}
and
\begin{eqnarray*}
  u_1(z,k) &:=& h_+(z,k) -i \, h_-(z,k) = \mbox{Re} f_{\mu}(z,k) + i \, \mbox{Im} f_{-\mu}(z,k),\\
  u_2(z,k) &:=& i \, (h_+(z,k) + i\, h_-(z,k)) =  -\mbox{Im} f_{\mu}(z,k) + i \, \mbox{Re} f_{-\mu}(z,k),
\end{eqnarray*}
it turns out that $u_1$, $u_2$ are the unique solutions to the following elliptic problems:
\begin{eqnarray}
  \label{form:condeq1}\nabla\cdot\sigma\nabla u_1 &=& 0, \qquad u_1(z,k) =  e^{ikz}\big(1+\mathcal{O}\big(1/ z\big)\big),\\
  \label{form:condeq2}\nabla\cdot\sigma^{-1}\nabla u_2 &=& 0, \qquad u_2(z,k) = e^{ikz}\big(i + \mathcal{O}\big(1 / z\big)\big),
\end{eqnarray}
as $|z|\rightarrow\infty$.

The complex parameter $k$ above can be understood as a nonlinear frequency-domain variable. As well as the function $\tau(k)$, defined in \eqref{scat_AP}, plays the role of the nonlinear Fourier transform in the {\it shortcut method} described below, it is played by the following function
\[
\nu_{z_0}(k) := i\frac{h_-(z_0,k)}{h_+(z_0,k)},
\]
in the {\it low-pass transport matrix method} presented in Section \ref{section:APmethod}, where $z_0$ is a pivot point outside $\overline{\Omega}$.

A practical algorithm was introduced in \cite{Astala2006,Astala2006a,Astala2011} for recovering $\sigma$ approximately from a noisy measurement of $\Lambda_\sigma$. This algorithm involves a truncation operation (nonlinear low-pass filtering) in the frequency-domain.

The scattering transform used in \cite{Nachman1996} is defined  as
\begin{equation}\label{form:Nscat}
\T(k) = \int_{\R^2} e^{i(kz+\overline{k}\,\overline{z})}q(z) m(z)\,dz,
\end{equation}
where $q = \sigma^{-1/2}\Delta \sigma^{1/2}$ and $\sigma$ is a strictly positive-real valued function in $W^{2,p}(\R^2)$ for some $p>1$ with $\sigma\equiv 1$ on the unit disc near the boundary and outside the unit disc, and $m$ solves the equation
$$
m = 1-g_k\ast (qm).
$$
Here $\ast$ denotes convolution over the plane and $g_k$ the Faddeev fundamental solution of  $-\Delta-4ik\dbar$. For smooth enough $\sigma$ one can prove that
\begin{equation}\label{transform_comparison}
  \T(k) = - 4 \pi i \,{\overline k} \,\tau(k).
\end{equation}

In the rest of the paper, for non-smooth conductivities $\T$ denotes the function defined through formula \eqref{transform_comparison}. Why are $\T$ and $\tau$ called nonlinear Fourier transforms? There is admittedly some abuse of terminology involved, but the main reason is this. If one replaces $m$ in \eqref{form:Nscat} with $1$, the result is the linear Fourier transform of $q$ at $(-2k_1,2k_2)\in\R^2$. But $m$ is only asymptotically close to $1$ and actually depends on $q$.

The following theorem holds:

\begin{theorem}\label{thrm:rot_sym}
Assume $\sigma\in L^{\infty}(\R^2)$ is real-valued, radially
symmetric, bounded away from zero and $\sigma\equiv 1$ outside the
unit disc. Then the function $\T$ is real-valued and rotationally
symmetric.
\end{theorem}
\begin{proof}
The rotational symmetry of $\T$ follows from the property $\T(k) =
\T (\eta k)$ for any $\eta, k\in\C$ with $|\eta|=1$, which is
tantamount to $\tau(k) = \overline{\eta}\, \tau(\eta k)$. By
uniqueness of the asymptotic boundary value problem
\eqref{form:beltrami}-\eqref{form:normcond}, one can check that
$f_{\mu}(z,\eta k) = f_{\mu}(\eta z, k)$. Therefore,
$\omega(z,\eta k) = \omega(\eta z, k)$. For $-\mu$ the same
equalities follow. Using definition \eqref{scat_AP} we deduce
$\tau(k) = \overline{\eta}\, \tau(\eta k)$.

Thanks to $\T(k) = \T (\eta k)$, real-valuedness of $\T$ follows
from $\T(k) = \overline{\T ( \overline{k})}$ and this fact is
equivalent to seing $\tau(k) = \overline{\tau(-\overline{k})}$ due
to the identity $\tau(\,\overline{k}\,) =
\overline{\eta}\,\tau(\eta\,\overline{k}\,)$ with $\eta = -1$.
Again, by uniqueness of
\eqref{form:beltrami}-\eqref{form:normcond}, it is straightforward
to prove
\[
f_{\mu}(z,k) = \overline{f_{\mu}(\overline{z},-\overline{k})},
\]
and $\omega(z,k) = \overline{\omega(\overline{z} ,
-\overline{k})}$. Attaining the same identity for $-\mu$ and by
\eqref{scat_AP} we deduce $\tau(k) =
\overline{\tau(-\overline{k})}$.
\end{proof}

\section{Inverse transforms and reconstruction algorithms}

\subsection{The low-pass transport matrix method}\label{section:APmethod}

Given noisy EIT data $\Lambda_\sigma^\delta$, we can evaluate the
traces of $M_{\mu}(\,\cdot\,,k)=1+\omega(\,\cdot\,,k)$ at the
boundary $\DOm$ for $|k|<R$ with some $\mbox{R}>0$ depending on
the noise level $\delta$. This is achieved by the numerical method
described in Section 4.1 of \cite{Astala2011} based on solving the
boundary integral equations studied in \cite{Astala2006}.

Once $M_{\mu}$ is known on the boundary, so is $f_{\mu}(\,\cdot\,,k)|_{\DOm}$, which can be used to expand $f_\mu$ as a power series outside $\Om$. Choose a point $z_0\in\R^2\setminus\overline{\Om}$.
Set
\begin{equation}\label{lowpass}
  \nu^{(R)}_{z_0}(k) :=
\left\{\begin{array}{cl}
 \displaystyle{i\frac{h_-(z_0,k)}{h_+(z_0,k)}} & \mbox{ for }|k|<R,\\\\
 0 & \mbox{ for }|k|\geq R.
\end{array}\right.
\end{equation}

We next solve the truncated Beltrami equations
\begin{eqnarray}
  \label{RatruncBeltrami}
  \dbar_k\alpha^{(R)} &=& \nu^{(R)}_{z_0}(k)\,\overline{\partial_k \alpha^{(R)}},\\
  \label{RbtruncBeltrami}
  \dbar_k\beta^{(R)} &=& \nu^{(R)}_{z_0}(k)\,\overline{\partial_k \beta^{(R)}},
\end{eqnarray}
with solutions represented in the form
\begin{eqnarray}
  \label{Ratrunc_asymp}
  \alpha^{(R)}(z,z_0,k) &=& \exp(ik(z-z_0)+\varepsilon(k)),\\
  \label{Rbtrunc_asymp}
  \beta^{(R)}(z,z_0,k) &=& i\exp(ik(z-z_0)+ \widetilde{\varepsilon}(k)),
\end{eqnarray}
where $\varepsilon(k)/k\ra 0$ and  $\widetilde{\varepsilon}(k)/k \ra 0$ as $k\ra \infty$.
Requiring
\begin{equation}\label{normalization}
\alpha^{(R)}(z,z_0,0) = 1 \quad \mbox{and} \quad \beta^{(R)}(z,z_0,0) = i
\end{equation}
fixes the solutions uniquely.

To compute $\alpha^{(R)}$ we first find solutions $\eta_1$ and $\eta_2$  to the equation
\begin{equation}
\label{uus1}
\dbar_k \eta_j =  \nu^{(R)}_{z_0}(k) \,  \overline{\partial_k \eta_j},
\end{equation}
with asymptotics
\begin{eqnarray}
  \eta_1(k)
  &=& \label{uus2}
  e^{ik(z-z_0)}\bigl(1+{\mathcal O}(1/k)\bigr), \\
  \eta_2(k)
  &=&  \label{uus3}
  i \,e^{ik(z-z_0)}\bigl(1+{\mathcal O}(1/k)\bigr),
\end{eqnarray}
respectively, as $ |k|\to \infty$. There are constants $A, B \in \R$ such that
$
  A\, \eta_1(0) + B\,  \eta_2(0) = 1.
$
We now set
$
  \alpha^{(R)}(z,z_0,k) =  A\, \eta_1(k) + B\,  \eta_2(k).
$
Then $\alpha^{(R)}(z,z_0,k)$ satisfies the equation \eqref{RatruncBeltrami} and the appropriate asymptotic conditions.

Let $\alpha^{(R)}_{-\mu}$ denote the solution to \eqref{RatruncBeltrami} with the condition \eqref{Ratrunc_asymp} so that the coefficient $\nu^{(R)}_{z_0}$ is computed interchanging the roles of $\mu$ and $-\mu$. Then $\beta^{(R)}_{\mu} = i \, \alpha^{(R)}_{-\mu}$.
For computational construction of the frequency-domain {\sc cgo} solutions $\eta_j$ satisfying (\ref{uus1}--\ref{uus3}), see \cite[Section 5.2]{Astala2011} and \cite{Huhtanen2012}.

Fix any nonzero $k_0\in\C$  and choose any point $z$ inside the unit disc. Denote  $\alpha^{(R)} = a^{(R)}_1+ia^{(R)}_2$ and $\beta^{(R)} =  b^{(R)}_1+ib^{(R)}_2$. We can now use the truncated transport matrix
\begin{equation}\label{trunctrans}
T^{(R)} = T^{(R)}_{z,z_0,k_0} := \left(\begin{array}{ccc}  a^{(R)}_1 &  a^{(R)}_2  \\  b^{(R)}_1 &  b^{(R)}_2 \\ \end{array}\right)
\end{equation}
to compute
\begin{eqnarray}\label{transtrunc}
u^{(R)}_1(z,k_0) &=&  a^{(R)}_1 u_1(z_0,k_0) +a^{(R)}_2 u_2(z_0,k_0),\\
u^{(R)}_2(z,k_0) &=& b^{(R)}_1 u_1(z_0,k_0) +b^{(R)}_2 u_2(z_0,k_0). \nonumber
\end{eqnarray}

We know the approximate solutions $u^{(R)}_j(z,k_0)$ for $z\in\Om$ and one fixed $k_0$. We use formulas $u_1^{(R)} = h_+^{(R)} - ih_-^{(R)}, \,\, u_2^{(R)} = i(h_+^{(R)} + ih_-^{(R)})$ and \eqref{form:hplus}-\eqref{form:hminus} (with $z\in\Om$) to connect $u^{(R)}_1,u^{(R)}_2$ with $f^{(R)}_\mu,f^{(R)}_{-\mu}$. Define
\begin{equation}\label{finalstep1}
  \mu^{(R)}(z) = \frac{\dbar f^{(R)}_\mu(z,k_0)}{\overline{\partial f^{(R)}_\mu}(z,k_0)}.
\end{equation}

Finally we reconstruct the conductivity $\sigma$ approximately as
\begin{equation}\label{finalstep2}
  \sigma^{(R)} = \frac{1-\mu^{(R)}}{1+\mu^{(R)}}.
\end{equation}

\subsection{The shortcut method}\label{sec:shortcutm}

Fix $\mbox{R}>0$ corresponding to a certain noise level $\delta$ in the EIT data. For all $|k|<R$, we compute the traces of $M_{\mu}(\cdot , k)$ and $M_{-\mu}(\cdot , k)$ by solving the boundary integral equation aforementioned in Section \ref{section:APmethod}. The analyticity of $M_{\mu}(\cdot , k)$, $M_{-\mu}(\cdot , k)$ outside the unit disc allows to develop these factors as follows
\[
M_{\mu}(z , k) = 1 + {a_1^+ (k)\over z} + {a_2^+ (k)\over z^2} + \ldots, \quad M_{-\mu}(z , k) = 1 + {a_1^-(k)\over z} + {a_2^-(k)\over z^2} + \ldots,
\]
for $|z|>1$.

In Section 5 of \cite{Astala2006a} it is proved that the scattering transform $\tau$ satisfies
\[
\tau(k) = {1\over 2} \,(\,\overline{a_1^+(k)}\, -\, \overline{a_1^-(k)}\,)\, , \qquad k\in \C.
\]
Notice that this formula is consistent with \eqref{scat_AP}. Let $\tau^{\delta}_R(k)$ be a numerical approximation to $\tau(k)$ such that $\tau^{\delta}_R(k) = 0$ for $|k|> R$.


For all $z\in\Omega$, we solve the D-bar equation
\begin{equation}\label{form: Dbar}
  \dbar_k \Nmu^\delta_R(z,k) = -i\,\tau^\delta_R(k) e_{-z}(k)\,\overline{\Nmu^\delta_R(z,k)}
\end{equation}
with $\Nmu^\delta_R(z,\,\cdot\,)-1\in L^r(\R^2)\cap L^\infty(\R^2)$, for some
$r>2$.


Finally, we reconstruct $\sigma$ approximately as $\sigma(z)\approx  (\Nmu^\delta_R(z,0))^2$. This method is proven to work for $C^2$ conductivities $\sigma$ by computing the approximation $\tau^\delta_R (k)$ in a different way, namely the methods described in \cite{Nachman1996,Siltanen2000,Mueller2003,Isaacson2004,Isaacson2006,Knudsen2007,Knudsen2009}. Nevertheless, there is no theoretical understanding of what happens when $\sigma$ is nonsmooth and only essentially bounded.

\subsection{Periodization of the nonlinear inverse transform}\label{sec:periodic}

The algorithm we use to solve equation \eqref{form: Dbar} is
presented in \cite{Knudsen2004} and is a modification of the
method introduced by Vainikko in \cite{Vainikko2000} for the
Lippmann-Schwinger equation related to the Helmholtz equation.
Such algorithm is based in reducing the equation
\begin{equation}\label{form:16} m_R(z,k) = 1 + {1\over\pi}\,\int_{B_R}
{F(z,k')\, \overline{m_R(z,k')}\over k-k'}\, dk',
\end{equation}
to a periodic one. Let us see how this reduction is done.

Here $F(z,k) := -i\,\tau(k) e_{-k}(z)$ and $e_k(z) :=
e^{ikz+i\overline{k}\overline{z}}$.

Take $\epsilon >0$ and set $s = 2R + 3\epsilon$. Define $Q = [-s,s)^2$. Choose an infinitely smooth cutoff function $\eta \in C^{\infty}_0 (\Rdos)$ verifying
$$
\eta (k) = \left\{
\begin{array}{ll}
1 & \text{for }|k| < 2R +\epsilon ,\\
0 & \text{for }|k|\geq 2R + 2\epsilon,\\
\end{array}
\right.
$$
and $0\leq\eta(k)\leq 1$ for all $k\in\C$.

\begin{definition}
We say a function $f$ defined on the plane is $2s$-periodic if $f(k) = f(k + 2s(j_1 + i j_2))$ for any $k\in\C$, $j_1$, $j_2\in\mathbb{Z}$. That is to say, we mean that $f$ is $2s$-periodic in each coordinate.
\end{definition}

\begin{notation}
$L^p(Q)$ denotes the space of $2s$-periodic functions in $L^p_{\text{loc}}(\Rdos)$ for $1\leq p\leq \infty$. Note that $L^{\infty}(Q) \subset L^{\infty}(\Rdos)$.

For any measurable subset $A$ of $\Rdos$ we denote by $\chi_A$ both the characteristic function of $A$ and the multiplier operator with such function in $k$.

We write $\mathcal{E}_0$ to denote the extension by zero outside the torus $Q$ defined as
$$
\mathcal{E}_0 f (z,k) = \left\{
\begin{array}{ll}
f(z,k), & \text{if }k\in Q,\\
0, &\text{if }k\in\C\setminus Q,
\end{array}
\right.
$$
for any $f(z,\cdot)\in L^1(Q)$.

\end{notation}

Define a $2s$-periodic approximation to the principal value $1/(\pi k)$ by setting $\widetilde{\beta} (k) = \eta(k)/(\pi k)$ for $k\in Q$ and extending periodically:
$$
\widetilde{\beta} (k + 2j_1 s + i 2 j_2 s ) = {\eta(k)\over \pi k }\qquad\text{for }k\in Q,\, j_1,j_2\in \mathbb{Z}.
$$

Define a periodic approximate solid Cauchy transform in $k$ by
$$
\widetilde{\mathcal{C}}_0 f (z,k) = (\widetilde{\beta} \,\widetilde{\ast} f) (z,k) = \int_Q \widetilde{\beta}(k-k') f(z,k')\, dk',
$$
where $\widetilde{\ast}$ denotes convolution on the torus.

We shall use another periodization of the Cauchy transform given by
$$
\widetilde{\mathcal{C}} f(z,k) = \mathcal{E} \chi_{\mathbb{D}_{R+\epsilon}} \widetilde{\mathcal{C}}_0 f(z,k) ,
$$
where $\mathcal{E}$ denotes the periodic extension operator defined as
$$
\mathcal{E} f (z,k + 2s (j_1 + i j_2)) = f(z,k), \qquad\text{for any }k\in Q,\, j_1,\, j_2\in \mathbb{Z}.
$$

Define the multiplication operator $F_R$ by $F_R f (z,k) = \chi_{B_R}(k) F(z,k) f(z,k)$ if $k\in\C$ and its periodization as
$$
\widetilde{F}_R f (z,k + 2j_1 s + i\, 2 j_2 s) = F_R f(z,k),\qquad\text{for }j_1,j_2\in\mathbb{Z}, k\in Q.
$$

Let us denote the complex conjugation operator by $\textbf{c}$, i.e. $\textbf{c} (f) = \overline{f}$. We periodize the complex conjugation operator $\widetilde{\textbf{c}}$ in the same way.

We shall need some boundedness properties of the so-called Cauchy transform. The Cauchy transform, we denote it by $\dbar^{\,-1}$, is defined as the following weakly singular integral operator:
\begin{equation*}
\dbar^{\,-1} f (k) = {1\over \pi}\, \int_{\Rdos} {f(k')\over k-k'}\, dk_1' dk_2' \, ,
\end{equation*}
where $k' = k_1' + i k_2'$ and  $dk_1' dk_2'$ denotes the Lebesgue measure on $\Rdos$.

Theorem 4.3.8 in \cite{Astala2009} states that
\begin{align}\label{form:cauchyqq*}
\norm{\dbar^{\,-1}\phi}_{L^{q^{\ast}}(\C)}\leq {C\over (q-1)(2-q)}\, \norm{\phi}_{L^q(\C)}
\end{align}
for some constant $C$, $1<q<2$ and $q^{\ast} = 2q/(2-q)$.

For any $1<q<\infty$ and $1/q + 1/q' = 1$ the space $L^q(\C)\cap L^{q'}(\C)$ with the norm $\norm{\cdot}_{L^q} + \norm{\cdot}_{L^{q'}}$ is a Banach space. Theorem 4.3.11 in \cite{Astala2009} claims the following:

\begin{theorem}\label{theo:boundcauchy}
When $1<q<2$ the Cauchy transform maps continuously the space $L^q(\C)\cap L^{q'}(\C)$ into the space of continuous functions on the plane vanishing at infinity with the sup norm and satisfies the estimate
\begin{equation*}
\norm{\dbar^{\,-1}\phi}_{L^{\infty}(\C)}\leq (2-q)^{-1/2} \big( \norm{\phi}_{L^q(\C)} + \norm{\phi}_{L^{q'}(\C)} \big) .
\end{equation*}
\end{theorem}

It is easy to see that if $1<q_1<2<q_2<\infty$ but not necessarily $q_1^{-1}+q_2^{-1}=1$ then
\begin{equation}\label{form:boundcauchy}
\norm{\dbar^{\,-1}\phi}_{L^{\infty}(\C)}\lesssim \big( \norm{\phi}_{L^{q_1}(\C)} + \norm{\phi}_{L^{q_2}(\C)} \big),
\end{equation}
and $\dbar^{\,-1}\phi$ is a continuous function assuming the right hand side bounded, where the implicit constant just depends on $q_1$, $q_2$.

We conclude this section with the Theorem \ref{lema:per}. Its
proof uses the aforementioned properties of the Cauchy transform
and is quite similar to the arguments presented in Sections 14.3.3
and 15.4.1 of the book \cite{Mueller2012}. Some of the techniques
for Theorem 4.1 in \cite{Nachman1996} are involved.

\begin{theorem}\label{lema:per} Assume $c^{-1}\leq \sigma\leq c$ a.e. in the plane and $\sigma\equiv 1$ outside the unit disc. Let $z\in\C$. There exists a unique $2s$-periodic (in $k$) solution $\widetilde{m}_R(z,k)$ to
\begin{equation}\label{form:per}
\widetilde{m}_R(z,k) = 1 + \widetilde{\mathcal{C}} \widetilde{F}_R \widetilde{\textbf{c}}\, \widetilde{m}_R(z,k),\qquad \text{a.e.}\,\,k\in Q,
\end{equation}
with $\widetilde{m}_R(z,\cdot)\in L^{r}(Q)$ for some $2<r<\infty$.

Furthermore, the solutions of \eqref{form:16} and \eqref{form:per} agree on the disc of the $k$-plane with radius $R$: $m_R(z,k) = \widetilde{m}_R(z,k)$ for $|k|<R$.

\end{theorem}

\section{Computational results}

Once scattering data with big cutoff frequencies are generated, we test the nonlinear inversion procedure consisting of computing the approximate conductivity from the truncated scattering transform by solving
\begin{equation}\label{form:17aa}
\dbar_k \, m_R(z,k) = \chi_{B_R}(k) (-i) \tau (k) e_{-k}(z) \,\overline{m_R(z,k)},
\end{equation}
with
\begin{equation}\label{form:18aa}
m_R(z,\cdot) -1\in L^{\infty}(\Rdos)\cap L^{r_0}(\Rdos),\qquad \text{for some }r_0\in (2,\infty).
\end{equation}

The unique solution $m_R(z,k)$ to \eqref{form:17aa} with the
normalization condition \eqref{form:18aa} can be obtained by
solving \eqref{form:16}. Finally, the approximate reconstruction
$\widetilde{\sigma}_R(z)$ is obtained by doing
$\widetilde{\sigma}_R(z):= m_R(z,0)^2$ for $|z|<1$.

\subsection{Nonlinear Gibbs phenomenon in radial cases}\label{section:symm_examples}

\noindent
In this section we test the shortcut method for two radially symmetric conductivities as follows: Firstly, we solve numerically the Beltrami equation and obtain reliable scattering data over a mesh for a real interval of the form $[0,\mbox{R}]$ with $\mbox{R}>0$. Secondly, from such scattering data the approximate reconstruction is computed by solving numerically the equation \eqref{form:per} on a mesh for the real interval $[0,1]$. It suffices to confine such computations to real positive numbers since the results on the whole disc can be obtained by simple rotation with respect to the origin by symmetry. The pictures for such reconstructions corresponding to the largest cutoff frequencies exhibit certain artifacts near the jump discontinuities of the actual conductivity quite similar to the Gibbs phenomenon for the truncated linear Fourier transform.

\subsubsection{Example $\sigma_1$}

We consider the rotationally symmetric conductivity $\sigma_1$ on the plane defined as $\sigma_1(z) = 2$ if $|z|<0.5$ and $\sigma_1(z) = 1$ if $|z|\geq 0.5$. Figure \ref{fig:sigma1} shows its representation.

\begin{figure}[!hbp]
\includegraphics[keepaspectratio=true, height = 5cm]{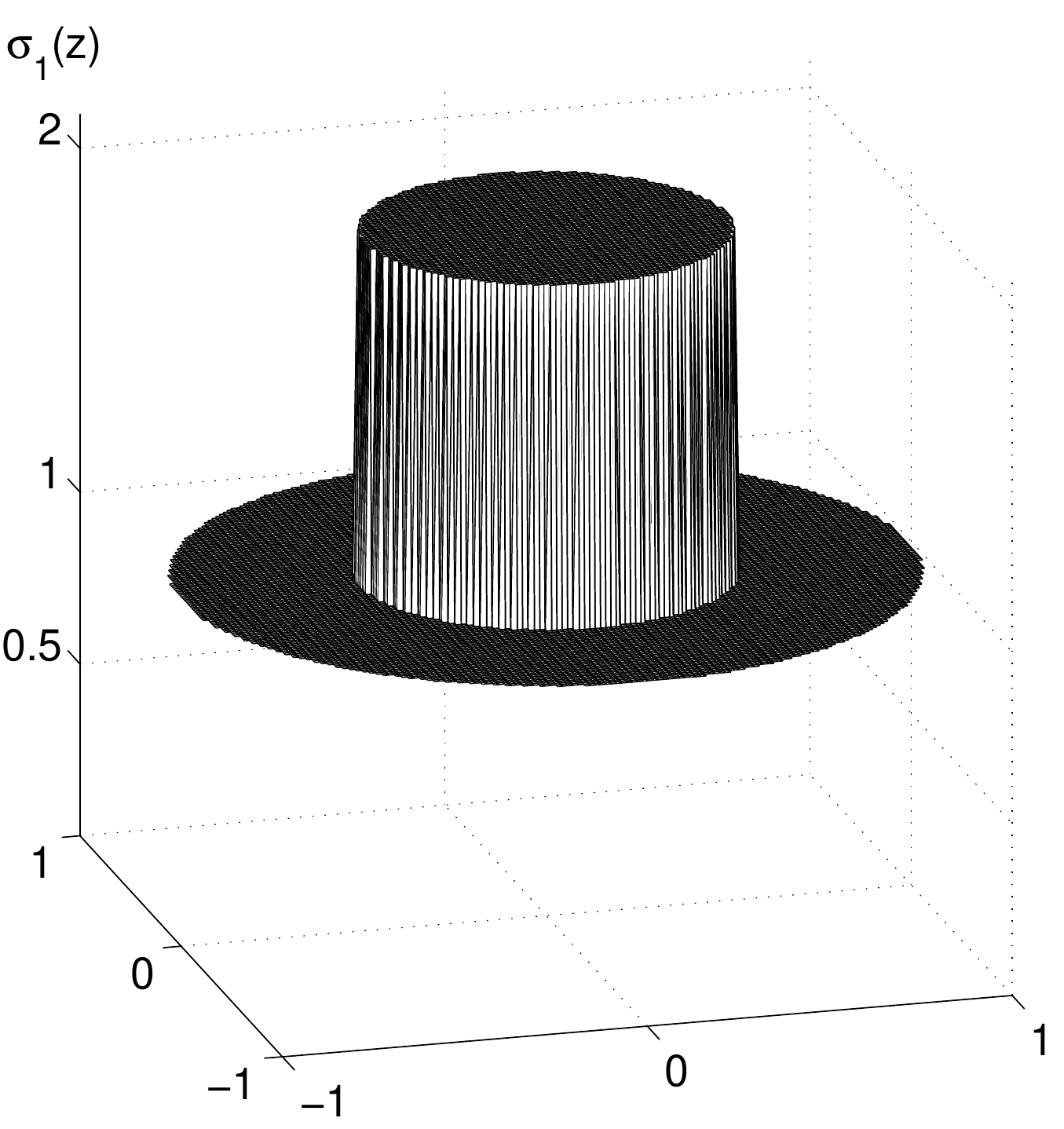}
\includegraphics[keepaspectratio=true, height = 5.5cm]{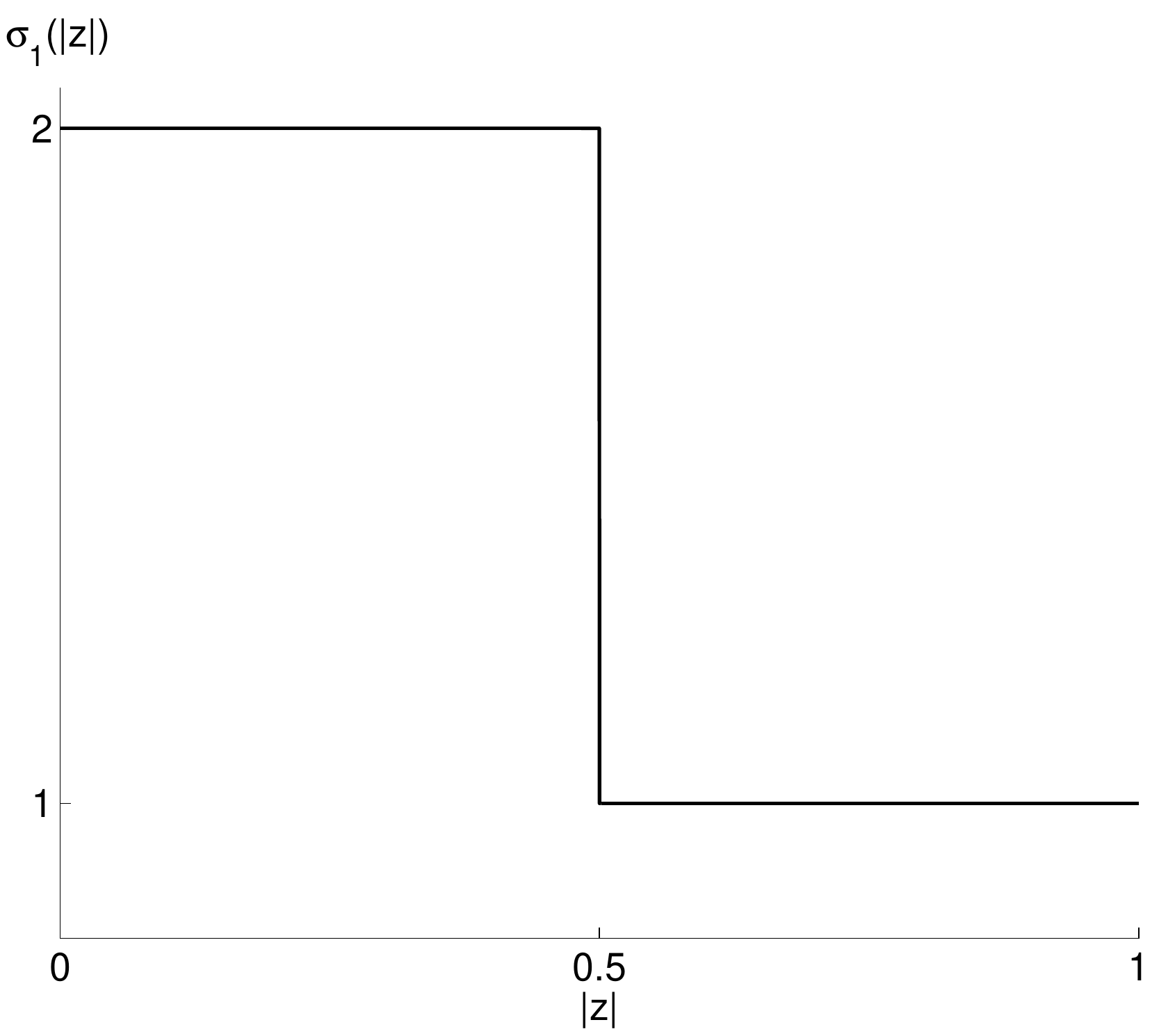}
\caption{\label{fig:sigma1}{\footnotesize Rotationally symmetric conductivity example $\sigma_1$. Left: plot of $\sigma_1$ as a function of $z$ ranging in the unit disc with jump discontinuities along the curve $|z|=0.5$. Right: Profile of $\sigma_1$ as a function of $|z|$ ranging in [0,1] with a jump discontinuity at $|z| = 0.5$.}}
\end{figure}


\emph{Computing the scattering transform}

\vskip 5mm

The numerical evaluation of the AP scattering transform $\tau (k)$ was computed for this example conductivity $\sigma_1$. The used algorithm generates for each fixed $k$ \CGO \ solutions for the Beltrami equation in the $z$-plane using the Huhtanen-Per\"am\"aki preconditioned solver (\cite{Huhtanen2012}) by executing standard $\C$-linear GMRES. To do so a $z$-grid is required. Such a $z$-grid is determined by a size parameter $\mbox{m}_{\mbox{\tiny z}}$ so that $2^{\mbox{\tiny m}_{\mbox{\tiny z}}}\times 2^{\mbox{\tiny m}_{\mbox{\tiny z}}}$ equidistributed points are taken in the $z$-plane within the square $[-\mbox{s}_{\mbox{\tiny z}},\mbox{s}_{\mbox{\tiny z}})\times [-\mbox{s}_{\mbox{\tiny z}},\mbox{s}_{\mbox{\tiny z}})$ with $\mbox{s}_{\mbox{\tiny z}}=2.1$ and step $\mbox{h}_{\mbox{\tiny z}} = \mbox{s}_{\mbox{\tiny z}}/2^{\mbox{\tiny m}_{\mbox{\tiny z}}-1}$. Let us denote the numerical computation of $\tau$ with a $z$-grid size parameter $\mbox{m}_{\mbox{\tiny z}}$ by $\widetilde{\tau}_{\mbox{\tiny m}_{\mbox{\tiny z}}}$.

We know that for radial, real-valued conductivities the function $\T$ is rotationally symmetric too (Theorem \ref{thrm:rot_sym}). Therefore, we can restrict the numerical computation of $\T$ to positive values of the $k$-plane. The chosen $k$-grid was all the points in the interval $[0,150]$ with step $\mbox{h} = 0.1$. Thus, the scattering transform is computed farther away from the origin than any previous computation.

The vector $\widetilde{\tau}_{\mbox{\tiny m}_{\mbox{\tiny z}}}$ corresponding to the aforementioned $k$-grid was computed for both $\mbox{m}_{\mbox{\tiny z}} = 10$ and $\mbox{m}_{\mbox{\tiny z}} = 11$. Next, the function $\widetilde{\T}_{\mbox{\tiny m}_{\mbox{\tiny z}}}$ defined as
\begin{equation*}
  \widetilde{\T}_{\mbox{\tiny m}_{\mbox{\tiny z}}}(k) = - 4 \pi i \,{\overline k} \,\,\widetilde{\tau}_{\mbox{\tiny m}_{\mbox{\tiny z}}}(k),
\end{equation*}
was evaluated on the same $k$-grid in $[0,150]$ for $\mbox{m}_{\mbox{\tiny z}} = 10$,$11$ from the AP scattering transform using identity \eqref{transform_comparison}.

One can consider the approximate scattering transform $\widetilde{\tau}_{11}$ (with $\mbox{m}_{\mbox{\tiny z}} = 11$) reliable since its accuracy is high in a double sense.

On one hand, Figure \ref{fig:comparison} shows the lack of significant differences between $\widetilde{\T}_{10}$ and $\widetilde{\T}_{11}$ for $k$ close to the origin. As a matter of fact, using the notation
\begin{equation}\label{rel_err_notation}
E_B := {\displaystyle\max_{|k|\in B} |\mbox{Im}(\widetilde{\tau}_{10}(k) - \widetilde{\tau}_{11}(k))|\over \displaystyle\max_{|k|\in B} |\mbox{Im}(\widetilde{\tau}_{11}(k))|}\cdot 100\%
\end{equation}
for a subinterval $B$ of $[0,150]$, it follows that
$$E_{[0,20]} = 0.3533 \%,\qquad E_{[20,40]} = 4.0258 \%,\qquad E_{[40,60]} = 9.8138 \% .$$

\begin{figure}[!hbp]
\includegraphics[keepaspectratio=true, height = 5.2cm]{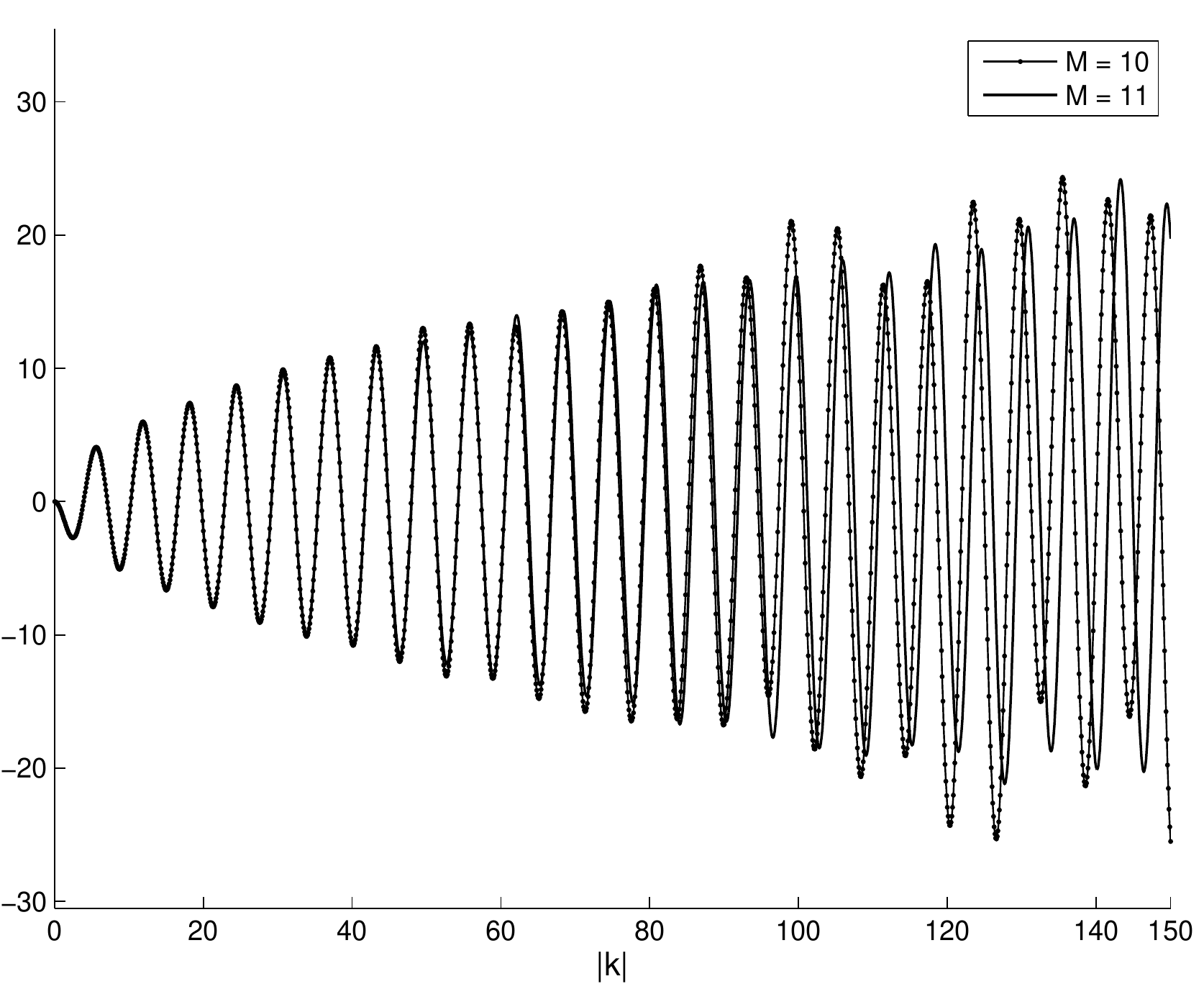}
\includegraphics[keepaspectratio=true, height = 5.2cm]{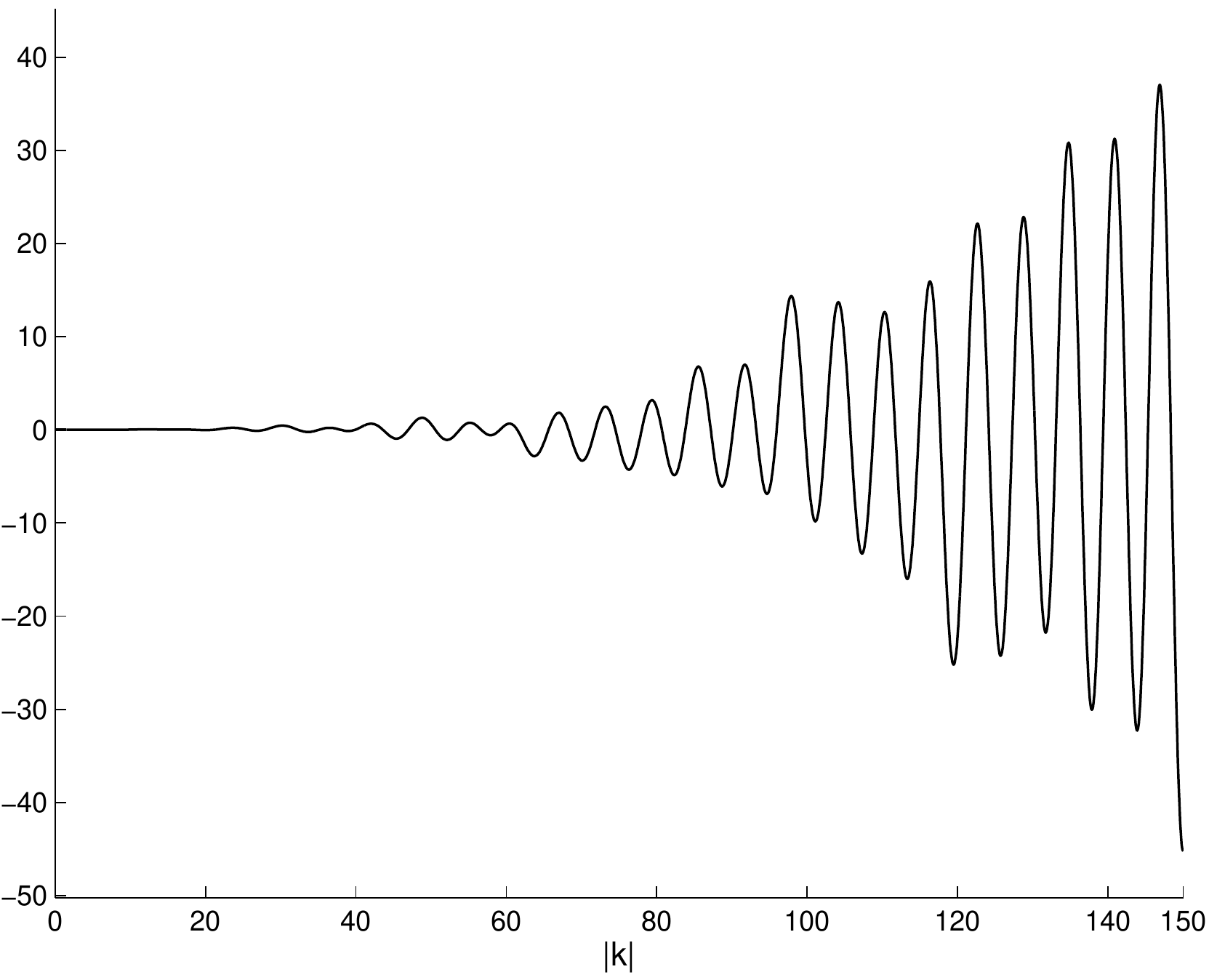}
\caption{\label{fig:comparison}{\footnotesize Comparison between $\widetilde{\T}_{10}$ and $\widetilde{\T}_{11}$ for $\sigma_1$. Left: Profiles of the real parts of $\widetilde{\T}_{10}$ (thin dotted line) and $\widetilde{\T}_{11}$ (thick solid line). Right: Profile of the difference $\mbox{Re}(\widetilde{\T}_{10} - \widetilde{\T}_{11})$. In both pictures $|k|$ ranges in [0,150] with step $\mbox{h} = 0.1$.}}
\end{figure}

On the other hand, $|\mbox{Re}( \widetilde{\tau}_{11})|$ is very small. By Theorem \ref{thrm:rot_sym} and formula \eqref{transform_comparison}, for a rotationally symmetric, real-valued conductivity the scattering transform $\tau$ restricted to real numbers has to be pure-imaginary valued. We obtain a nonzero real part $\mbox{Re}( \widetilde{\tau}_{\mbox{\tiny m}_{\mbox{\tiny z}}})$ for both $\mbox{m}_{\mbox{\tiny z}}$ values due to precision errors. In particular,
\[
|\mbox{Re}( \widetilde{\tau}_{11})|\leq 3.7196\times 10^{-9}.
\]

Note that one can extend $\widetilde{\T}_{11}$ to the disc centered at the origin with radius 150 in the $k$-plane by simple rotation since $\widetilde{\T}_{11}(k_1)=\widetilde{\T}_{11}(k_2)$ if $|k_1|=|k_2|$.

\vskip 5mm

\emph{Solving the D-bar equation}

\vskip 5mm

We compute a number of reconstructions of $\sigma_1$ using the D-bar method truncating the approximate scattering transform $\widetilde{\tau}_{11}$. This way the computation is optimized by restricting the degrees of freedom to the values of $\widetilde{\tau}_{11}(k)$ at grid points satisfying $|k|<R$, for some positive number $\mbox{R}$ called the cutoff frequency.


The program constructs the set of evaluation points in the $k$-plane for which the D-bar equation is solved for every fixed $z$, as a $k$-grid of $2^{\mbox{\tiny m}_{\mbox{\tiny k}}}\times 2^{\mbox{\tiny m}_{\mbox{\tiny k}}}$ points within the square $[-\mbox{s}_{\mbox{\tiny k}},\mbox{s}_{\mbox{\tiny k}})\times [-\mbox{s}_{\mbox{\tiny k}},\mbox{s}_{\mbox{\tiny k}})$ with $\mbox{s}_{\mbox{\tiny k}}=2.3\times R$ and step $\mbox{h}_{\mbox{\tiny k}} = \mbox{s}_{\mbox{\tiny k}}/2^{\mbox{\tiny m}_{\mbox{\tiny k}}-1}$. Let $\widetilde{\sigma}_1(z;\mbox{m}_{\mbox{\tiny k}},R)$ denote the shortcut reconstruction of $\sigma_1$ with size parameter $\mbox{m}_{\mbox{\tiny k}}$ and truncation radius $\mbox{R}$.

Again we can restrict the computation (of $\widetilde{\sigma}_1(z;\mbox{m}_{\mbox{\tiny k}},R)$ in this case) to positive values in the $z$-domain and afterwards extend to the plane by rotation. We evaluated $\widetilde{\sigma}_1(z;\mbox{m}_{\mbox{\tiny k}},R)$ for each point $z$ in a grid of a certain number $N_x$ of equispaced points in the real interval [0,1].

The vector $\widetilde{\sigma}_1(z;\mbox{m}_{\mbox{\tiny k}},R)$ was computed for $\mbox{R} = $ 5, 10, 15, 20, 25, 30, 35, 40, 50,..., 150.  For $\mbox{R}$ = 5, 10 we took $\mbox{m}_{\mbox{\tiny k}}$ = 8, 9 and for the rest of $\mbox{R}$ values we chose $\mbox{m}_{\mbox{\tiny k}}$ = 12. These are by far the largest nonlinear cut-off frequencies studied numerically so far. In Figure \ref{fig:Nrecon_sigma1} the reader can see three of these reconstruction examples.

Since the shortcut reconstruction for $\sigma_1$ is real-valued,
the obtained imaginary parts  are due to precision errors. Such
errors are displayed multiplied by $10^8$ in the following tables:

\vskip 3mm

{\tiny

\begin{tabular}{|c|c|c|c|c|c|c|c|c|c|c|c|c|c|c|}
  \hline

  $\mbox{m}_{\mbox{\tiny k}}$ & 8 & 8 & 9 & 9 & 12 & 12 & 12 & 12 & 12 & 12 & 12 & 12 & 12 & 12\\
  \hline
  $\mbox{R}$ & 5 & 10 & 5 & 10 & 15 & 20 & 25 & 30 & 35 & 40 & 50 & 60 & 70 & 80\\
  \hline
  \textbf{error$_1\times 10^8$} & 4.67 & 2.36 & 4.67 & 2.36 & 2.15 & 2.24 & 2.57 & 2.84 & 3.04 & 3.26 & 3.55 & 3.78 & 4 & 4.19\\
  \hline
\end{tabular}

\vskip 2mm

\begin{tabular}{|c|c|c|c|c|c|c|c|}
  \hline

  $\mbox{m}_{\mbox{\tiny k}}$ &  12 & 12 & 12 & 12 & 12 & 12 & 12\\
  \hline
  $\mbox{R}$ &  90 & 100 & 110 & 120 & 130 & 140 & 150\\
  \hline
  \textbf{error$_1\times 10^8$} &  4.36 & 4.48 & 4.57 & 4.61 & 4.62 & 4.76 & 4.89\\
  \hline
\end{tabular}

}

\vskip 2mm

\noindent where $\text{error}_1 = \,\,\displaystyle\max_{|z|\leq
1}|\mbox{Im}(\widetilde{\sigma}_1(z;\mbox{m}_{\mbox{\tiny
k}},R))|.$

\clearpage

\begin{figure}[!hbp]
\includegraphics[keepaspectratio=true, height = 5.3cm]{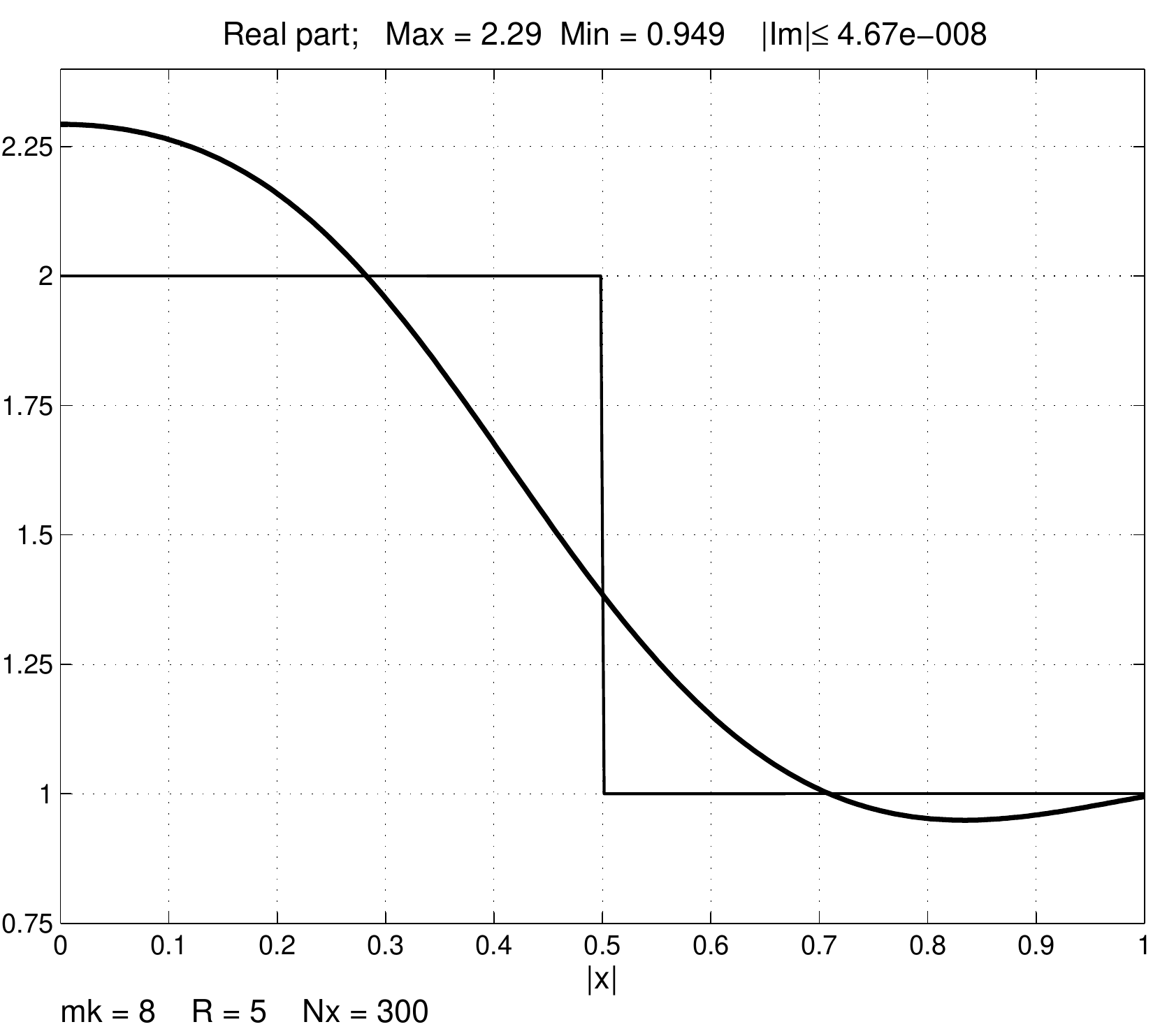}
\includegraphics[keepaspectratio=true, height = 5.3cm]{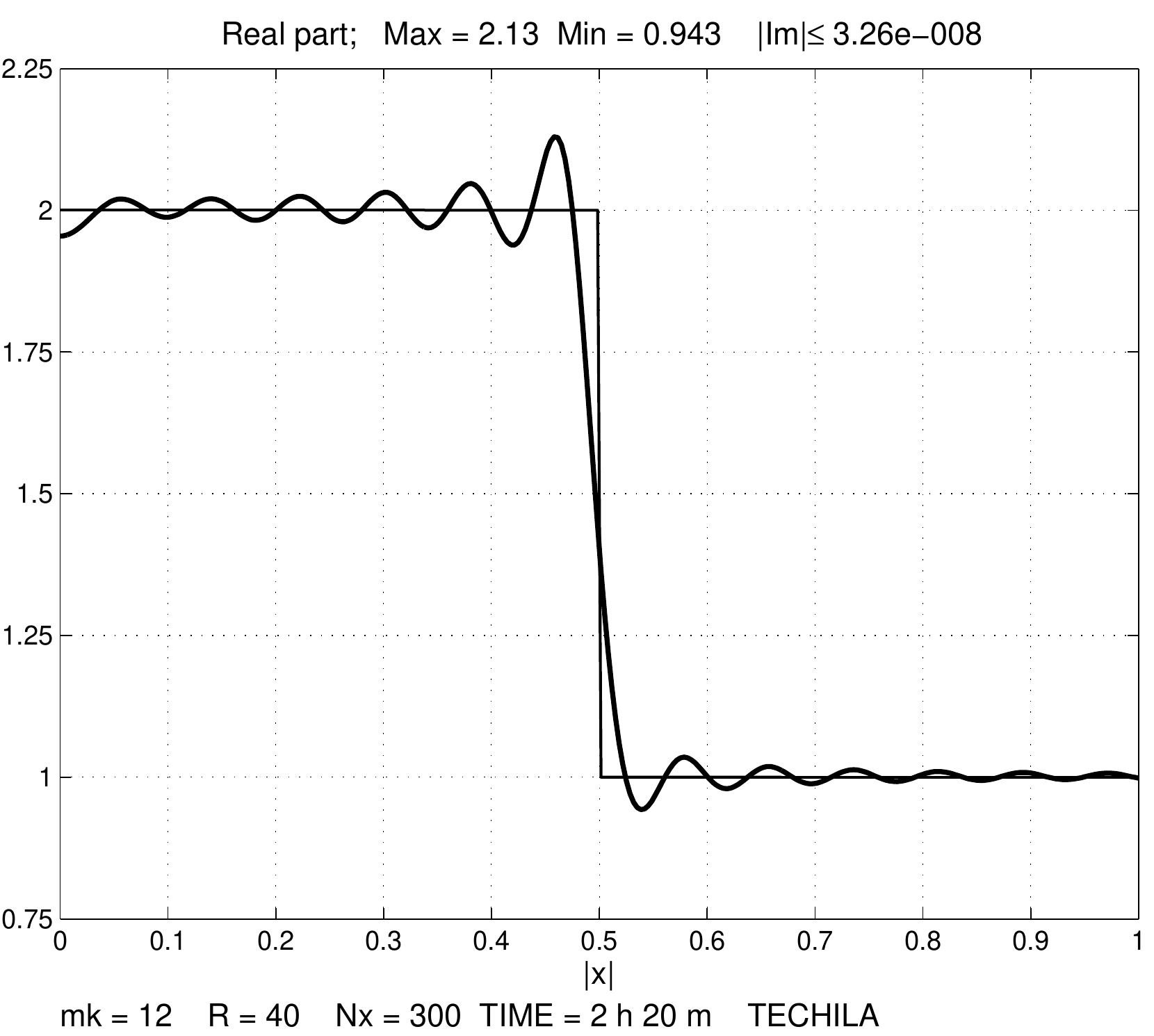}
\includegraphics[keepaspectratio=true, height = 6cm]{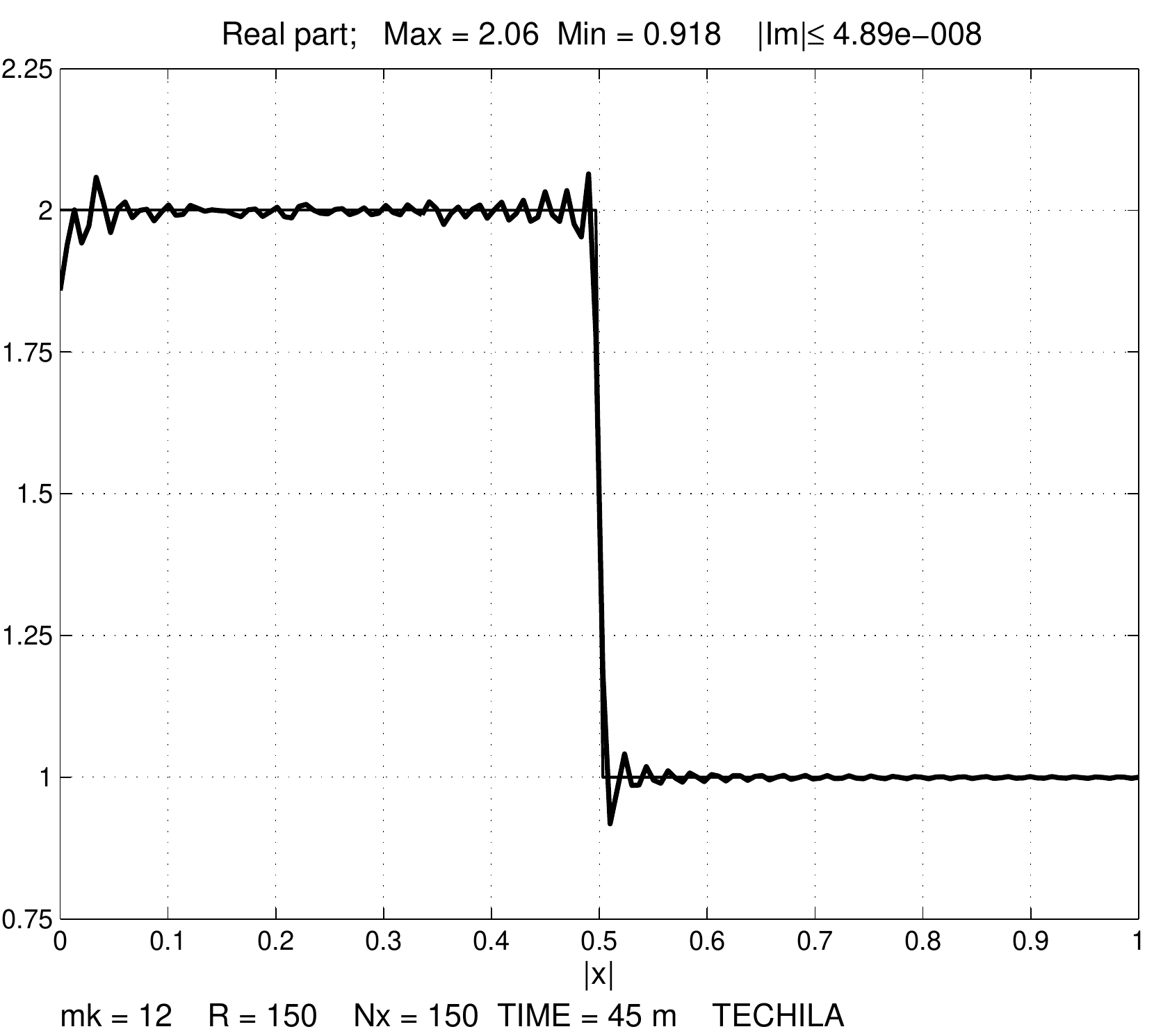}
\caption{\label{fig:Nrecon_sigma1}{\footnotesize  Three shortcut method reconstruction examples computed from a scattering transform corresponding to $\sigma_1$. Each picture shows the real part of the approximate reconstruction (thick line) and the original conductivity $\sigma_1$ (thin line). In addition, the following information is shown: maximum and minimum values of the real part of the reconstruction, maximum of the absolute value of the imaginary part of the reconstruction (error), taken parameters $\mbox{m}_{\mbox{\tiny k}}$, $\mbox{R}$, $N_x$ and computation time if possible. The examples for $\mbox{m}_{\mbox{\tiny k}} = 12$ and $\mbox{R} = 40, 150$ were computed using grid computation provided by Techila.}}
\end{figure}

\subsubsection{Example $\sigma_2$}

Now, we repeat the same experiments for another rotationally symmetric conductivity example $\sigma_2$ defined as follows: $\sigma_2(z) = 2$ if $|z|<0.1$, $0.2 < |z|< 0.3$ or $0.4 < |z|< 0.5$ and $\sigma_2(z) = 1$ otherwise. Figure \ref{fig:sigma2} shows its profiles.

\clearpage

\begin{figure}[!hbp]
\includegraphics[keepaspectratio=true, height = 5cm]{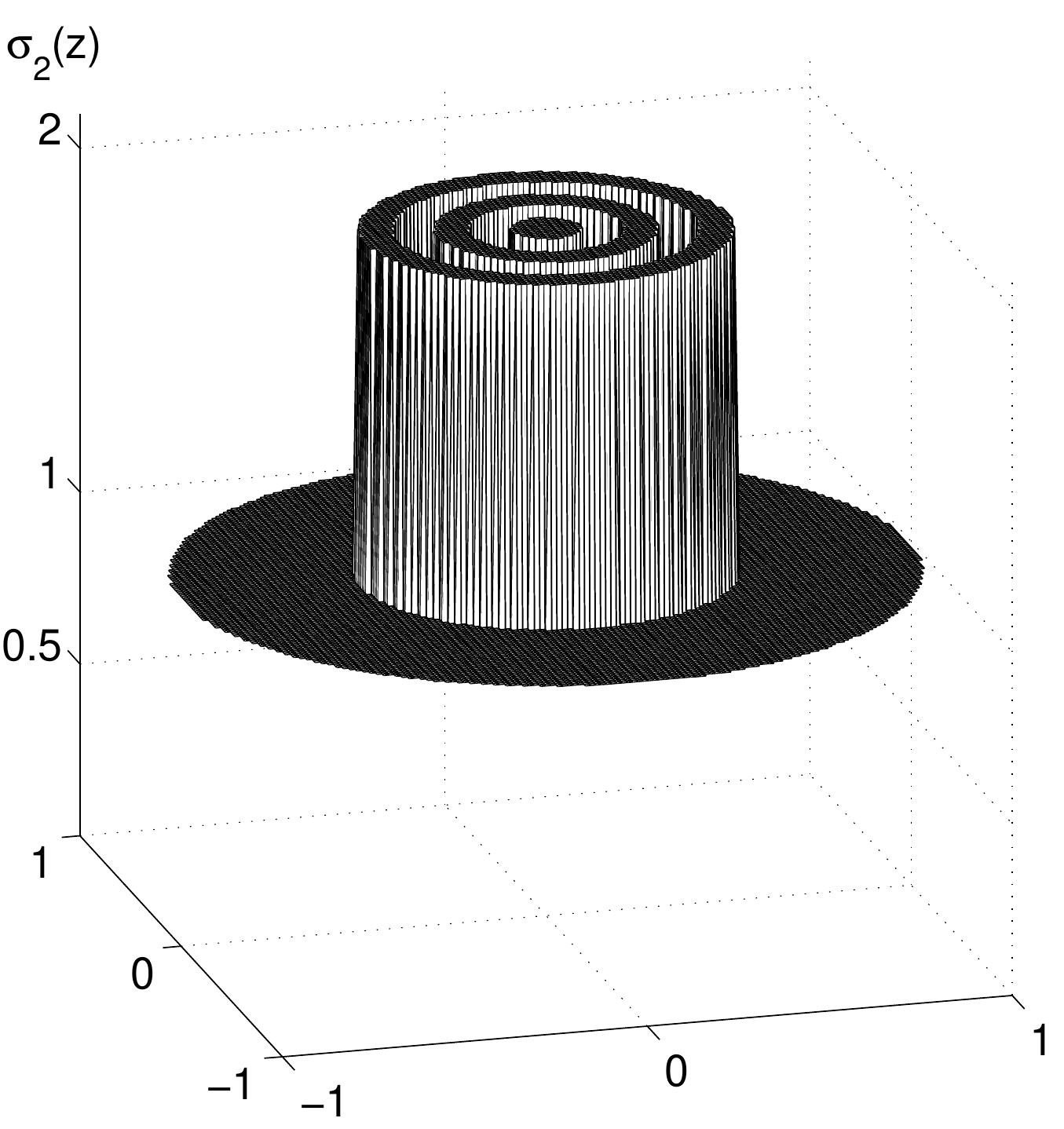}
\includegraphics[keepaspectratio=true, height = 5.5cm]{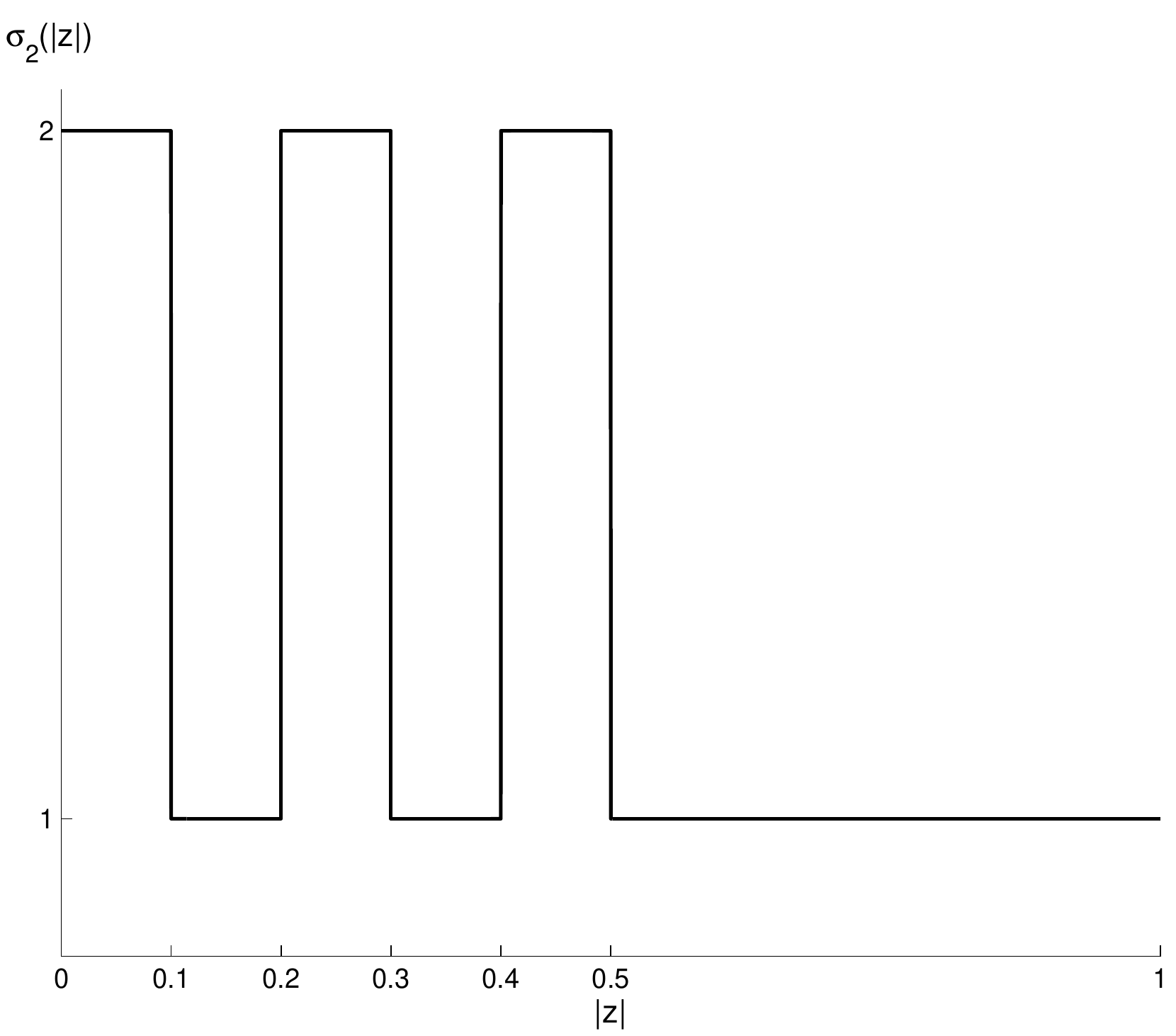}
\caption{\label{fig:sigma2}{\footnotesize Rotationally symmetric conductivity example $\sigma_2$. Left: plot of $\sigma_2$ as a function of $z$ ranging in the unit disc with jump discontinuities along the curves $|z|=0.1$, $|z|=0.2$, $|z|=0.3$, $|z|=0.4$, $|z|=0.5$. Right: Profile of $\sigma_2$ as a function of $|z|$ ranging in [0,1] with jump discontinuities at $|z|=0.1$, $|z|=0.2$, $|z|=0.3$, $|z|=0.4$, $|z| = 0.5$.}}
\end{figure}

\vskip 5mm

\emph{Computing the scattering transform}

\vskip 5mm

For conductivity $\sigma_2$ the scattering transform $\widetilde{\tau}_{11}$ (with $\mbox{m}_{\mbox{\tiny z}} = 11$) turns out to be very accurate too. Indeed, $\widetilde{\T}_{10}$ and $\widetilde{\T}_{11}$ corresponding to $\sigma_2$ are quite similar for $|k|$ small in view of Figure \ref{fig:comparison3} and the following facts (see notation \eqref{rel_err_notation}):
\begin{align*}
&E_{[0,20]} = 0.4606 \%,\qquad E_{[20,40]} = 3.7474 \%,\qquad E_{[40,60]} = 6.4308 \% ,\qquad\text{for }\sigma_2,\\
&|\mbox{Re}( \widetilde{\tau}_{11})|\leq 1.3653\times 10^{-7},\qquad\text{for }\sigma_2 .\\
\end{align*}

\emph{Solving the D-bar equation}

\vskip 5mm

Likewise the vector $\widetilde{\sigma}_2(z;\mbox{m}_{\mbox{\tiny k}},R)$ was computed for $\mbox{R} = $ 5, 10, 15, 20, 25, 30, 35, 40, 50,..., 150.  For $\mbox{R}$ = 5, 10 we took $\mbox{m}_{\mbox{\tiny k}}$ = 8, 9 and for the rest of $\mbox{R}$ values we chose $\mbox{m}_{\mbox{\tiny k}}$ = 12. All these reconstructions were computed using parallel computation provided by Techila. Figure \ref{fig:Nrecon_sigma3} shows some of these numerical results. Below we list the errors $\text{error}_2 = \,\,\max_{|z|\leq 1}|\mbox{Im}(\widetilde{\sigma}_2(z;\mbox{m}_{\mbox{\tiny k}},R))|$:

{\tiny

\begin{tabular}{|c|c|c|c|c|c|c|c|c|c|c|c|c|c|c|}
  \hline

  $\mbox{m}_{\mbox{\tiny k}}$ & 8 & 8 & 9 & 9 & 12 & 12 & 12 & 12 & 12 & 12 & 12 & 12 & 12 & 12\\
  \hline
  $\mbox{R}$ & 5 & 10 & 5 & 10 & 15 & 20 & 25 & 30 & 35 & 40 & 50 & 60 & 70 & 80\\
  \hline
  \textbf{error$_2\times 10^7$} & 1.56 & 2.16 & 1.56 & 2.17 & 5.12 & 5.65 & 3.97 & 4.88 & 6.48 & 14.1 & 19.3 & 14.8 & 18.4 & 28.4\\
  \hline
\end{tabular}

\begin{tabular}{|c|c|c|c|c|c|c|c|}
  \hline

  $\mbox{m}_{\mbox{\tiny k}}$ &  12 & 12 & 12 & 12 & 12 & 12 & 12\\
  \hline
  $\mbox{R}$ &  90 & 100 & 110 & 120 & 130 & 140 & 150\\
  \hline
  \textbf{error$_2\times 10^7$}  & 20.3 & 20.7&  28.8 & 25.3 & 24.3 & 21.8 & 19.7\\
  \hline
\end{tabular}

}


\begin{figure}[!hbp]
\includegraphics[keepaspectratio=true, height = 4.9cm]{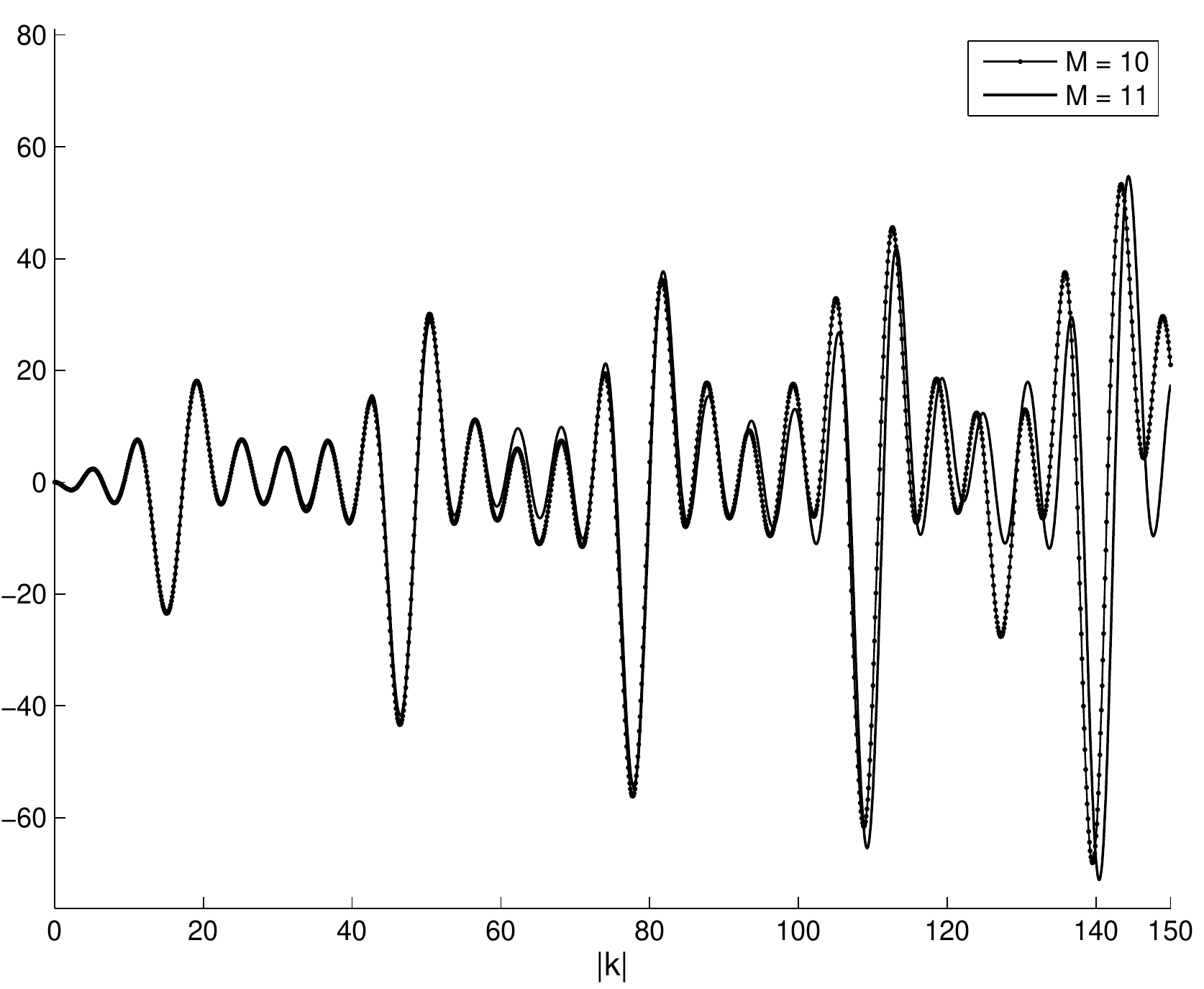}
\includegraphics[keepaspectratio=true, height = 4.9cm]{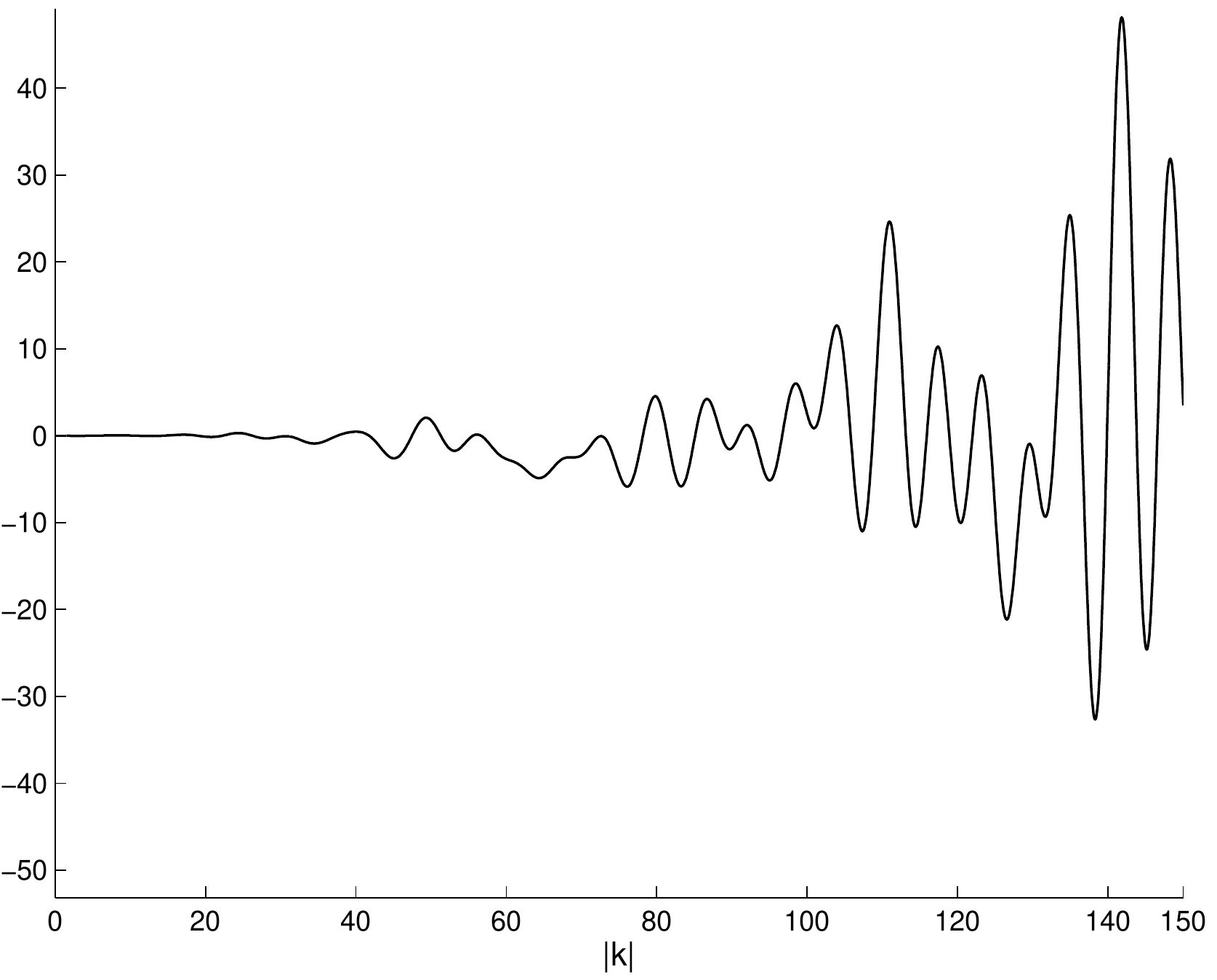}
\caption{\label{fig:comparison3}{\footnotesize Comparison between $\widetilde{\T}_{10}$ and $\widetilde{\T}_{11}$ for $\sigma_2$. Left: Profiles of the real parts of $\widetilde{\T}_{10}$ (thin dotted line) and $\widetilde{\T}_{11}$ (thick solid line). Right: Profile of the difference $\mbox{Re}(\widetilde{\T}_{10} - \widetilde{\T}_{11})$. In both pictures $|k|$ ranges in [0,150] with step $\mbox{h} = 0.1$. Note that the vertical axis scales are different from Figure \ref{fig:comparison}. }}
\end{figure}

\begin{figure}[!hbp]
\includegraphics[keepaspectratio=true, height = 5.3cm]{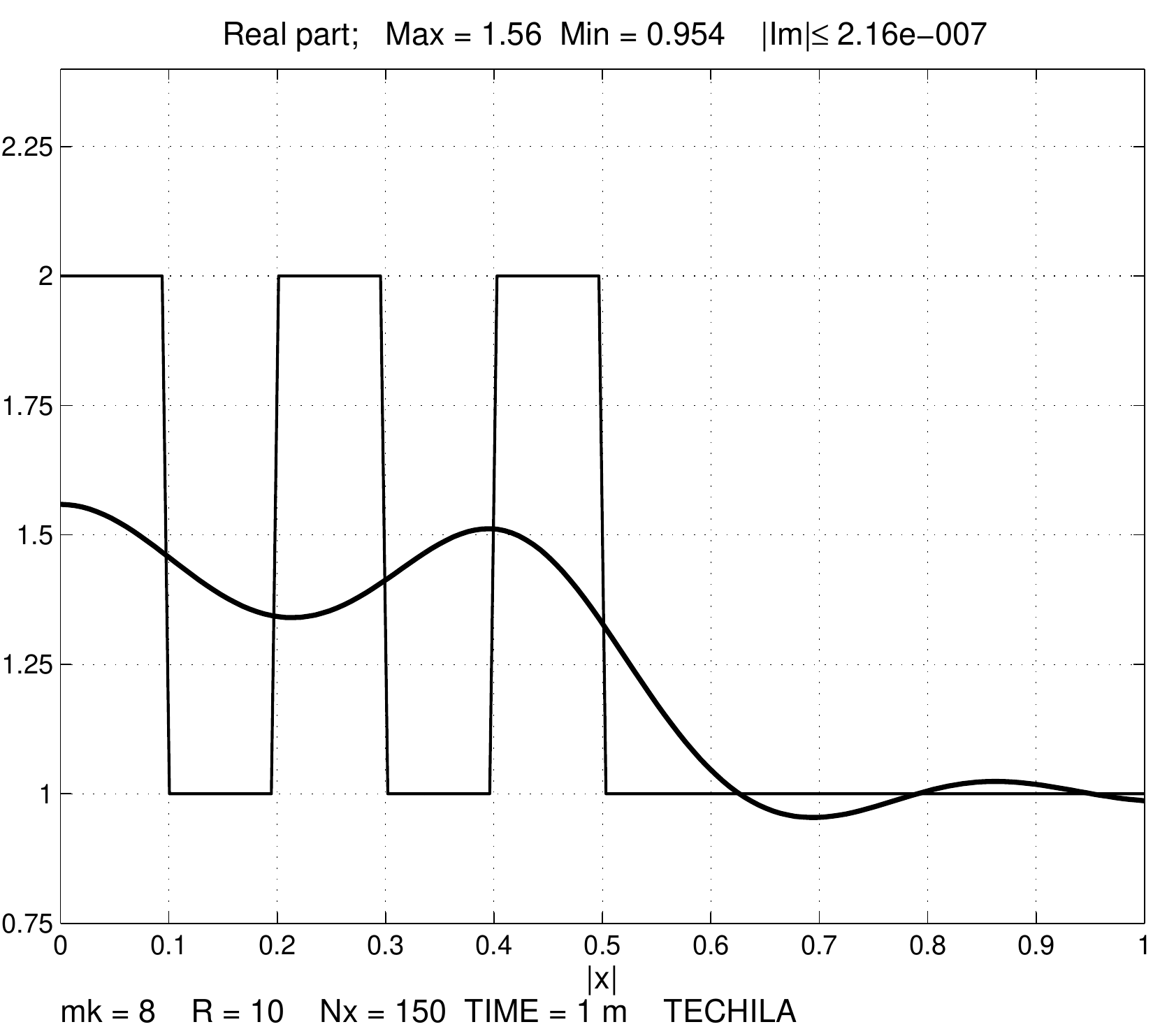}
\includegraphics[keepaspectratio=true, height = 5.3cm]{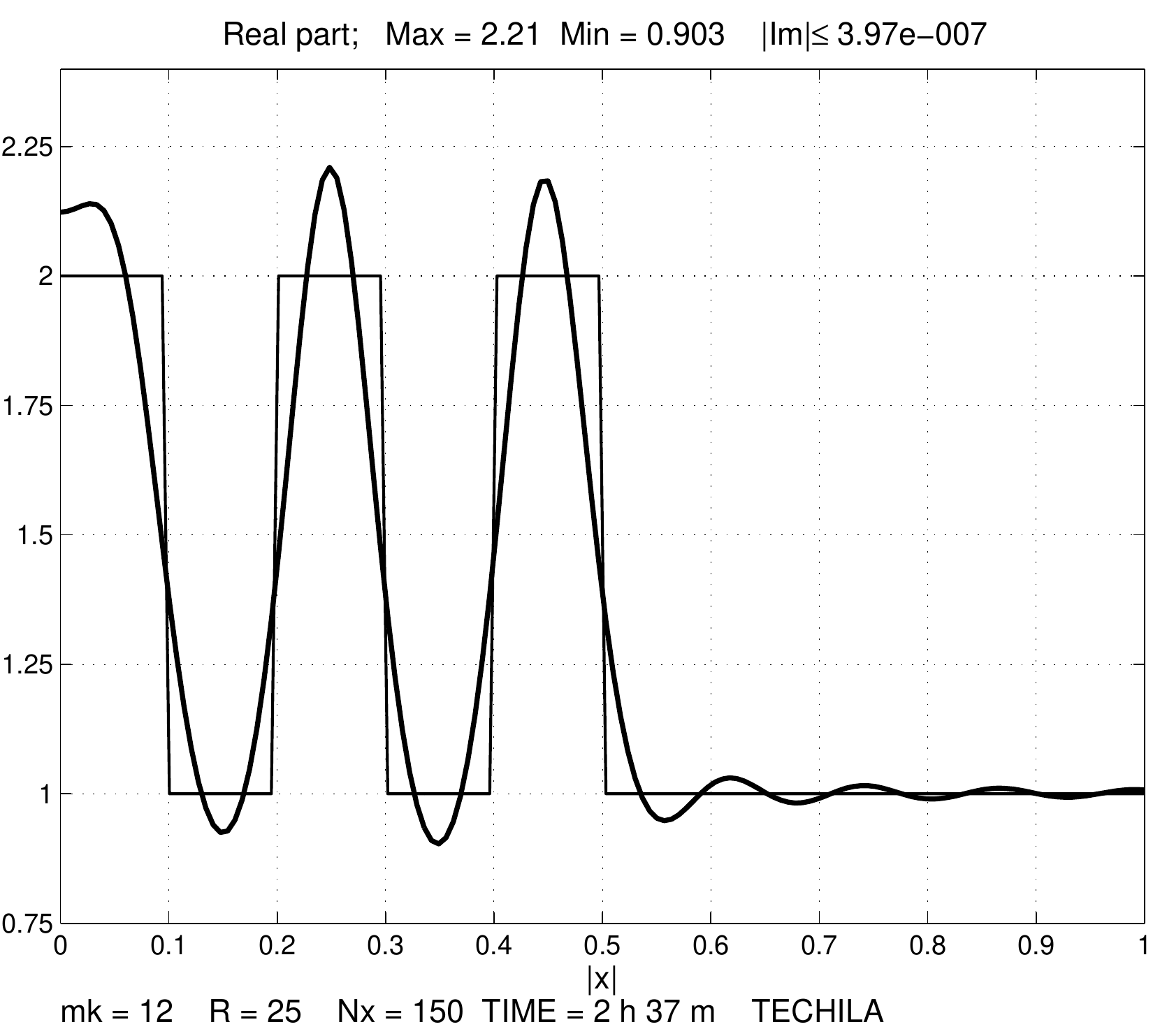}
\includegraphics[keepaspectratio=true, height = 5.3cm]{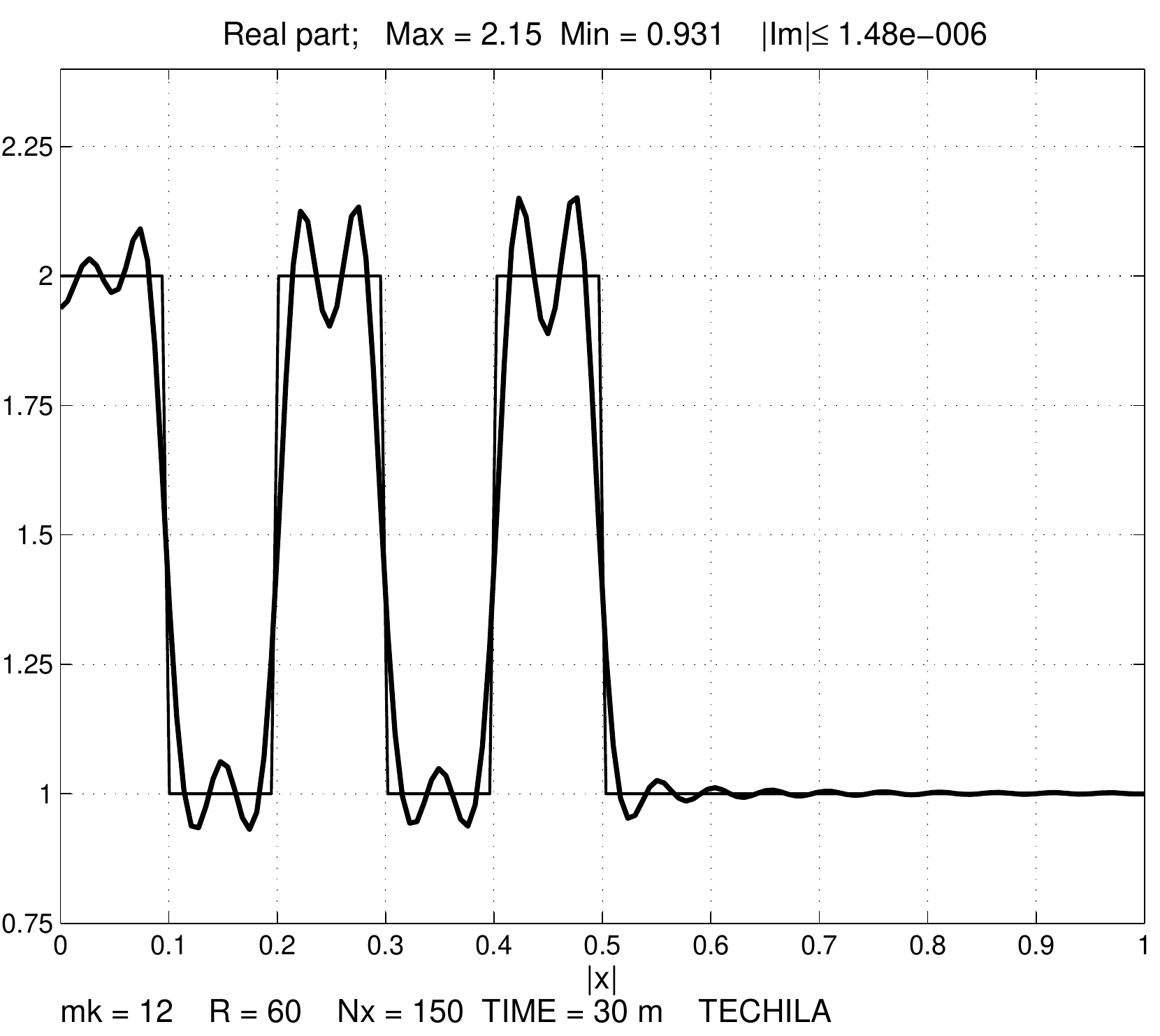}
\includegraphics[keepaspectratio=true, height = 5.3cm]{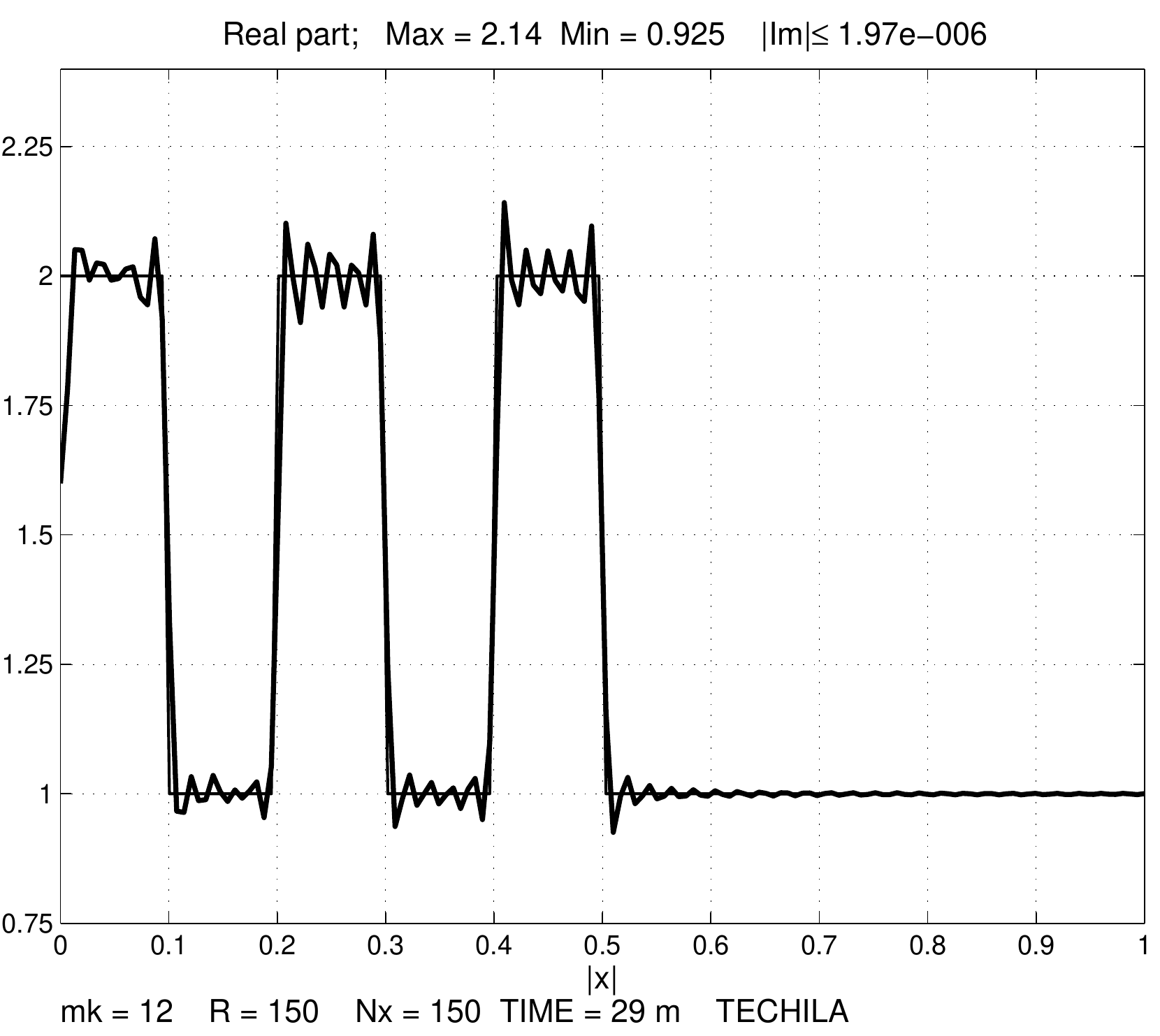}
\caption{\label{fig:Nrecon_sigma3}{\footnotesize  Four shortcut method reconstruction examples computed from a scattering transform corresponding to $\sigma_2$. Each picture shows the real part of the approximate reconstruction (thick line) and the original conductivity $\sigma_2$ (thin line). In addition, the following information is shown: maximum and minimum values of the real part of the reconstruction, maximum of the absolute value of the imaginary part of the reconstruction (error), taken parameters $\mbox{m}_{\mbox{\tiny k}}$, $\mbox{R}$, $N_x$ and computation time. Axes scale is the same in the four profiles. These examples were computed using grid computation provided by Techila.}}
\end{figure}

\clearpage

\subsection{Comparison of the two methods in nonsymmetric
cases}\label{sec:comparison}

We apply both reconstruction methods to some discontinuous non-radial ``checkerboard\-style'' conductivities without assuming EIT data.  We choose the conductivity examples $\sigma_3$, $\sigma_4$ shown in Figure \ref{fig:PWC}. Both take value 1 near the boundary of the unit disc. Notice that the contrasts of these examples $\sigma_3$, $\sigma_4$ (defined as the difference between the maximum and the minimum) are $1.5$ and $2.8$, respectively.

\begin{figure}[!hp]
\begin{picture}(320,150)
\put(-20,290){\includegraphics[width=7cm]{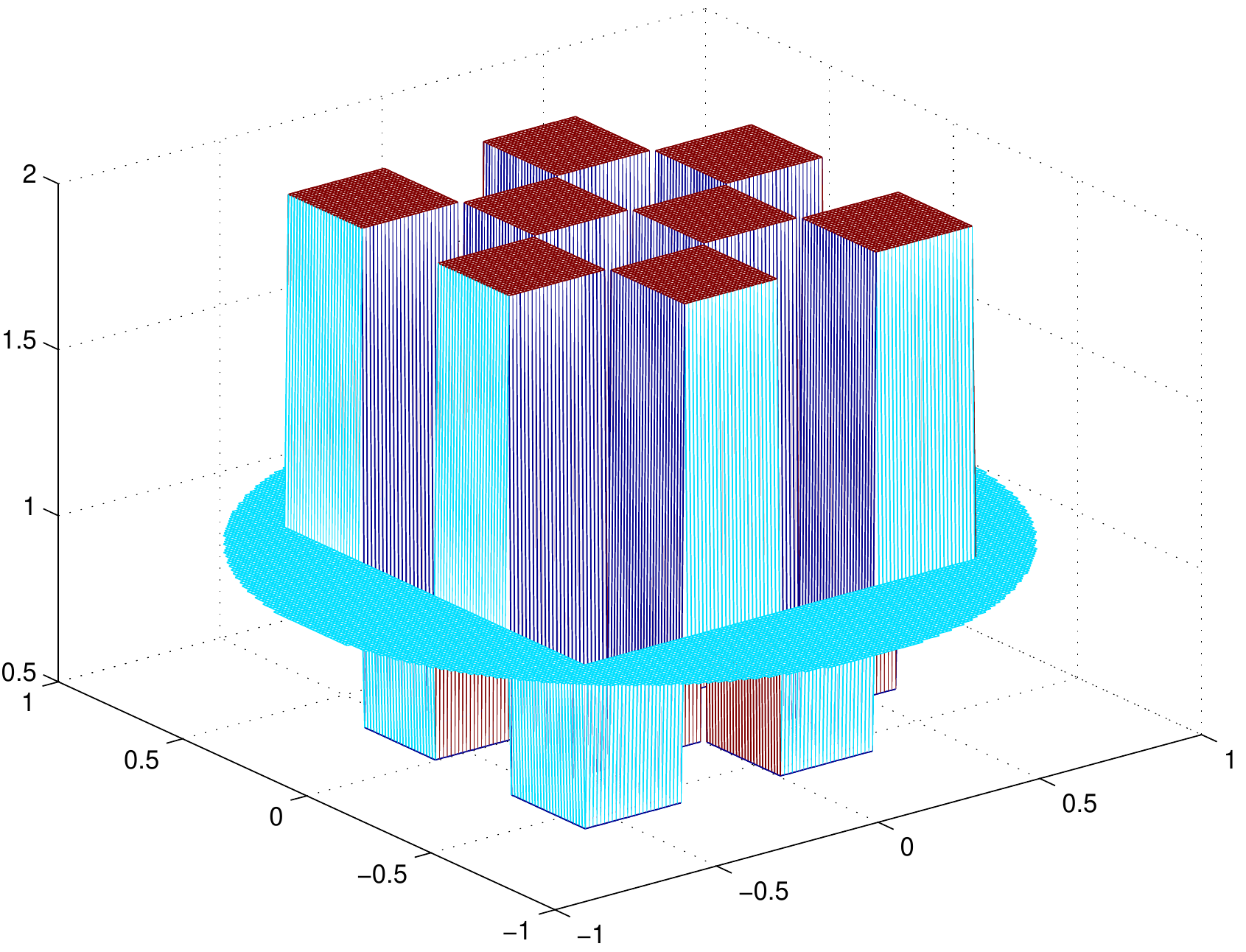}}
\put(70,260){\large a)}
\put(190,290){\includegraphics[width=5cm]{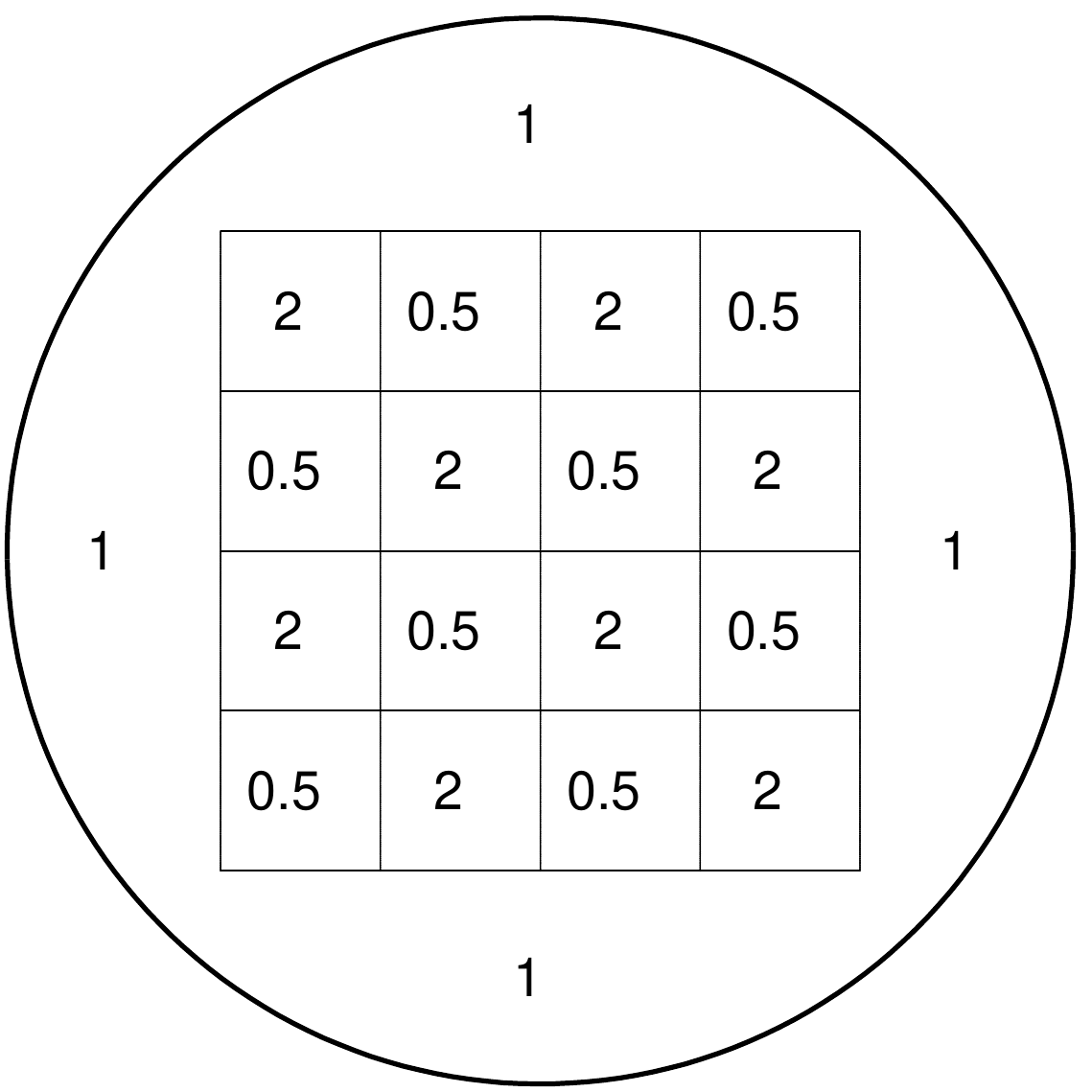}}
\put(260,260){\large b)}
\put(-20,90){\includegraphics[width=7cm]{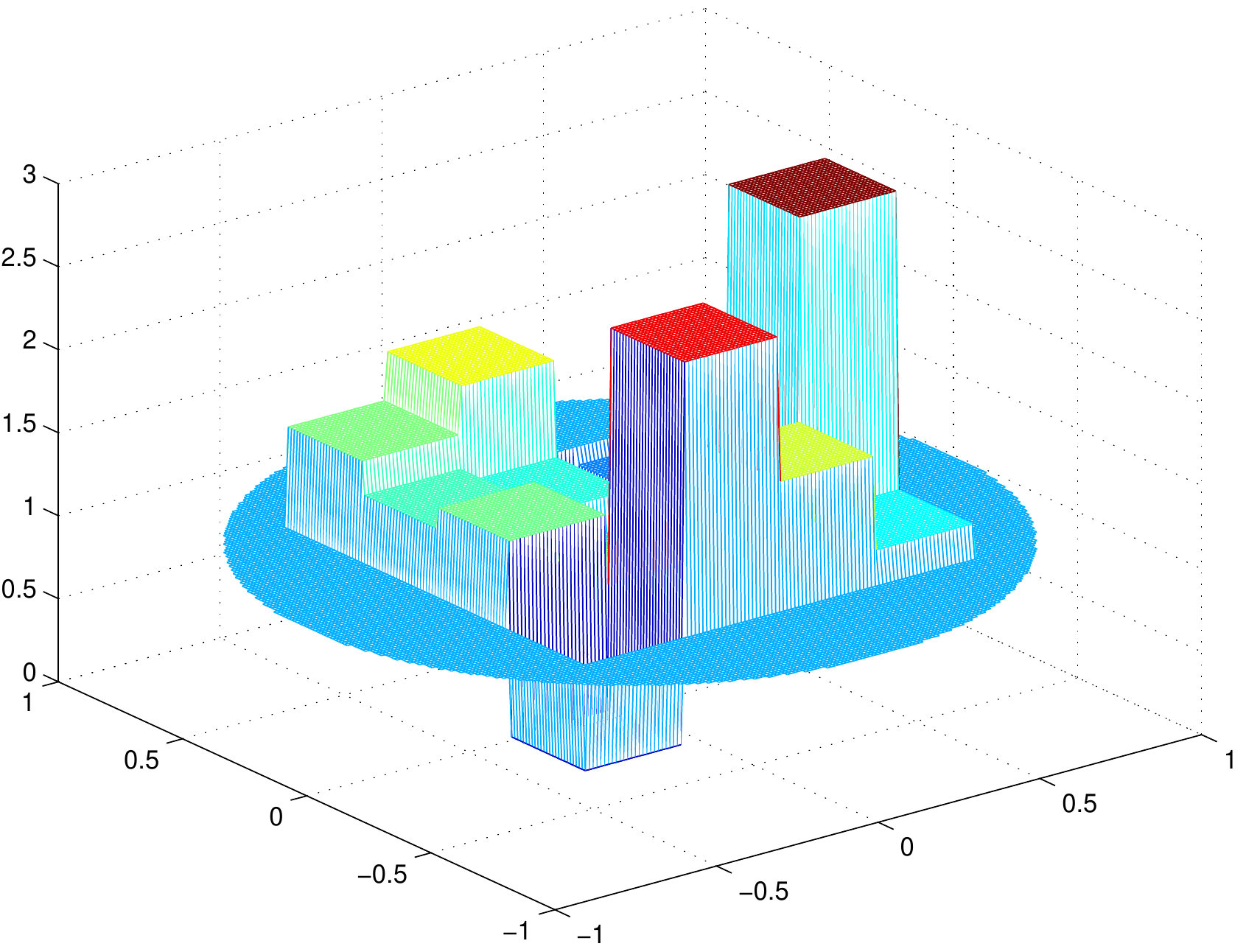}}
\put(70,60){\large c)}
\put(190,90){\includegraphics[width=5cm]{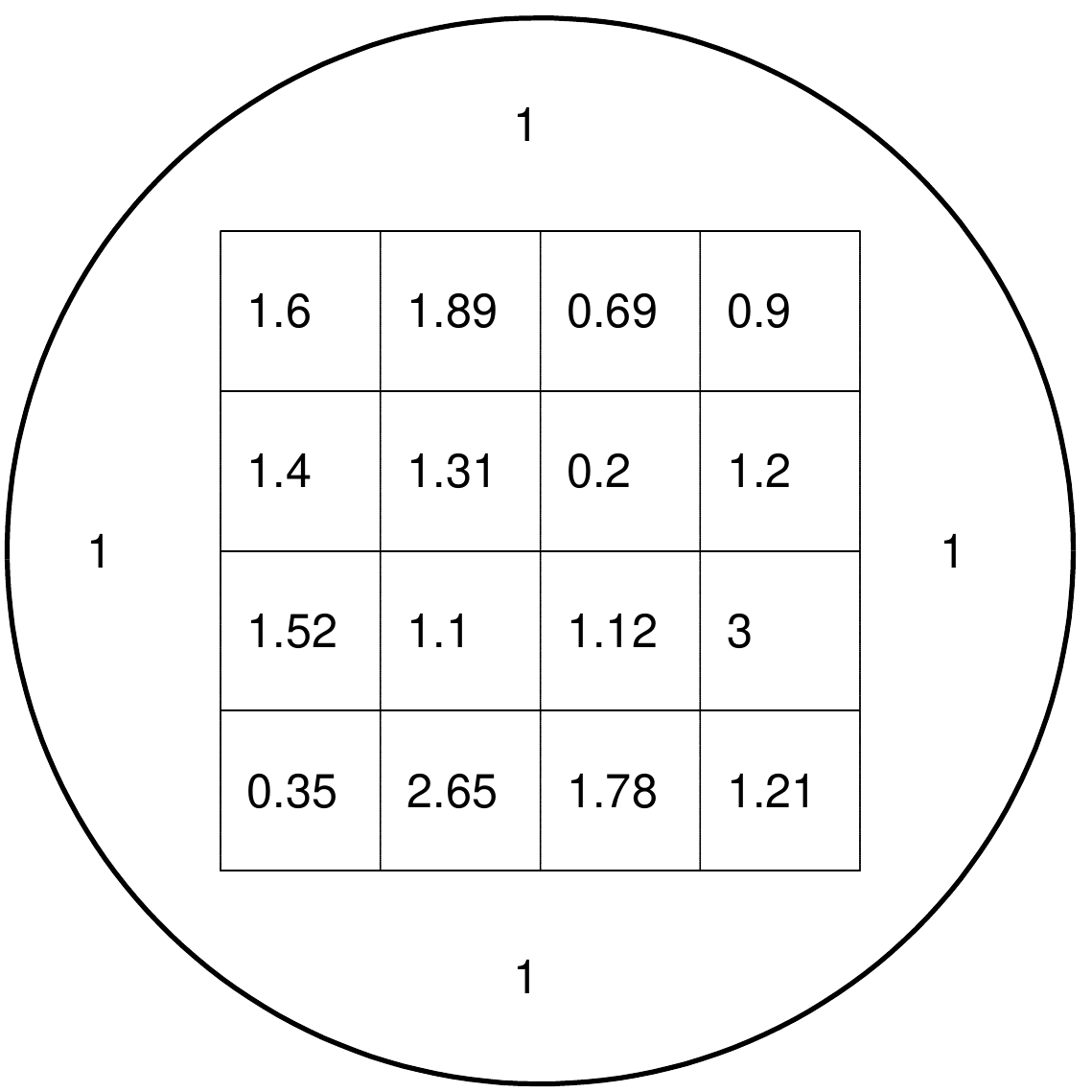}}
\put(260,60){\large d)}
\end{picture}
\caption{\label{fig:PWC}Checkerboard style conductivities $\sigma_3$, $\sigma_4$: a), c) show the picture for $\sigma_3$, $\sigma_4$, resp., using mesh Matlab function. b), d) show the values for $\sigma_3$, $\sigma_4$, resp., over each tile. }
\end{figure}
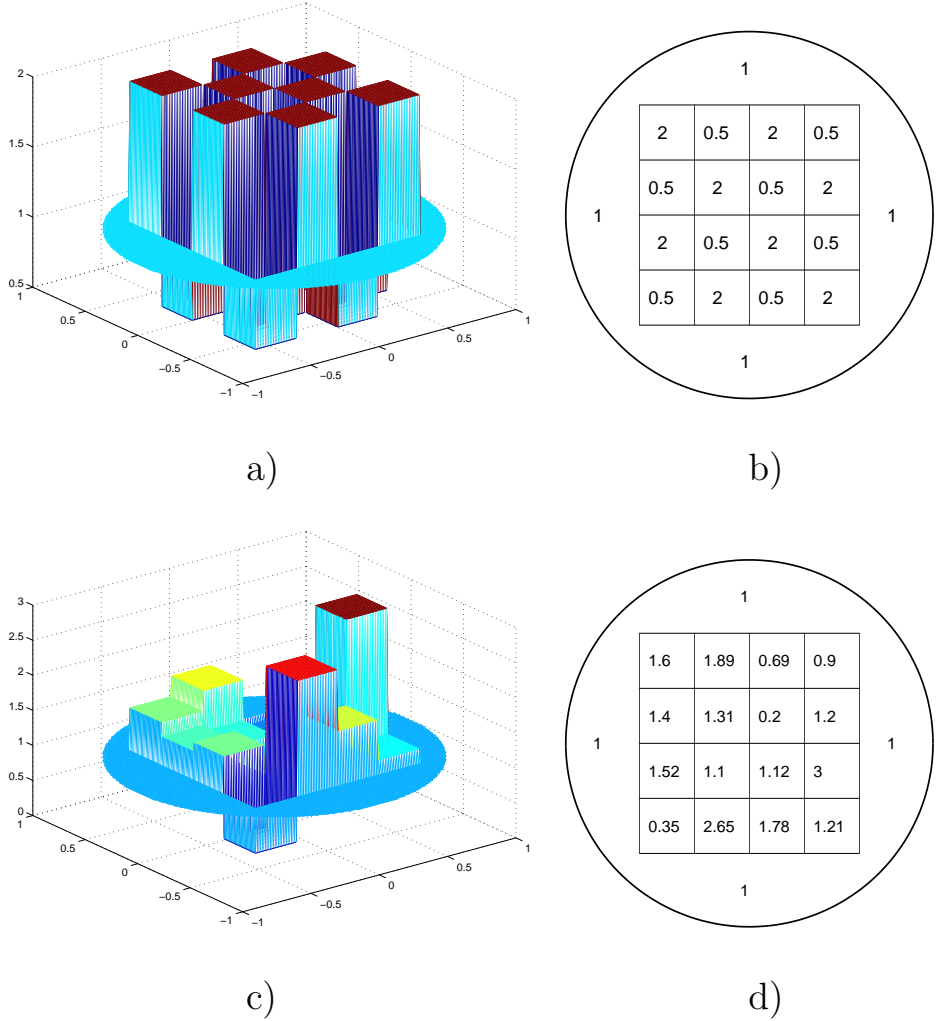

\vskip 10cm

Figures \ref{fig:nice_PWC_sigma2}, \ref{fig:nice_board_sigma2} show the approximate reconstructions for both examples and certain cut-off frequencies $\mbox{R}$ using both methods. The relative errors with sup and $l^2$ norms over the matrix of chosen points in the unit disc are pointed out. We denote such errors as follows:
\[
\text{sup} = {\norm{\sigma-\widetilde{\sigma}}_{l^{\infty}}\over \norm{\sigma}_{l^{\infty}}}\cdot 100\, \%,
\qquad \text{sqr} = {\norm{\sigma-\widetilde{\sigma}}_{l^{2}}\over \norm{\sigma}_{l^{2}}}\cdot 100 \, \%,
\]
\noindent where $\sigma$ denotes the actual conductivity and $\widetilde{\sigma}$ the approximate reconstruction.

\vskip 3mm

\underline{\emph{Low-pass transport matrix method}}

\vskip 3mm

In order to test the low-pass transport matrix method for each conductivity, we obtain the CGO solutions $f_{\mu}(z_0 , k)$ , $f_{-\mu}(z_0 , k)$ directly from the known conductivities without simulating EIT data. The point $z_0$ is outside the unit disc and we take the $k$-grid of $2^{\mbox{\tiny m}_{\mbox{\tiny k}}}\times 2^{\mbox{\tiny m}_{\mbox{\tiny k}}}$ equispaced points in the square $[-\mbox{R}, \mbox{R})^2$ with step $\mbox{h}_{\mbox{\tiny k}} = \mbox{R} / 2^{\mbox{\tiny m}_{\mbox{\tiny k}}-1}$ and cut-off frequencies $\mbox{R}\leq 50$. Next, transportation of $f_{\mu}(z_0 , k)$ , $f_{-\mu}(z_0 , k)$ to the unit disc and the final reconstruction $\sigma^{(\mbox{R})}$ are performed following the steps described in Section \ref{section:APmethod}. To this end, in addition to the $k$-grid, a grid in the $z$-variable of $2^{\mbox{\tiny m}_{\mbox{\tiny z}}}\times 2^{\mbox{\tiny m}_{\mbox{\tiny z}}}$ equidistributed points in the square $[-\mbox{s}_{\mbox{\tiny z}}, \mbox{s}_{\mbox{\tiny z}})^2$, with step $h_z = \mbox{s}_{\mbox{\tiny z}} / 2^{\mbox{\tiny m}_{\mbox{\tiny z}}-1}$, is required.

For $\sigma_3$, $\sigma_4$ we choose $z_0 = (-0.88594,-0.88594)$, $z_0 = (0, 1.2656 )$, respectively. For $\sigma_3$ we take $\mbox{m}_{\mbox{\tiny z}} = 7$, $\mbox{m}_{\mbox{\tiny k}} = 8$ and for $\sigma_4$ $\mbox{m}_{\mbox{\tiny z}} = 7$, $\mbox{m}_{\mbox{\tiny k}} = 7$. The reconstructions of $\sigma_3$ could be computed for $\mbox{R}\leq 50$ but not for $\sigma_4$ with $\mbox{R} > 20$.

\vskip 3mm

\underline{\emph{Shortcut method}}

\vskip 3mm

Concerning the shortcut method, firstly the approximate scattering transform, denote it by $\tau^{SC}_j$ for $\sigma_j$ ($j=3,4$), is computed directly from the known conductivity through the Beltrami equation solver using formula \eqref{scat_AP}. The code evaluates the approximation over the points inside the disc of radius $\mbox{R}$ (cut-off frequency) from a $k$-grid of $2^{\mbox{\tiny m}_{\mbox{\tiny k}}}\times 2^{\mbox{\tiny m}_{\mbox{\tiny k}}}$ equispaced points in the square $[-\mbox{R} , \mbox{R}]^2$. Another $z$-grid with $2^{\mbox{\tiny m}_{\mbox{\tiny z}}}\times 2^{\mbox{\tiny m}_{\mbox{\tiny z}}}$ points in $[-\mbox{s}_{\mbox{\tiny z}}, \mbox{s}_{\mbox{\tiny z}})^2$ ($\mbox{s}_{\mbox{\tiny z}}>2$) is involved here.

 To compute $\tau^{SC}_3$ we take $\mbox{R} = 50$, $\mbox{m}_{\mbox{\tiny k}} = 8$, $\mbox{m}_{\mbox{\tiny z}} = 10$, and for $\tau^{SC}_4$ we choose $\mbox{R} = 50$, $\mbox{m}_{\mbox{\tiny k}} = 7$, $\mbox{m}_{\mbox{\tiny z}} = 7$. Let us remark that this choice of $\mbox{R}$ enables reconstructions corresponding to cut-off frequencies less than or equal to $50$.

Secondly, Figures \ref{fig:APscat}, \ref{fig:Nscat} show plots of $\tau^{SC}_j$ and $\textbf{t}^{SC}_j$, where $j=3,4$ and $\textbf{t}^{SC}_j$ is computed from $\tau^{SC}_j$ by formula \eqref{transform_comparison}.

\begin{figure}[!hpb]
\begin{picture}(320,175)
\put(-40,10){\includegraphics[width=7cm]{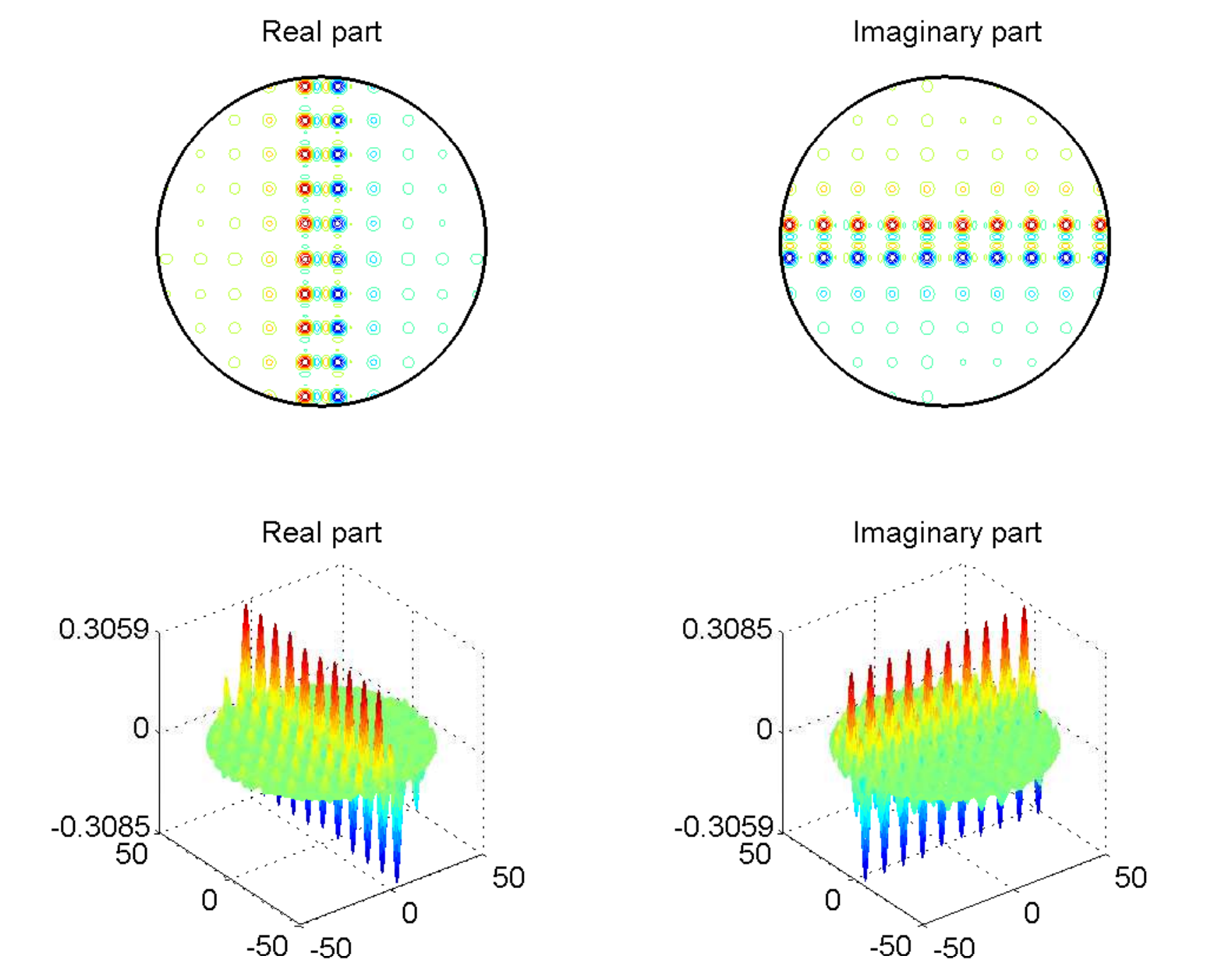}
\includegraphics[width=7cm]{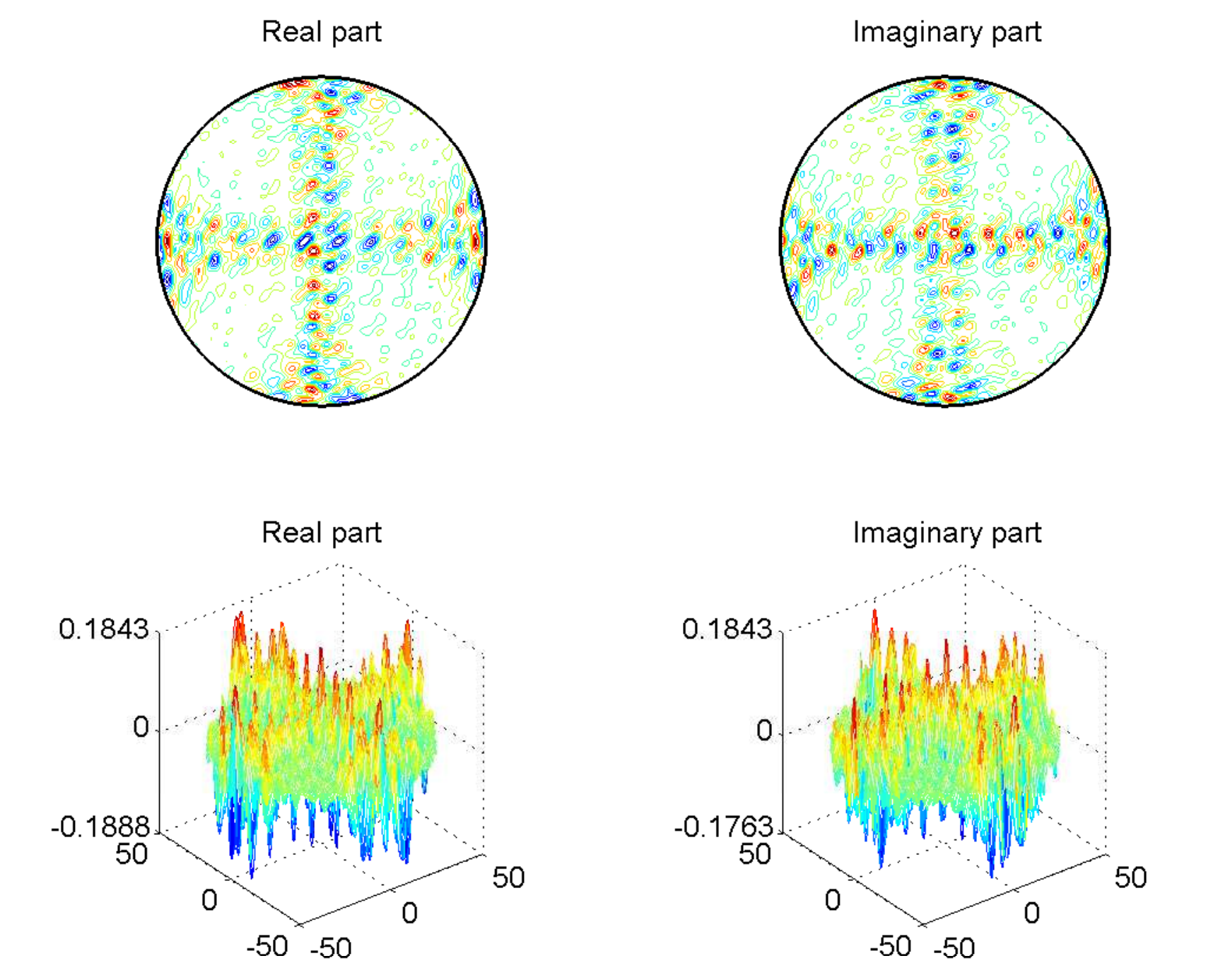}}
\put(-22,0){\footnotesize a) $\tau^{SC}_3$ for $\sigma_3$   ($\mbox{R} = 50,\, \mbox{m}_{\mbox{\tiny k}} = 8,\, \mbox{m}_{\mbox{\tiny z}} = 10$).}
\put(180,0){\footnotesize b) $\tau^{SC}_4$ for $\sigma_4$  ($\mbox{R} = 50,\, \mbox{m}_{\mbox{\tiny k}} = 7,\, \mbox{m}_{\mbox{\tiny z}} = 7$).}
\end{picture}
\caption{\label{fig:APscat}a) Plot of real (left) and imaginary (right) parts of the scattering transform $\tau^{SC}_3$ for the conductivity example $\sigma_3$. b) Idem for  $\tau^{SC}_4$ corresponding to $\sigma_4$.}
\end{figure}

\begin{figure}[!hpb]
\begin{picture}(320,175)
\put(-40,10){\includegraphics[width=7cm]{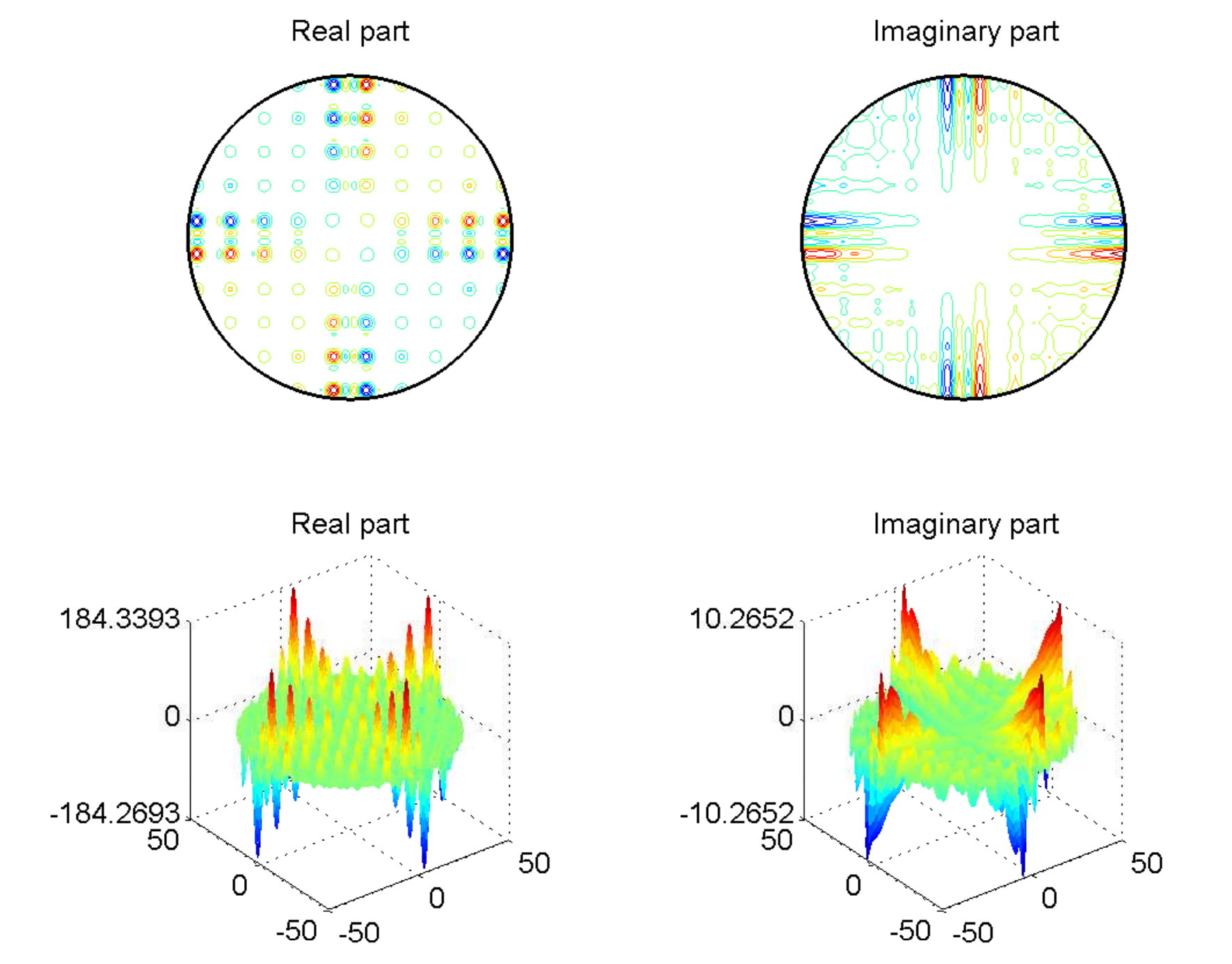}
\includegraphics[width=7cm]{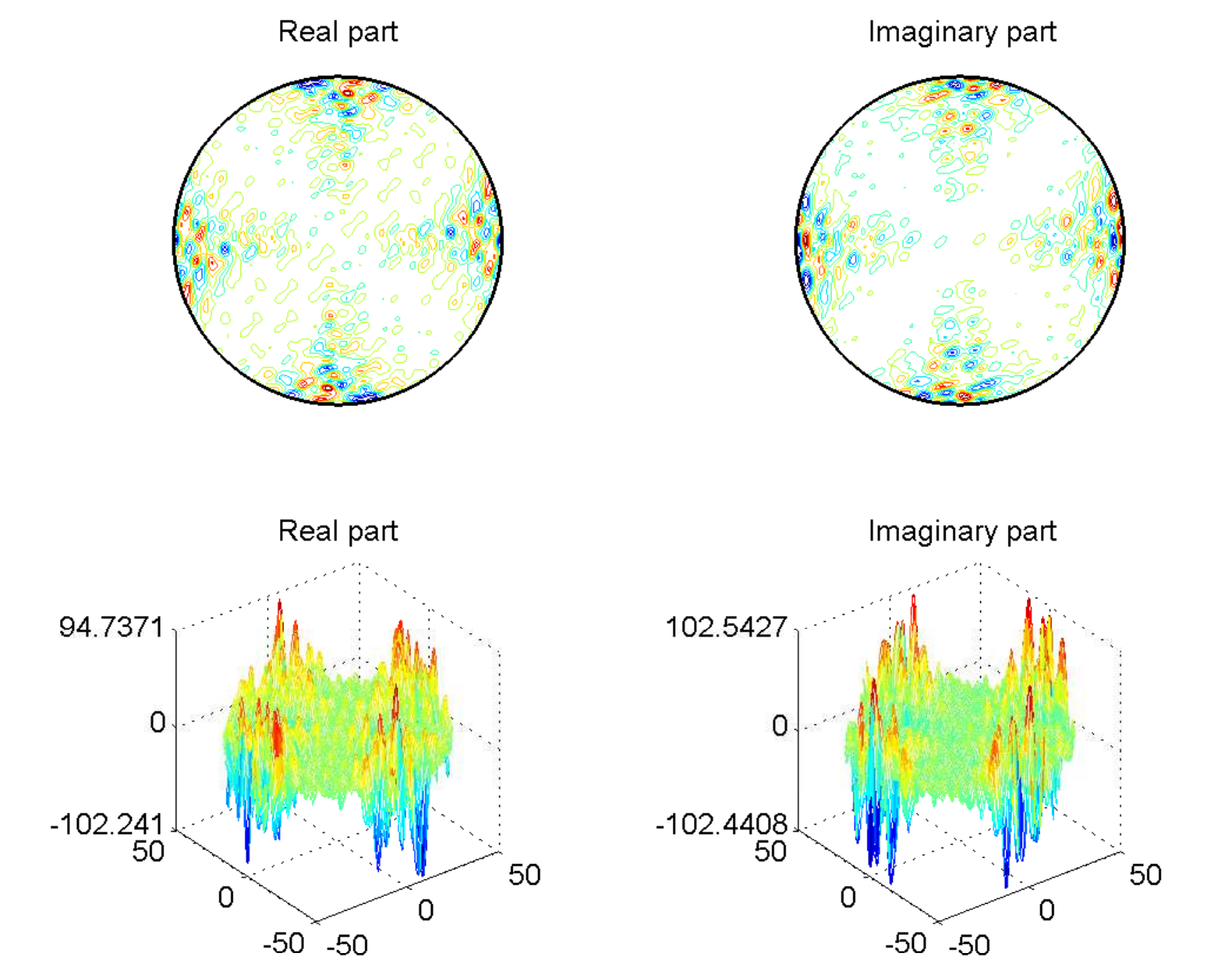}}
\put(-22,0){\footnotesize a) $\textbf{t}^{SC}_3$ for $\sigma_3$ ($\mbox{R} = 50,\, \mbox{m}_{\mbox{\tiny k}} = 8,\, \mbox{m}_{\mbox{\tiny z}} = 10$).}
\put(180,0){\footnotesize b) $\textbf{t}^{SC}_4$ for $\sigma_4$ ($\mbox{R} = 50,\, \mbox{m}_{\mbox{\tiny k}} = 7,\, \mbox{m}_{\mbox{\tiny z}} = 7$).}
\end{picture}
\caption{\label{fig:Nscat}a) Plot of real (left) and imaginary (right) parts of the scattering transform $\textbf{t}^{SC}_3$ for the conductivity example $\sigma_3$. b) Idem for  $\textbf{t}^{SC}_4$ corresponding to $\sigma_4$.}
\end{figure}

Finally, we evaluate the approximate reconstruction by solving the D-bar equation with the truncated scattering transform $\tau^{SC}_j$ as a coefficient. The reconstructions are computed over the points of the same kind of $z$-grid as above belonging to the unit disc with size parameter $\mbox{m}_{\mbox{\tiny z}}$. The code computes the solution to the D-bar equation on every point in the $z$-grid independently and uses a number of $k$-grids with size parameter $\mbox{m}_{\mbox{\tiny k}}$ and cut-off frequency $\mbox{R}$.

These reconstructions were evaluated with $\mbox{m}_{\mbox{\tiny z}} = 7$ and some $\mbox{R}$ values with $\mbox{R}\leq 50$. In addition, with respect to $\sigma_3$ we took $\mbox{m}_{\mbox{\tiny k}} = 9$ and for $\sigma_4$, $\mbox{m}_{\mbox{\tiny k}} = 7$.

\vskip 2mm

Notice that the case $\mbox{R} = 50$ was correctly computed for both conductivities $\sigma_3$, $\sigma_4$ using the shortcut method, but for $\sigma_4$ the highest $\mbox{R}$-value through the low-pass transport matrix method was $\mbox{R} = 20$.

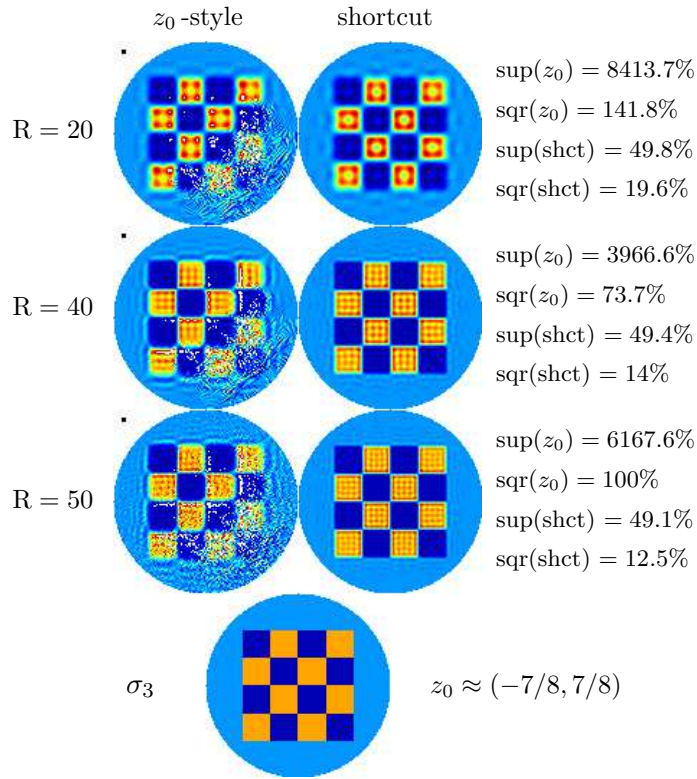
\begin{figure}[!hbp]
\begin{picture}(250,300)
 \put(35,0){\includegraphics[width=5cm]{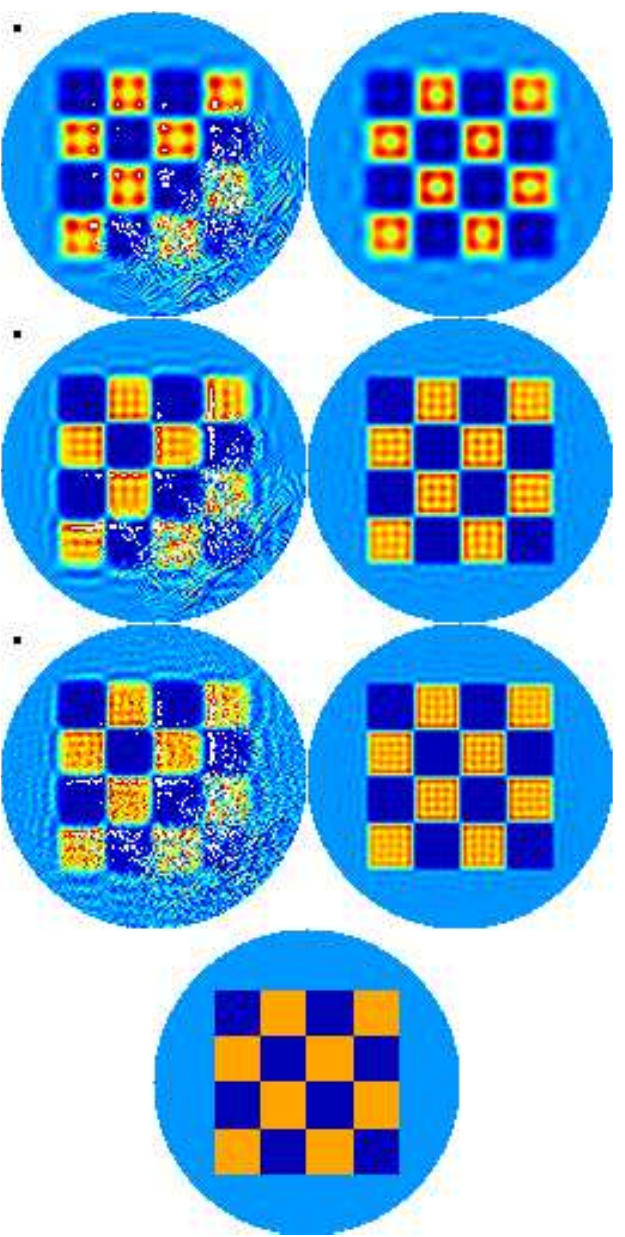}}

 \put(-3,242){\footnotesize $\text{R} = 20$}
 \put(-3,175){\footnotesize $\text{R} = 40$}
 \put(-3,102){\footnotesize $\text{R} = 50$}

 \put(50,285){\footnotesize $z_0\,$-style}
 \put(120,285){\footnotesize shortcut}

 \put(40,32){\small $\sigma_3$}
 \put(155,32){\footnotesize $z_0 \approx (-7/8,7/8)$}

 \put(180,265){\scriptsize sup$(z_0) = 8413.7$\%}
 \put(180,250){\scriptsize sqr$(z_0) = 141.8$\%}
 \put(180,235){\scriptsize sup(shct) $= 49.8$\%}
 \put(180,220){\scriptsize sqr(shct) $= 19.6$\%}

 \put(180,195){\scriptsize sup$(z_0) = 3966.6$\%}
 \put(180,180){\scriptsize sqr$(z_0) = 73.7$\%}
 \put(180,165){\scriptsize sup(shct) $= 49.4$\%}
 \put(180,150){\scriptsize sqr(shct) $= 14$\%}

 \put(180,125){\scriptsize sup$(z_0) = 6167.6$\%}
 \put(180,110){\scriptsize sqr$(z_0) = 100$\%}
 \put(180,95){\scriptsize sup(shct) $= 49.1$\%}
 \put(180,80){\scriptsize sqr(shct) $= 12.5$\%}
 \end{picture}
\caption{\label{fig:nice_PWC_sigma2}Comparison of reconstructions for $\sigma_3$ by the low-pass transport matrix method (left) and the shortcut method (right) using different cut-off frequencies $\mbox{R}$. The chosen point $z_0$ appears in black. The true conductivity is represented at the bottom row.}
\end{figure}

\begin{figure}[!hbp]
\begin{picture}(250,300)
 \put(20,-10){\includegraphics[width=6cm]{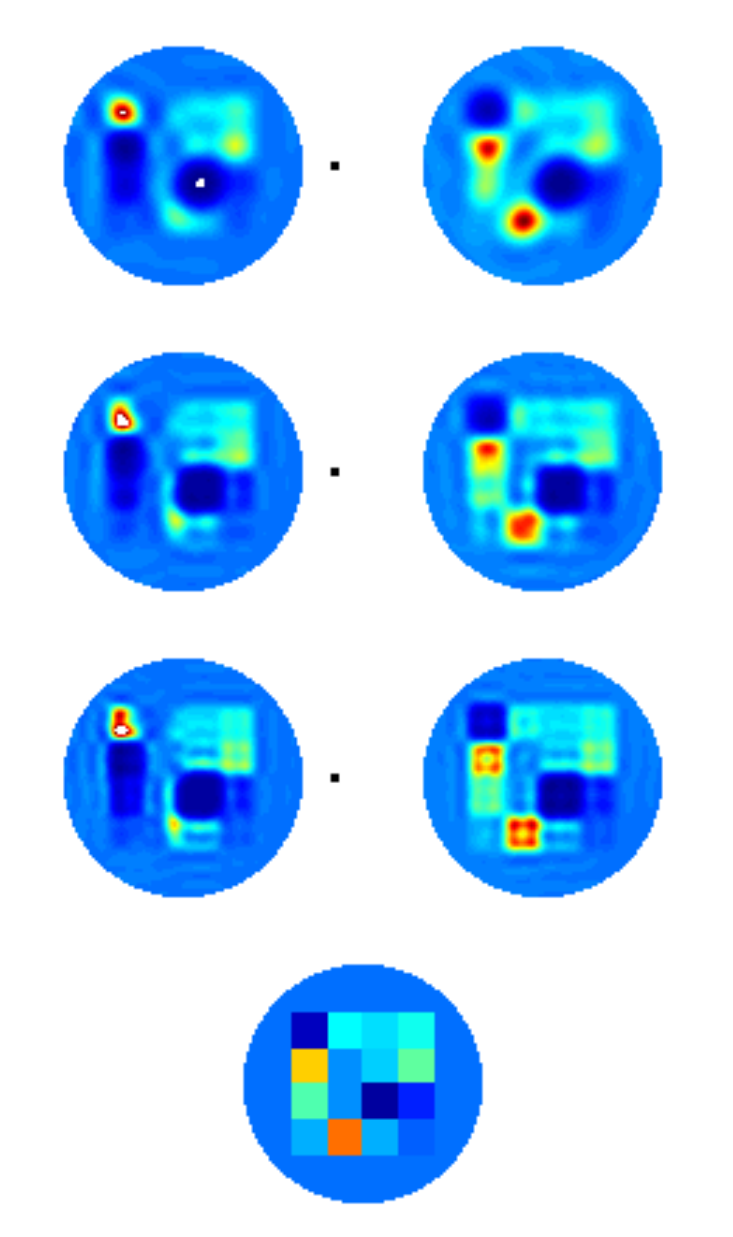}}

 \put(0,238){\footnotesize $\text{R} = 10$}
 \put(0,165){\footnotesize $\text{R} = 15$}
 \put(0,94){\footnotesize $\text{R} = 20$}

 \put(45,272){\footnotesize $z_0\,$-style}
 \put(130,272){\footnotesize shortcut}

 \put(50,22){\small $\sigma_4$}
 \put(147,22){\footnotesize $z_0 \approx (0,1.266)$}

 \put(180,254){\scriptsize sup$(z_0) = 119.1$\%}
 \put(180,243){\scriptsize sqr$(z_0) = 49.2$\%}
 \put(180,232){\scriptsize sup(shct) $= 65.9$\%}
 \put(180,221){\scriptsize sqr(shct) $= 19.6$\%}

 \put(180,183){\scriptsize sup$(z_0) = 144.6$\%}
 \put(180,173){\scriptsize sqr$(z_0) = 50.3$\%}
 \put(180,161){\scriptsize sup(shct) $= 67$\%}
 \put(180,150){\scriptsize sqr(shct) $= 17.5$\%}

 \put(180,112){\scriptsize sup$(z_0) = 154.9$\%}
 \put(180,101){\scriptsize sqr$(z_0) = 51.2$\%}
 \put(180,90){\scriptsize sup(shct) $= 69.8$\%}
 \put(180,79){\scriptsize sqr(shct) $= 15.8$\%}
 \end{picture}
\caption{\label{fig:nice_board_sigma2}Comparison of reconstructions for $\sigma_4$ by the low-pass transport matrix method (left) and the shortcut method (right) using different cut-off frequencies $\mbox{R}$. The chosen point $z_0$ appears in black. The true conductivity is represented at the bottom row.}
\end{figure}
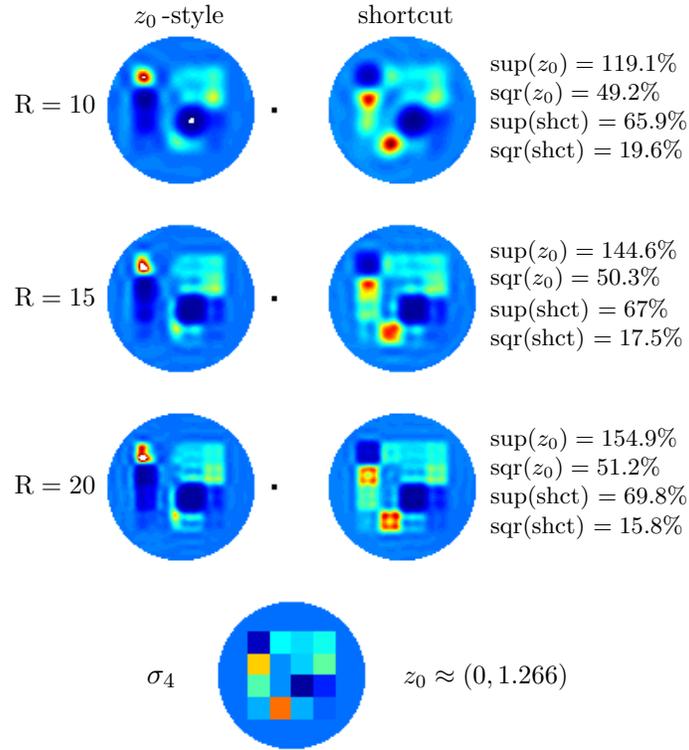

\clearpage

\subsection{Numerical evidence of the transport matrix efficiency}

In addition, further numerical experiments were made aimed at
comparing the actual complex geometric optics solution to the
Beltrami equation
\[
\dbar_z f_{\mu} (z,k_0) = \mu(z)\, \overline{\partial_z f_{\mu}
(z,k_0) },\qquad\text{for }z\in\Omega,\qquad k_0 = 1,
\]
and the transported solution $\widetilde{f}_{\mu}(z,k_0)$ to the
unit disc from $f_{\mu} (z_0,k_0)$ and $f_{-\mu} (z_0,k_0)$ at
certain $z_0\in\C\setminus \overline{\Omega}$.

Figure \ref{fig:cmp_sigma4_20} shows some pictures for $f_{\mu}
(z,k_0)$ and $\widetilde{f}_{\mu} (z,k_0)$ corresponding to the
conductivity example $\sigma_4$. Here $|z|<1$, $k_0 \approx 1$ and
the pivot point outside the unit disc is $z_0\approx (0,1.266)$.
To compute $\widetilde{f}_{\mu}$ the transport matrix was
generated with a cut-off frequency $\mbox{R} = 20$ and we took the
size parameters as follows: $\mbox{m}_{\mbox{\tiny k}} = 7$,
$\mbox{m}_{\mbox{\tiny z}} = 7$, $\mbox{s}_{\mbox{\tiny z}} =
1.5$.

The relative errors with sup and $l^2$ norms
\[
\mbox{sup} := {\norm{f_{\mu} (\cdot,k_0)-\widetilde{f}_{\mu}
(\cdot,k_0)}_{l^{\infty}}\over \norm{f_{\mu}
(\cdot,k_0)}_{l^{\infty}}}\cdot 100\, \%, \,\,\, \mbox{sqr} :=
{\norm{f_{\mu} (\cdot,k_0) - \widetilde{f}_{\mu} (\cdot,k_0)
}_{l^{2}}\over \norm{f_{\mu} (\cdot,k_0)}_{l^{2}}}\cdot 100 \, \%
\]
are
\begin{equation}\label{form:errors}
\mbox{sup} = 20.23\%,\qquad \mbox{sqr}= 10.28\%,
\end{equation}
respectively.

\clearpage

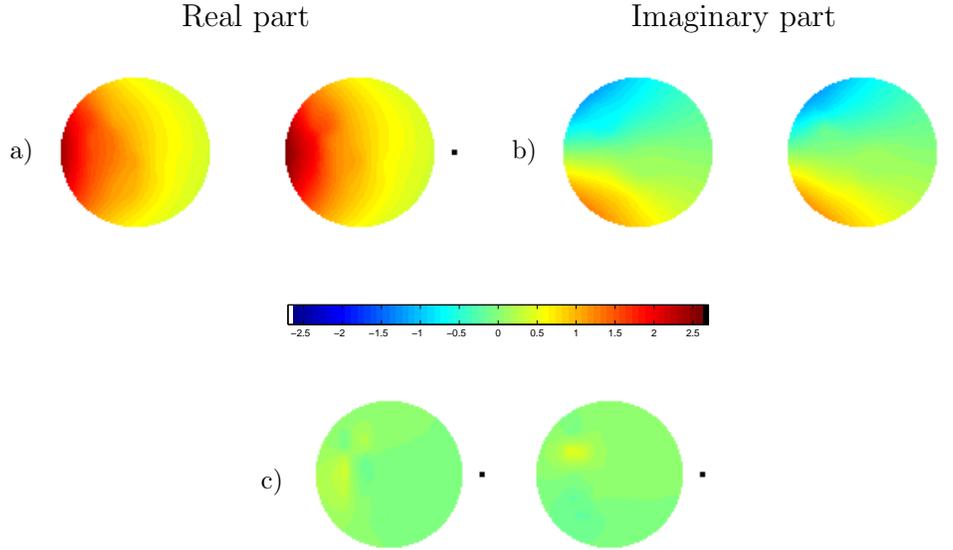
\begin{figure}[!hbp]
\begin{picture}(320,270)
\put(-20,150){\includegraphics[width=6cm]{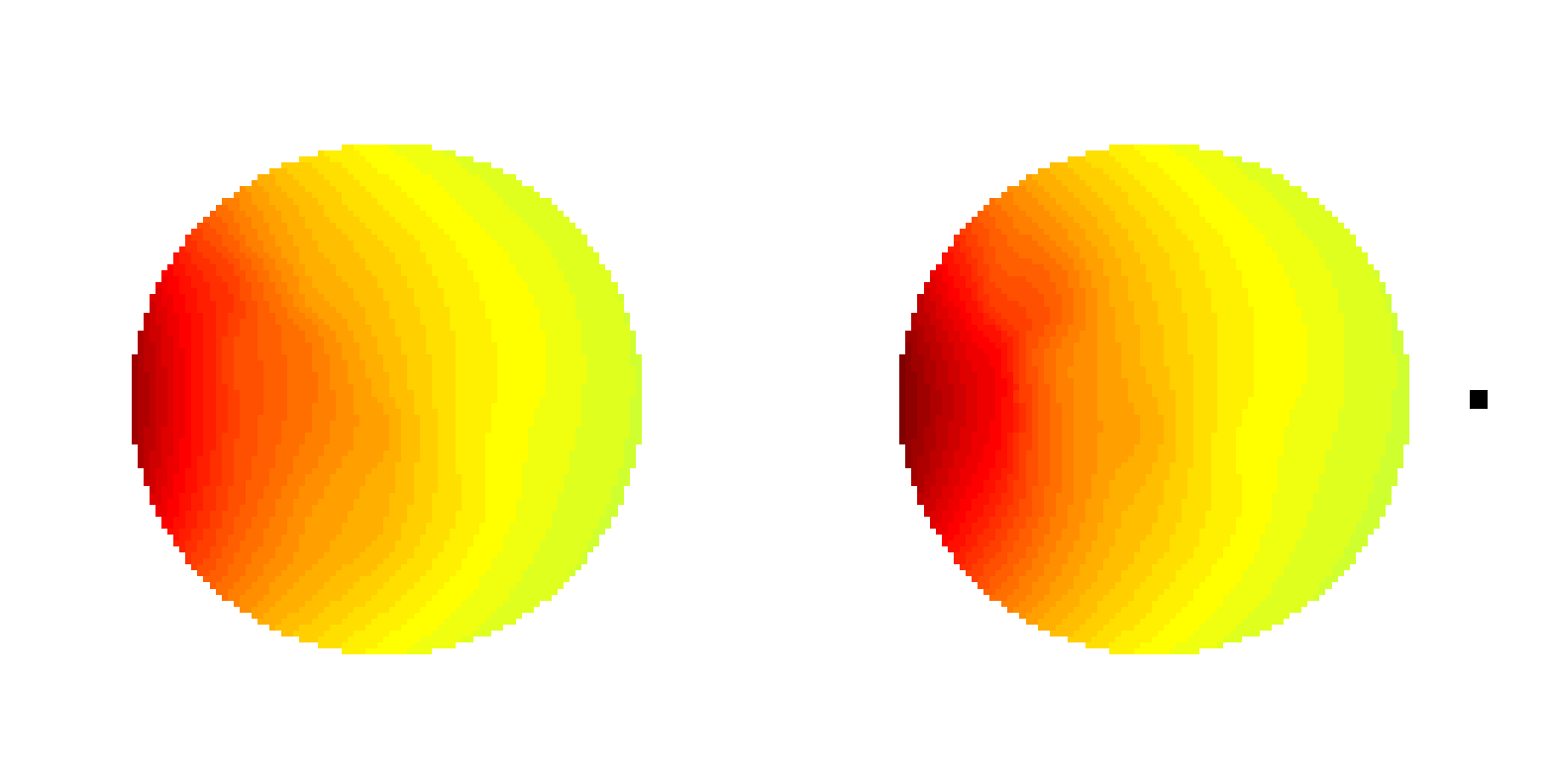}}
\put(170,150){\includegraphics[width=6cm]{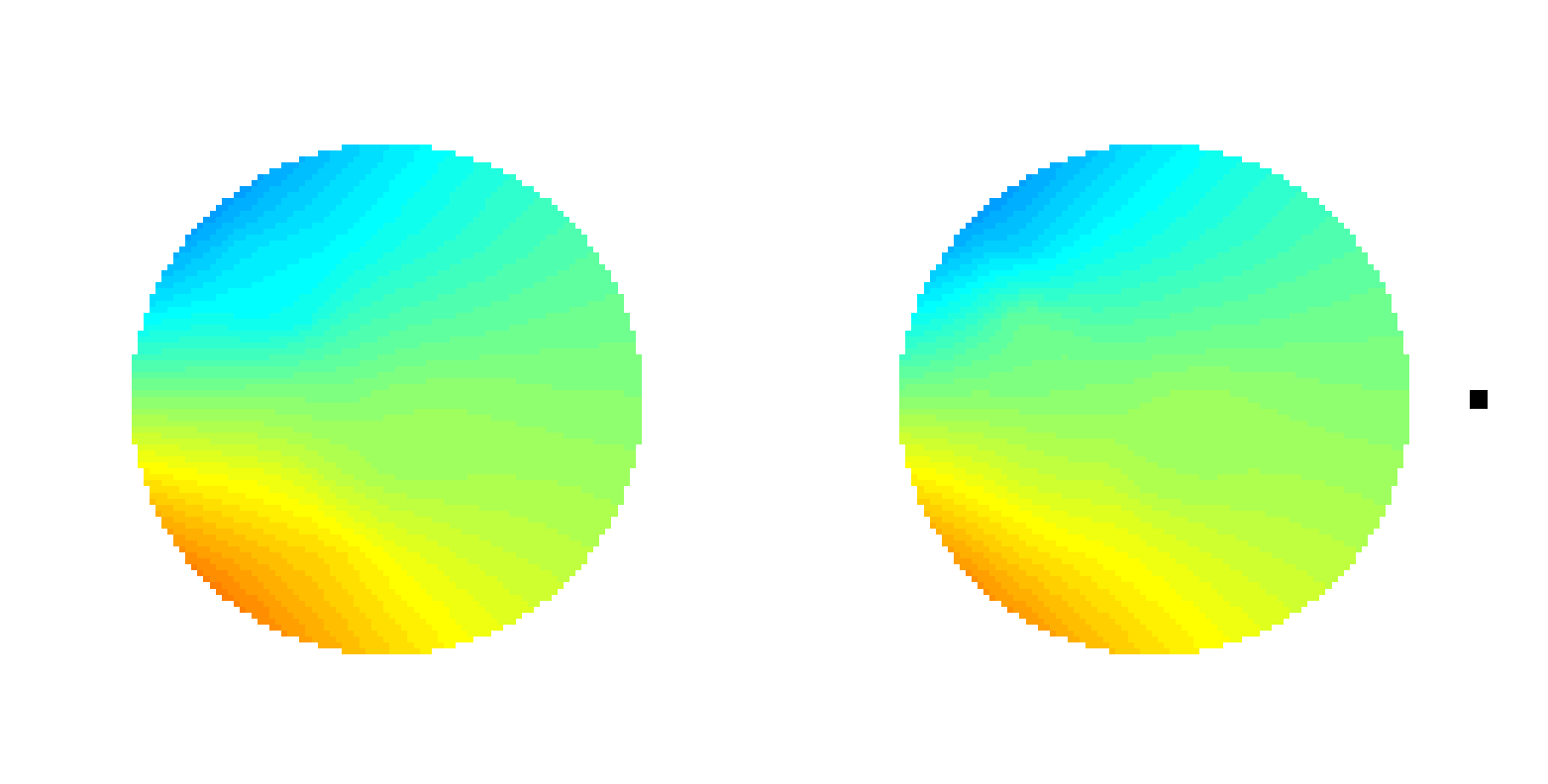}}
\put(75,2){\includegraphics[width=6cm]{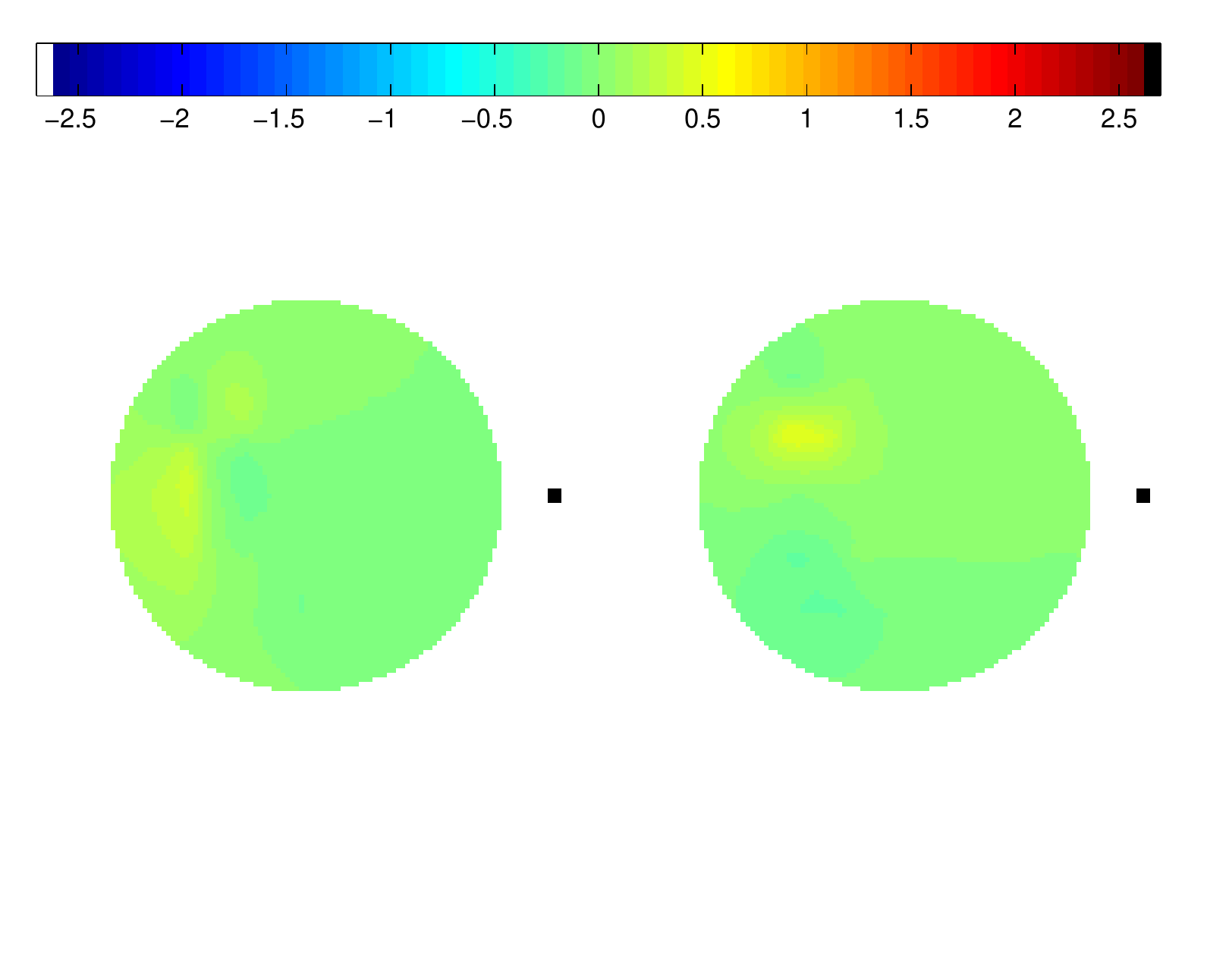}}

\put(-25,190){\footnotesize a)} \put(165,190){\footnotesize b)}
\put(70,65){\footnotesize c)}

\put(40,240){\normalsize Real part} \put(210,240){\normalsize
Imaginary part}

\end{picture}
\caption{\label{fig:cmp_sigma4_20}Comparison of $f_{\mu} (z,k_0)$
and $\widetilde{f}_{\mu} (z,k_0)$ for the conductivity $\sigma_4$
with $\mbox{R} = 20$. In the first row we show real and imaginary
parts of both functions. In a) the real part of the actual
solution $f_{\mu} (z,k_0)$ is depicted on the left and the real
part of the transported $\widetilde{f}_{\mu} (z,k_0)$ on the right
(where the point $z_0$ is represented in black). In b) the same
information is showed for the imaginary part of $f_{\mu} (z,k_0)$
and $\widetilde{f}_{\mu} (z,k_0)$. The last row c) shows the
difference $\widetilde{f}_{\mu} (z,k_0)-f_{\mu} (z,k_0)$. The real
part is exhibited on the left and the imaginary part on the right.
Again, the pivot point $z_0$ appears in black. All the pictures
were generated under the same colormap using Matlab.}
\end{figure}


\section{Conclusion}


Regarding the radially symmetric examples studied in Section
\ref{section:symm_examples}, the conclusion is clear: numerical
evidence suggests that discontinuous conductivities can be
reconstructed with the shortcut method more and more accurately
when $\mbox{R}$ tends to infinity.

The errors in the reconstructions are very similar to the Gibbs
phenomenon observed with truncated linear Fourier transforms.

The numerical study of the nonsymmetric examples comparing both reconstruction types leads us to two observations:

- For these discontinuous conductivity examples, the shortcut method generates considerably better reconstructions than the transport matrix method.

- The approximate solution computed by the transport matrix method is reliable on an area within the unit disc close enough to the selected $z_0$ point outside the unit disc.

Finally, since errors \eqref{form:errors} are small and plots for
both functions $f_{\mu} (z,k_0)$, $\widetilde{f}_{\mu} (z,k_0)$
are similar in view of Figure \ref{fig:cmp_sigma4_20}, we have
numerical evidence supporting the fact that the transport matrix
method ``transports well'' and the worse results by the transport
matrix method are explained by the final algebraic steps,
including numerical differentiation in \eqref{finalstep1}, just
after the computation of the solutions $u^{(R)}_1(z,k_0)$,
$u^{(R)}_2(z,k_0)$ in \eqref{transtrunc} through the transport
matrix itself.


\appendix

\section*{Appendix}

\noindent This Appendix is aimed at presenting two final
discussions as follows. On a hand, the noisy scattering transform
referred to in Section \ref{sec:shortcutm} is compared with the
free-noise scattering transform computed by solving the Beltrami
equation and formula \eqref{scat_AP} for the checkerboard-style
conductivity phantom $\sigma_3$ from Figure \ref{fig:PWC}. On the
other hand, a quantitative discussion on the precision required to
the DN map to obtain the accuracy of the scattering transform on
the disc of radius $50$ is introduced.

\vskip 5mm

\emph{Example of noisy scattering transforms}

For simplicity of notation, let us write $\tau_p$ for the
scattering data obtained on the disc centered at the origin of
radius $6$ by solving the boundary integral equation mentioned in
Section \ref{section:APmethod} from a simulated
Dirichlet-to-Neumann map, $\Lambda^p$, with $p\,\%$ noise.
Additionally, denote by $\tau$ the aforementioned free-noise
scattering transform on the same disc (denoted by $\tau^{SC}_3$ in
Section \ref{sec:comparison}).

The approximate Dirichlet-to-Neumann map $\Lambda^p$ is computed
adding $p\%$ Gaussian noise to the FEM boundary voltages (for
$\sigma_3$) as it is explained in Section 5.3 of
\cite{Hamilton2014}.

The $k$-grid consists of $2^{\mbox{\tiny m}_{\mbox{\tiny
k}}}\times 2^{\mbox{\tiny m}_{\mbox{\tiny k}}}$ equispaced points
in the square $[-6 , 6]^2$ with $\mbox{m}_{\mbox{k}} = 7$. The
scattering data are computed on the $12,851$ points of such grid
belonging to the disc of radius $6$. For the grid in the
$z$-variable used to compute $\tau$, $2^{\mbox{\tiny
m}_{\mbox{\tiny z}}}\times 2^{\mbox{\tiny m}_{\mbox{\tiny z}}}$
equidistributed points are taken in the square
$[-\mbox{s}_{\mbox{\tiny z}}, \mbox{s}_{\mbox{\tiny z}})^2$, with
step $h_z = \mbox{s}_{\mbox{\tiny z}} / 2^{\mbox{\tiny
m}_{\mbox{\tiny z}}-1}$ and $\mbox{s}_{\mbox{\tiny z}}=2.3$.

For different choices of $p$ the following absolute error is
computed for $r$ varying over a partition of the interval $(0,6]$:
$$
E_p(r):=\norm{\chi_{D(0,r)} (\tau-\tau_p)}_{L^{\infty}}\, ,
$$
where $\chi_{D(0,r)}$ denotes the characteristic function of the
disc centered at the origin and radius $r$.

In Figure \ref{fig:Ep(r)} the profiles of $E_p(r)$  for $p =
0.1\%$, $0.5\%$, $1\%$, $5\%$ are shown in black, blue, green and
red colour, respectively. The maximum value shown on the vertical
axis is $1$.

\begin{figure}[!htp]
\begin{picture}(300,210)
\put(30,-10){\includegraphics[keepaspectratio=true, height =
7cm]{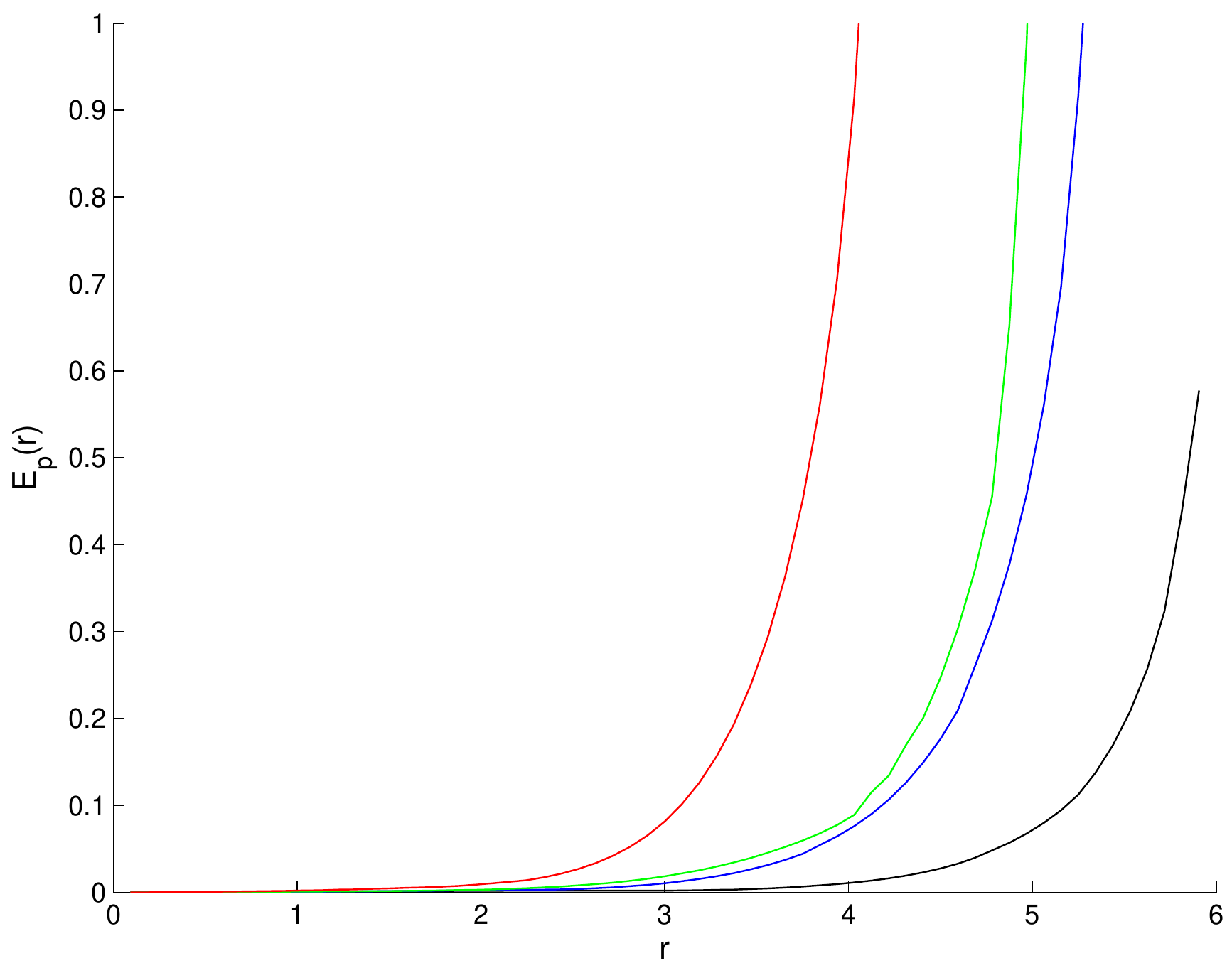}}
\end{picture}
\caption{\label{fig:Ep(r)}Profiles of the values of $E_p(r)$ in
$(0,1)$ with $r\in (0,6)$, for $p = 0.1\,\%$ (black), $p=0.5\,\%$
(blue), $p=1\,\%$ (green), $p=5\,\%$ (red).}
\end{figure}

\emph{Accuracy of the DN map to compute $\T(k)$ on $D(0,50)$}

How many correct significant digits would one need in the DN map
for computing the scattering transform  $\T(k)$ for $|k|<50$? It
is well-known \cite{Alessandrini1988} that there is a logarithmic
relation between measurement accuracy and details in conductivity.
In other words, EIT is an exponentially ill-posed inverse problem.
In terms of scattering transforms, a fixed noise level $p$\% in
the EIT measurement leads to, roughly speaking, the computation of
$\T(k)$ to be stable and accurate in a disc $|k|<R(p)$ and
unstable outside that disc. See \cite[Figure 2]{Knudsen2009}. The
exponential ill-posedness shows up as a logarithmic dependence of
$R$ on $p$, see \cite[Figure 3]{Knudsen2009}. Improving the
measurement accuracy so that the DN matrix entries have one more
significant digit increases the stability radius from $R$ to
$R+\Gamma$ with $\Gamma>0$ a fixed real number.

Let us make a crude quantitative estimation based on \cite[Figure
3]{Knudsen2009}. Two significant digits in the DN map corresponds
to $R=3$ and six significant digits to $R=7$. A quick computation
shows that achieving the stability radius $R=50$ would require 49
correct significant digits in the elements of the DN matrix (and,
actually, using a larger matrix with more oscillatory basis
functions involved as well).

\begin{acknowledgements} This work was supported by the Finnish Centre of Excellence in Inverse Problems Research $2012$-$2017$
(Academy of Finland CoE-project $250215$). In addition, KA was supported by Academy of Finland, projects $75166001$ and $1134757$,
and LP was supported by ERC-$2010$ Adv. Grant, $267700$ - InvProb. The third author was also partially supported by Academy of Finland
 (Decision number $141075$) and Ministerio de Ciencia y Tecnolog\'ia de Espa\~na (project MTM
$2011$-$02568$).
\end{acknowledgements}


\bibliographystyle{siam}
\bibliography{Inverse_problems_references}

\end{document}